\documentclass[10pt]{amsart}
\usepackage{amsmath}
\usepackage{amsthm}
\usepackage{amsfonts}
\usepackage{amssymb}
\usepackage{amscd}
\usepackage[all]{xy}
\usepackage[percent]{overpic}
\usepackage{graphicx}
\usepackage{xcolor}
\usepackage{enumitem}
\usepackage{hyperref}
\usepackage{tikz}
\usetikzlibrary{decorations.pathmorphing, decorations.markings, calc, cd}
\usepackage{pgf}

\hypersetup{
    colorlinks=true,       
    linkcolor=blue,          
    citecolor=blue,        
    filecolor=blue,      
    urlcolor=blue           
}

    \newtheorem{Lem}{Lemma}[section]
    \newtheorem{Lem-Def}[Lem]{Lemma-Definition}
		\newtheorem{CorDef}[Lem]{Corollary-Definition}
    \newtheorem{Prop}[Lem]{Proposition}

    \newtheorem{Thm}[Lem]{Theorem}  
		\newtheorem*{Thm*}{Theorem}
    \newtheorem{Cor}[Lem]{Corollary}

\theoremstyle{definition}
\font\smallsc=cmcsc10
\font\smallsl=cmsl10

    \newtheorem{Def}[Lem]{Definition}
    \newtheorem{Exa}[Lem]{Example}
    \newtheorem{Rem}[Lem]{Remark}

\DeclareMathOperator{\Pic}{Pic}

\newcommand{\ora}[1]{\overrightarrow{#1}}
\newcommand{\rank}{\text{rank}}

\newcommand{\E}{\mathcal E}

\newcommand{\Spec}{\text{Spec}\,}

\newcommand{\F}{\mathcal F}

\newcommand{\I}{\mathcal I}
\newcommand{\M}{\mathcal M}

\renewcommand{\L}{\mathcal L}
\renewcommand{\O}{\mathcal O}
\newcommand{\C}{\mathcal C}

\newcommand{\X}{\mathcal X}
\newcommand{\Y}{\mathcal Y}
\newcommand{\D}{\mathcal D}
\newcommand{\R}{\mathbb R}
\renewcommand{\P}{\mathcal{P}}
\newcommand{\col}{\colon}

\newcommand{\ra}{\rightarrow}
\newcommand{\ol}{\overline}
\newcommand{\supp}{\text{supp}}

\newcommand{\wt}{\widetilde}

\newcommand{\wh}{\widehat}

\newcommand{\J}{\mathcal{J}}

\newcommand{\lra}{\longrightarrow}

\DeclareMathOperator{\Rel}{Rel}

\DeclareMathOperator{\rel}{Rel}

\DeclareMathOperator{\Hom}{Hom}

\renewcommand{\l}{\ell}

\newcommand{\Pos}{\mathcal{QD}_{v_0,\mu}}
\newcommand{\Posg}{\mathcal{QD}_{\mu, g}}
\newcommand{\mon}{\textnormal{Mon}}
\newcommand{\eff}{\textnormal{Eff}}

\newcommand{\Aut}{\textnormal{Aut}}
\newcommand{\cone}{\textnormal{cone}}
\newcommand{\val}{\text{val}}

\newcommand{\out}{\text{Out}}
\newcommand{\trop}{\text{trop}}
\renewcommand{\div}{\textnormal{div}}
\newcommand{\Div}{\text{Div}}
\newcommand{\Prin}{\text{Prin}}
\renewcommand{\Im}{\text{Im}}
\newcommand{\Graph}{\mathbf{Graph}}
\newcommand{\grap}{\mathcal{SG}}
\newcommand{\ord}{\textnormal{ord}}

\title[The universal tropical Jacobian]{The universal tropical Jacobian and the skeleton of the Esteves'  universal Jacobian}

\author{Alex Abreu and Marco Pacini}

\begin{document}

\begin{abstract}
For each universal genus-$g$ polarization $\mu$ of degree $d$, we construct a universal tropical Jacobian $J_{\mu,g}^{trop}$ as a generalized cone complex over the moduli space of stable pointed genus-$g$ tropical curves. We show several properties of the space $J_{\mu,g}^{trop}$. In particular, we prove that the natural compactification of $J_{\mu,g}^{trop}$ is the tropicalization of the Esteves' compactified universal Jacobian over the moduli space of stable pointed genus-$g$ curves.
\end{abstract}

\maketitle

\tableofcontents

\section{Introduction}
   
\subsection{History and motivation}

A recent approach to study moduli spaces in algebraic geometry is the construction of a  suitable tropical analogue from which one can extract useful combinatorial and topological properties. The construction of the moduli space of stable tropical curves by Mikhalkin \cite{M}, and Brannetti, Melo and Viviani \cite{BMV}, and of its compactification by Caporaso \cite{Caporaso}, revealed strong similarities with the moduli space of Deligne-Mumford stable curves: these spaces share the same dimension and have dual stratifications. These analogies were made rigorous by Abramovich, Caporaso and Payne \cite{ACP}. They proved that the skeleton of the Berkovich analytification of the moduli stack of stable curves is isomorphic to the compactified moduli space of stable tropical curves. In the last few years, these type of constructions have been carried out also for other important moduli spaces: Cavalieri, Hampe, Markwig and Ranganathan \cite{CHMR}, and Ulirsch \cite{U} for weighted stable curves; Ranganathan \cite{R1} and \cite{R2} for rational stable maps;  Cavalieri, Markwig and Ranganathan \cite{CMR} for admissible covers; M\"oeller, Ulirsch and Werner \cite{MUW} for effective divisors; Caporaso, Melo and Pacini \cite{CMP} for spin curves. \par

 Another central moduli space in algebraic geometry is the Jacobian of an algebraic curve, and its compactifications. The construction of a compactified Jacobian for a curve dates back to Igusa \cite{I56}. Since then, several compactified Jacobians have been constructed. Mayer and Mumford \cite{MM} introduced for the first time the use of torsion-free rank-$1$ sheaves to describe the boundary points, and the technique of Geometric Invariant Theory (GIT). This technique was later employed by D'Souza \cite{Ds} to compactify Jacobians of integral curves, and by Oda and Seshadri \cite{OS}, who first considered the case of reducible curves. Caporaso \cite{C} and, soon after, Pandharipande \cite{Pand}, used GIT to construct a universal compactified Jacobian over the moduli space of stable curves.
Finally, Altman and Kleiman \cite{AK}, and Esteves \cite{Es01}, constructed a compactified relative Jacobian for families of reduced curves, without using GIT. This compactification employs the notion of quasistability for torsion-free rank-$1$ sheaves  with respect to a polarization, and depends on the choice of a section of the family of curves. 

Recently, Caporaso \cite{C08}, Melo \cite{M11}, \cite{M15}, and Kass and Pagani \cite{KP}, carried out a stacky version of theses Jacobians. In particular Melo \cite{M15} constructed a Deligne-Mumford stack $\ol{\J}_{\mu,g}$ over the moduli space $\ol{\M}_{g,1}$ of stable pointed genus-$g$ curves, following the work of Esteves \cite{Es01}. The Jacobian $\ol{\J}_{\mu,g}$, which we call the Esteves' universal Jacobian, compactifies the relative degree-$d$ Jacobian $\J_{d,g,1}$ over $\M_{g,1}$ and parametrizes torsion-free rank-$1$ sheaves which are quasistable with respect to a polarization $\mu$.\par
 
The goal of this paper is the construction of a universal tropical Jacobian and the study of its interplay with algebraic geometry. This is the first of a series of papers dedicated to the subject. In this paper we construct the universal tropical Jacobian $J_{\mu,g}^{trop}$ as a generalized cone complex and prove that there is an isomorphism between its natural compactification and the skeleton of the Esteves' universal compactified Jacobian $\ol{\J}_{\mu,g}$. This isomorphism is compatible with the maps of tropicalization and retraction to the skeleton. In the forthcoming paper \cite{APfuture}, we analyze the universal tropical Abel map from $M_{g,n}^{\trop}$ to $J_{\mu,g}^\trop$ and use it to resolve the universal Abel map for nodal curves having $\ol{\J}_{\mu,g}$ as a target. \par

It is worth noting that Baker and Rabinoff \cite{BR} already studied the problem of the tropicalization of the Jacobian of a curve. They prove that the skeleton of the analytification of the Jacobian of a curve over an algebraically closed, complete, non Archimedean valuation field is isomorphic to the usual tropical Jacobian defined in \cite{MZ} and \cite{BF}. Our analysis can be seen as an extension of  this result to the universal setting.\par

A comment about the difference between the pointed and non-pointed case is due. 
 The Jacobian $\ol{\J}_{\mu,g}$ is a toroidal Deligne-Mumford stack proper over $\ol{\M}_{g,1}$. This allows us to apply the setting and results of \cite[Section 6]{ACP}. The known degree-$d$ universal Jacobian stacks over $\ol{\M}_g$ are not separated when $\gcd(d-g+1,2g-2)\neq 1$, so in these cases a different approach is needed to tropicalize these compactified Jacobians. On the other hand, if $\gcd(d-g+1,2g-2)=1$, the quasistability condition does not depend on the chosen marked point, and our approach for constructing $\ol{J}^\trop_{\mu,g}$ can be employed, essentially in the same way, to build a universal tropical Jacobian over $M^\trop_g$.

   \subsection{The results}
   Let  $(X,p_0)$ be a pointed tropical curve and $\mu$ be a degree-$d$ polarization on $X$, i.e., a function $\mu\col X\to \R$ such that $\sum_{p\in X}\mu(p)=d$, where the sum is finite. We introduce the notion of $(p_0,\mu)$-quasistability for degree-$d$ divisors on $X$ imposing that certain inequalities hold for every tropical subcurve of $X$. We are able to prove the following theorem, which is the tropical analogue of a similar statement on $(p_0,\mu)$-quasistable torsion-free rank-$1$ sheaves appearing in \cite{Es01}.

\begin{Thm*}[\ref{thm:quasistable}]
Every degree-$d$ divisor on a tropical curve is equivalent to a unique $(p_0,\mu)$-quasistable degree-$d$ divisor.
\end{Thm*}

This is also a generalization of the same result for break divisors in degree $g$ (see \cite{MZ} and \cite[Theorem 1.1]{ABKS}). In fact, a break divisor is $(p_0,\mu)$-quasistable for the so-called degree-$g$ canonical polarization $\mu$ (see \cite[Introduction]{CC}, \cite{Shen}, and  Example \ref{exa:pol}). As for break divisors, quasistable divisors are well-behaved when varying the tropical curves under specializations. \par

 Then we define a polyhedral complex $J_{p_0,\mu}^\trop(X)$ parametrizing $(p_0,\mu)$-quasistable divisors on the tropical curve $X$. This polyhedral complex is the tropical analogue of the compactified Jacobian constructed by Oda-Seshadri \cite{OS} and Esteves \cite{Es01}. In Theorem \ref{thm:jX} we prove that $J_{p_0,\mu}^\trop(X)$ is homeomorphic to the usual tropical Jacobian $J^\trop(X)$ of $X$. When $\mu$ is the canonical degree-$g$ polarization, the polyhedral complex $J_{p_0,\mu}^\trop(X)$ is the same as the one constructed by An, Baker, Kuperberg and Shokrieh \cite{ABKS}.\par

These polyhedral complexes can be cast together to form a generalized cone complex $J^\trop_{\mu,g}$ parametrizing tuples $(X,p_0,\D)$ where $(X,p_0)$ is a stable pointed genus-$g$ tropical curve and $\D$ is a $(p_0,\mu)$-quasistable divisor on $X$. Here $\mu$ is a universal genus-$g$ polarization of degree-$d$, meaning that $\mu$ is compatible with specializations of genus-$g$ graphs. We prove that the generalized cone complex $J^\trop_{\mu,g}$ is a universal tropical Jacobian over $M^\trop_{g,1}$, as stated by the following theorem. 

\begin{Thm*}[\ref{thm:maintrop}]
The generalized cone complex $J^\trop_{\mu,g}$ has pure dimension $4g-2$ and is connected in codimension $1$. The natural forgetful map $\pi^{trop}\col J^\trop_{\mu,g}\to M^\trop_{g,1}$ is a map of generalized cone complexes and for every equivalence class of a  stable pointed genus-$g$ tropical curve $X$, we have a homeomorphism
\[
(\pi^{trop})^{-1}([X])\cong J^\trop(X)/\Aut(X).
\]
\end{Thm*}
 
 The above setup can be also given for graphs with $1$ leg. In fact, our starting point is the notion of $(v_0,\mu)$-quasistability for divisors on graphs with legs. In this case, the main objects of study have the natural structure of poset, instead of polyhedral/cone complex. The two frameworks are closely related, since every pointed tropical curve $(X,p_0)$ has a (not unique) model $(\Gamma_X,v_0)$ which is a graph with $1$ leg. There is a ranked poset $\Pos(\Gamma_X)$ of $(v_0,\mu)$-quasistable (pseudo-)divisors on $\Gamma_X$, and a natural continuous map $J^{\trop}_{p_0,\mu}(X)\to \Pos(\Gamma_X)$ (see Corollary \ref{cor:quasiquasi} and Remark \ref{rem:cont}). We also define a universal ranked poset $\Posg$ and a continuous map $J^\trop_{\mu,g}\to \Posg$, where $\Posg$ parametrizes tuples $(\Gamma,v_0,D)$, with $(\Gamma,v_0)$ a stable genus-$g$ graph with $1$ leg and $D$ a $(v_0,\mu)$-quasistable (pseudo-)divisor on $\Gamma$. The properties of $\Posg$ are essentially the same as the ones in the above theorem (see Theorem \ref{thm:maingraph}).  \par 
Next, we study the skeleton of the Esteves' universal compactified Jacobian $\ol{\J}_{\mu, g}$.  The first issue we need to address is the description of the natural stratification induced by the embedding $\J_{d,g,1}\subset \ol{\J}_{\mu, g}$. Recently, the analogous problem for other compactified Jacobians has been studied by Caporaso and Christ \cite{CC}. They give a combinatorially interesting incarnation for a stratification of the Caporaso's compactified Jacobians in degree $g-1$ and $g$.

 A natural stratification for the Esteves' compactified Jacobian $\J_{p_0,\mu}(X)$ of a fixed nodal curve $X$ is described by Melo and Viviani \cite{MV}. In Proposition \ref{prop:stratification} we show that their stratification can be refined to a graded stratification of $\ol{\J}_{\mu,g}$ by the poset $\Posg$. We give a useful description of the strata as quotient stacks in Proposition \ref{prop:stratades}. These are the key results for our analysis of the skeleton $\ol{\Sigma}(\ol{\J}_{\mu,g})$ of $\ol{\J}_{\mu,g}$, which allows us to prove the following theorem.
\begin{Thm*}[\ref{thm:tropj}]
There is an isomorphism of extended generalized cone complexes
\[
\Phi_{\ol{\J}_{\mu,g}}\col \ol{\Sigma}(\overline{\J}_{\mu,g})\to \ol{J}^{\trop}_{\mu,g}.
\]
Moreover, the following diagram is commutative
\begin{eqnarray*}
\SelectTips{cm}{11}
\begin{xy} <16pt,0pt>:
\xymatrix{
\ol{\J}_{\mu,g}^{an} \ar@/^2pc/[rr]^{\trop_{\ol{\J}_{\mu,g}}} \ar[d] \ar[r]^{{\bf p}_{\ol{\J}_{\mu,g}}\;}   
  & \ar[r]^{{\Phi}_{\ol{\J}_{\mu,g}}} \ar[d] \ol{\Sigma}(\ol{\J}_{g,n}) & \ol{J}_{\mu,g}^{trop}   \ar[d] \\
\ol{\M}_{g,1}^{an}  \ar@/_2pc/[rr]^{\trop_{\ol{\M}_{g,1}}} \ar[r]^{{\bf p}_{\ol{ {\M}}_{g,1}}\;}            &   \ar[r]^{\cong}         \ol{\Sigma}(\ol{{\M}}_{g,1})  & \ol{M}_{g,1}^{trop}
 }
\end{xy}
\end{eqnarray*}
where the vertical maps are forgetful maps, ${\bf p}_{-}$ denotes the retractions to the skeleton and $\trop_{-}$ denotes the tropicalization maps.
\end{Thm*}

In short, in Sections \ref{sec:preliminaries} and  \ref{sec:trop} we introduce the background material and the technical tools about graphs, posets, and tropical curves. In Section \ref{sec:quasigraph} we study quasistability for pseudo-divisors on graphs. In Section \ref{sec:univtropJ} we introduce the Jacobian of quasistable divisors on tropical curves and study its properties. Finally, in Section \ref{sec:skeleton}, we prove the result on the skeleton of Esteves' universal compactified Jacobian.

\section{Graphs and posets}
\label{sec:preliminaries}

\subsection{Graphs}
\label{subsec:graphs}
  Given a graph $\Gamma$, we denote by $E(\Gamma)$ the set of edges  and by $V(\Gamma)$ the set of vertices of $\Gamma$. If $\E\subset E(\Gamma)$ and $v$ is a vertex of $\Gamma$, we define the \emph{valence of $v$ in $\E$}, denoted by $\val_\E(v)$, as the number of edges in $\E$ incident to $v$ (with loops counting twice). In the case that $\E=E(\Gamma)$ we simply write $\val(v)$ and call it the \emph{valence} of $v$.	We say that $\Gamma$ is \emph{$k$-regular} if every vertex of $\Gamma$ have valence $k$; we say that $\Gamma$ is a \emph{circular graph} if $\Gamma$ is $2$-regular and connected. A \emph{cycle} on $\Gamma$ is a circular subgraph of $\Gamma$. \par
	A \emph{digraph} (directed graph) is a graph $\Gamma$ where each edge has a orientation, i.e., there are functions $s,t\col E(\Gamma)\to V(\Gamma)$ called \emph{source} and \emph{target} and each edge is oriented from the source to the target. We denote a digraph by $\overrightarrow{\Gamma}$ and by $\Gamma$ its underlying graph. We let $E(\overrightarrow{\Gamma})$ be the set of oriented edges of $\overrightarrow{\Gamma}$ and we set $V(\overrightarrow{\Gamma}):=V(\Gamma)$.
	\par
 Let $\Gamma$ be a graph. Fix disjoint subsets $V, W\subset V(\Gamma)$. We define $E(V,W)$ as the set of edges joining a vertex in $V$ with one in $W$. More generally if $V$ and $W$ have nonempty intersection, we define $E(V,W):=E(V\setminus W, W\setminus V)$. We set $V^c:=V(\Gamma)\setminus V$.  If $E(V,V^c)$ is nonempty, it is called a \emph{cut} of $\Gamma$. We set $\delta_{\Gamma,V}:=|E(V,V^c)|$. When no confusion may arise, we simply write $\delta_V$ instead of $\delta_{\Gamma,V}$. \par
		Fix a subset $\E\subset E(\Gamma)$. We define the graphs $\Gamma/\E$ and $\Gamma_\E$ as the graphs obtained by the contraction of edges in $\E$ and by the removal of edges in $\E$, respectively. We say that $\E$ is \emph{nondisconnecting} if $\Gamma_\E$ is connected. Note that $V(\Gamma_\E)=V(\Gamma)$ and $E(\Gamma_\E)=E(\Gamma)\setminus\E$. Also there is a natural surjection $V(\Gamma)\to V(\Gamma/\E)$ and a natural identification $E(\Gamma/\E)=E(\Gamma)\setminus\E$. Recall that a graph $\Gamma$ specializes to a graph $\Gamma'$ if there exists $\E\subset E(\Gamma)$ such that $\Gamma'$ is isomorphic to $\Gamma/\E$. We denote a specialization of $\Gamma$ to $\Gamma'$ by $\iota\col \Gamma\ra\Gamma'$. A specialization $\iota\col \Gamma\ra\Gamma'$ comes equipped with a surjective map $\iota^V\col V(\Gamma)\ra V(\Gamma')$ and an injective map $\iota^E\col E(\Gamma')\ra E(\Gamma)$. We usually write $\iota=\iota^V$ and we will see $E(\Gamma')$ as a subset of $E(\Gamma)$ via $\iota^E$. A similar notion of specialization can be given for digraphs. \par 
   We also define the graph $\Gamma^\E$ as the graph obtained from $\Gamma$ by adding exactly one vertex in the interior of each edge in $\E$. We call $\Gamma^\E$ the \emph{$\E$-subdivision} of $\Gamma$. Note that there is a natural inclusion $V(\Gamma)\subset V(\Gamma^\E)$. We call a vertex in $V(\Gamma^\E)\setminus V(\Gamma)$ an \emph{exceptional vertex}. We set $\Gamma^{(2)}:=\Gamma^{E(\Gamma)}$. \par
	More generally, we say that a graph $\Gamma'$ is a \emph{refinement} of the graph $\Gamma$ if $\Gamma'$ is obtained from $\Gamma$ by successive subdivisions. In other words, there is an inclusion $a\col V(\Gamma)\to V(\Gamma')$ and a surjection $b\col E(\Gamma')\to E(\Gamma)$ such that for any edge $e$ of $\Gamma$ there exist distinct vertices $x_0,\ldots, x_n\in V(\Gamma')$ and distinct edges $e_1,\ldots, e_n\in E(\Gamma')$ such that
	\begin{enumerate}
\item $x_0=a(v_0)$, $x_n=a(v_1)$ and $x_{i}\notin\Im(a)$ for every $i=1,\ldots,n-1$, where $v_0$ and $v_1$ are precisely the vertices of $\Gamma$ incident to $e$;
\item $b^{-1}(e)=\{e_1,\ldots, e_n\}$;
\item the vertices in $V(\Gamma')$ incident to $e_i$ are precisely $x_{i-1}$ and $x_i$, for $i=1,\dots,n$.
\end{enumerate}
We say that the edges $e_1,\ldots, e_n$ (respectively, $x_1,\ldots, x_{n-1}$) are the edges (respectively, the exceptional vertices) \emph{over} $e$.\par

	Let $\iota\col \Gamma\to\Gamma'$ be a specialization and $\E'\subset E(\Gamma')$ and $\E\subset E(\Gamma)$ be sets such that $\E'\subset\E\cap E(\Gamma')$. We call a specialization  $\iota^\E\col\Gamma^\E\ra\Gamma'^{\E'}$ \emph{compatible with $\iota$} if the following diagrams
\[
\SelectTips{cm}{11}
\begin{xy} <16pt,0pt>:
\xymatrix{ V(\Gamma)\ar[d]\ar[r]^{\iota} &V(\Gamma')\ar[d]&& E(\Gamma')\ar[r]^\iota &E(\Gamma)\\
             V(\Gamma^\E)\ar[r]^{\iota^\E}& V(\Gamma'^{\E'})&&E(\Gamma'^{\E'})\ar[u]\ar[r]^{\iota^\E} & E(\Gamma^\E)\ar[u]
}
\end{xy}
\]
are commutative. If $\iota$ is the identity of $\Gamma$, we just call $\iota^\E$ \emph{compatible}. Note that there are $2^{|\E\setminus\E'|}$ compatible specializations of $\Gamma^\E$ to $\Gamma^{\E'}$.

   A divisor $D$ on the graph $\Gamma$ is a function $D\colon V(\Gamma)\to \mathbb{Z}$. The degree of $D$ is the integer $\deg D:=\sum_{v\in V(\Gamma)}D(v)$. The set of divisors on $\Gamma$ forms an Abelian group denoted by $\Div(\Gamma)$. Given a subset $\E\subset E(\Gamma)$ and a divisor $D$ on $\Gamma$, we define $D^\E$ as the divisor on $\Gamma^\E$ such that 
\[
D^\E(v)=\begin{cases}
         \begin{array}{ll}
				  D(v),&\text{ if }v\in V(\Gamma);\\
					0,&\text{ if }v\in V(\Gamma^\E)\setminus V(\Gamma).
					\end{array}
					\end{cases}
\]
We also define a divisor $D_\E$ on $\Gamma_\E$ as $D_\E(v):=D(v)$, for every $v\in V(\Gamma_\E)=V(\Gamma)$.\par
	A \emph{pseudo-divisor} on the graph $\Gamma$ is a pair $(\E,D)$ where $\E\subset E(\Gamma)$ and $D$ is a divisor on $\Gamma^\E$ such that $D(v)=-1$ for every exceptional vertex $v\in V(\Gamma^\E)$. If $\E=\emptyset$, then $(\E,D)$ is just a divisor of $\Gamma$. Since every divisor $D$ on $\Gamma^\E$ can be lifted to a divisor on $\Gamma^{(2)}$, it is equivalent to define a pseudo-divisor on $\Gamma$ as a divisor $D$ on $\Gamma^{(2)}$ such that $D(v)=0,-1$ for every exceptional vertex $v\in V(\Gamma^{(2)})$.\par
	Let $\Gamma$ and $\Gamma'$ be graphs. Given a specialization $\iota\col\Gamma\ra\Gamma'$ and a divisor $D$ on $\Gamma$, we define the divisor $\iota_*(D)$ on $\Gamma'$ taking $v'\in V(\Gamma')$ to 
\[
\iota_*(D)(v'):=\sum_{v\in\iota^{-1}(v')}D(v).
\] 
	We say that a pair $(\Gamma,D)$ \emph{specializes} to a pair $(\Gamma',D')$, where $D$ is a divisor on $\Gamma$ and $D'$ is a divisor on $\Gamma'$, if there exists a specialization of graphs $\iota\col\Gamma\ra\Gamma'$ such that $D'=\iota_*(D)$; we denote by $\iota\col(\Gamma,D)\ra(\Gamma',D')$ a specialization of pairs. Note that, given a specialization $\iota\col \Gamma\ra\Gamma'$ and a subset $\E$ of $E(\Gamma)$, there exists an induced specialization $\iota^\E\col \Gamma^\E\to\Gamma'^{\E'}$, where $\E':=\E\cap E(\Gamma')$; in this case, if $(\E,D)$ is a pseudo-divisor on $\Gamma$, we define the pseudo-divisor $\iota_*(\E,D)$ on  $\Gamma'$ as $\iota_*(\E,D):=(\E',\iota_*^\E(D))$.
	Given pseudo-divisors $(\E,D)$ on $\Gamma$ and $(\E',D')$ on $\Gamma'$, we say that $(\Gamma,\E, D)$ \emph{specializes} to $(\Gamma',\E',D')$ if there exists a specialization $\iota\col\Gamma\to\Gamma'$, such that $\E'\subset \E\cap E(\Gamma')$ and there is a specialization $\iota^\E\col \Gamma^\E\ra\Gamma'^{\E'}$, compatible with $\iota$, such that $\iota^\E_*(D)=(D')$. We denote by $\iota\col (\Gamma,\E,D)\ra (\Gamma',\E',D')$ such a specialization. \par

	Let $\overrightarrow{\Gamma}$ be a digraph. A \emph{directed path} on $\overrightarrow{\Gamma}$ is a sequence $v_1, e_1, v_2,\ldots, e_n, v_{n+1}$ such that $s(e_i)=v_i$ and $t(e_i)=v_{i+1}$ for every $i=1,\dots,n$; a \emph{directed cycle} on $\overrightarrow{\Gamma}$ is a directed path on $\overrightarrow{\Gamma}$ such that $v_{n+1}=v_1$.  A \emph{cycle} on $\overrightarrow{\Gamma}$ is just a cycle on $\Gamma$.  \par
	
A \emph{source} (respectively, \emph{sink}) of $\overrightarrow{\Gamma}$ is a vertex in $V(\overrightarrow{\Gamma})$ such that $t(e)\neq v$ (respectively, $s(e)\neq v$) for every $e\in E(\overrightarrow{\Gamma})$. We say that $\overrightarrow{\Gamma}$ is {\emph{acyclic} if it has no directed cycles. It is a well known result that every acyclic (finite) digraph has at least a source and a sink. A \emph{flow} on $\overrightarrow{\Gamma}$ is a function $\phi \col E(\overrightarrow{\Gamma})\to \mathbb{Z}_{\geq 0}$. We say that $\phi$ is \emph{acyclic} if the digraph $\ora{\Gamma}/S$ is acyclic, where $S:=\{e\in E(\ora{\Gamma});\phi(e)=0\}$.  Moreover, we say that a flow $\phi$ is \emph{positive} if $\phi(e)>0$ for all $e\in E(\ora{\Gamma})$. Abusing terminology, we will say that a flow $\phi$ on a graph $\Gamma$ is a pair $(\ora{\Gamma},\phi)$ of an orientation on $\Gamma$ and a flow $\phi$ on $\ora{\Gamma}$.\par

	Given a graph $\Gamma$ and a ring $A$, we define 
\[
C_0(\Gamma,A):=\bigoplus_{v\in V(\Gamma)}A\cdot v\quad\text{and}\quad C_1(\Gamma,A):=\bigoplus_{e\in E(\Gamma)}A\cdot e.
\]
Fix a orientation on $\Gamma$, i.e., choose a digraph $\ora{\Gamma}$ with $\Gamma$ as underlying graph. We define the differential operator $d\col C_0(\ora{\Gamma},A)\to C_1(\ora{\Gamma}, A)$ as the linear operator taking a generator $v$ of $C_0(\ora{\Gamma},A)$ to
\[
d(v):=\underset{t(e)=v}{\sum_{e\in{E(\ora{\Gamma})}}}e-\underset{s(e)=v}{\sum_{e\in{E(\ora{\Gamma})}}}e.
\]
The adjoint of $d$ is the linear operator $d^*\col C_1(\ora{\Gamma},A)\to C_0(\ora{\Gamma},A)$ taking a generator $e$ of $C_1(\ora{\Gamma},A)$ to
\[
d^*(e):=t(e)-s(e).
\]
The space of $1$-cycles of $\ora{\Gamma}$ (over $A$) is defined as $H_1(\ora{\Gamma},A):=\ker(d^*)$.\par
We have a natural identification between $C_0(\Gamma,\mathbb{Z})$ and $\Div(\Gamma)$. Note that the composition $d^*d\col C_0(\Gamma,A)\to C_0(\Gamma,A)$ does not depend on the choice of the orientation. We define the group of principal divisors on $\Gamma$ as the subgroup $\Prin(\Gamma):=\Im(d^*d)$ of $\Div(\Gamma)$. Given $D,D'\in \Div(\Gamma)$, we say that $D$ is \emph{equivalent} to $D'$ if $D-D'\in \Prin(\Gamma)$.\par
  
  Given a flow $\phi$ on $\ora{\Gamma}$, we define the divisor associated to $\phi$ as the image $\div(\phi)=d^*(\phi)$, where $\phi$ is seen as an element of $C_1(\ora{\Gamma},\mathbb{Z})$. \par

	A \emph{(vertex) weighted graph} is a pair $(\Gamma,w)$, where $\Gamma$ is a connected graph and $w$ is a function $w\colon V(\Gamma)\to\mathbb{Z}_{\geq0}$, called the \emph{weight function}. The genus of a weighted graph $(\Gamma,w)$ is defined as $g(\Gamma):=\sum_{v\in V(\Gamma)} w(v)+b_1(\Gamma)$, where $b_1(\Gamma)$ is the first Betti number of $\Gamma$. \par
	A \emph{graph with legs} indexed by the finite set $L$ (the set of legs) is the data of a graph $\Gamma$ and a map $\text{leg}_\Gamma\col L\to V(\Gamma)$. Usually, we will simply write $\Gamma$ for a graph with legs and we denote by $L(\Gamma)$ its set of legs. We denote by $L(v)$ the set of legs incident to $v$, i.e., $L(v):=\text{leg}_\Gamma^{-1}(v)$. 	A \emph{graph with $n$ legs} is a  graph with legs $\Gamma$ such that $L(\Gamma)=I_n:=\{0,1,\ldots,n-1\}$. If $\Gamma$ is a graph with $n$ legs, we will always set $v_0:=\text{leg}_\Gamma(0)\in V(\Gamma)$. We will denote by $(\Gamma,v_0)$ a graph with $1$ leg. 
	
	If $(\Gamma,w,\text{leg}_{\Gamma})$ and $(\Gamma',w',\text{leg}_{\Gamma'})$ are weighted graphs with legs, we say that a specialization $\iota\col\Gamma\ra\Gamma'$ is a \emph{specialization of weighted graphs with legs} if, for every $v'\in V(\Gamma')$, we have  $w'(v')=g(\iota^{-1}(v))$, and $\text{leg}_{\Gamma}=\text{leg}_{\Gamma'}\circ \iota^V$.

	A weighted graph with $n$ legs $\Gamma$ is \emph{stable} if $\val(v)+2w(v)+|L(v)|\geq3$ for every vertex $v\in V(\Gamma)$. 
A \emph{tree} is a connected graph whose first Betti number is $0$ or, equivalently, a connected graph whose number of edges is equal to the number of its vertices minus one. A \emph{tree-like} graph is a graph that becomes a tree after contracting all the loops.
 		A \emph{spanning tree} $T$ of a graph $\Gamma$ is a subset $T\subset E(\Gamma)$ such that $\Gamma_{E(\Gamma)\setminus T}$ is a tree.  Note that, since a tree is connected, it follows that $\Gamma_{E(\Gamma)\setminus T}$ is connected, and hence every vertex of $\Gamma$ must be incident to an edge in $T$.

\begin{Lem}
\label{lem:tree}
Let $\Gamma$ be a connected graph. Let $T$ and $T'$ be distinct spanning trees of $\Gamma$. Then, there are spanning trees $T_0,T_1,\ldots, T_n$ of $\Gamma$, such that $T_0=T$, $T_n=T'$ and $|T_{i-1}\cap T_i|=|V(\Gamma)|-2$, for every $i=1,\dots,n$.
\end{Lem}
\begin{proof}
The idea is to take an edge $e$ in $T'\setminus T$ and construct a spanning tree $T_1$ by removing an edge $e'$ from $T\setminus T'$ and adding $e$. It is clear that $|T\cap T_1|=|T|-1=|V(\Gamma)|-2$ and the result will follow by iterating the reasoning.\par
Set $S:=T\cap T'$ and consider $e\in T'\setminus T$. Note that $e$ is not a loop. Hence, there is a unique cycle $\gamma$ on $\Gamma$ such that $e\in E(\gamma)$ and $\emptyset\ne E(\gamma)\setminus \{e\}\subset T$. Note that $\Gamma_{E(\Gamma)\setminus (S\cup\{e\})}$ does not have cycles, hence $E(\gamma)\setminus\{e\}$ is not contained in $S$. If we choose $e'$ as any edge in $E(\gamma)\setminus (S\cup \{e\})$, it follows that $(T\cup\{e\})\setminus \{e'\}$ is a spanning tree.  
\end{proof}

    The category whose objects are weighted graphs with $n$ legs of genus $g$ and whose morphisms are specializations is denoted by $\Graph_{g,n}$. The poset $\grap_{g,n}$ will be the set whose elements are isomorphism classes of weighted graphs with $n$ legs of genus $g$, where the partial order is given by specialization.

\subsection{Partially ordered sets}
\label{sub:tropicalcurves}
  A \emph{poset} (\emph{partially ordered set}) is a pair $(S,\leq)$ where $S$ is a set and $\leq$ is a partial order on $S$. In this paper we will only consider finite posets. A function $f\col S\to S'$ is called \emph{order preserving} if $x\leq y$ implies $f(x)\leq f(y)$. A subset $T\subset S$ is called a \emph{lower set} (respectively, \emph{upper set}) if $x\leq y$ (respectively, $y\leq x$) and $y\in T$ implies that $x\in T$. A poset $S$ can be given two natural topologies, one where the closed sets are the lower sets, and the other where the closed sets are the upper sets. In this paper we choose the topology induced by the lower sets.\par
	A function $f\col S\to S'$ between posets is continuous if and only if it is order preserving. Moreover, a continuous function $f$ is closed if for every $y_1, y_2\in S'$ with $y_1\leq y_2$ and $y_2=f(x_2)$, there exists $x_1\in S$ such that $x_1\le x_2$ and $f(x_1)=y_1$.\par
Let $S$ be a poset. A \emph{chain} in $S$ is a sequence $x_0<x_1<\ldots < x_n$. We call $n$ the \emph{length} of the chain, and we say that $x_0$ (respectively, $x_n$) is the \emph{starting point} (respectively, \emph{ending point}) of the chain. We say that $S$ is \emph{ranked} (of length $n$) if every maximal chain has length $n$. It is easy to see that the maximal length of chains in $S$ is precisely the Krull dimension of $S$ as a topological space. Therefore, if $S$ is ranked then $S$ will be pure dimensional. 
For every $x\in S$ we define the \emph{dimension} of $x$ as the maximum length for a chain ending in $x$; in other words, the dimension of $x$ is precisely the Krull dimension of $\overline{\{x\}}$. If $S$ is a ranked poset of length $n$, then we define the \emph{codimension} of $x\in S$ as $n-\dim_S(x)$, i.e., the length of all maximal chains starting from $x$. A poset $S$ is \emph{graded} if it has a function $\text{rk}\col S\to \mathbb{N}$, called \emph{rank function}, such that $\text{rk}(x)=\text{rk}(y)+1$ whenever $y<x$ and there is no $z\in S$ with $y<z<x$. Every ranked poset is graded with rank function given by the dimension.\par
	
Let $S$ be a ranked poset $S$. We say that $S$ is \emph{connected in codimension one} if for every maximal elements $y,y'\in S$ there are two sequences of elements $x_1,\ldots, x_n\in S$ and $y_0,\ldots, y_n\in S$ such that 
	\begin{enumerate}[label=(\roman*)]
	\item $x_i$ has codimension $1$ for every $i=1,\ldots, n$.
	\item $y_i$ is maximal for every $i=0,\ldots,n$.
	\item $y_0=y$ and $y_n=y'$.
	\item $x_{i+1}< y_i$ for every $i=0,\ldots,n-1$ and $x_i<y_i$ for every $i=1,\ldots, n$.
	\end{enumerate}
We call these sequences a \emph{path in codimension $1$} from $y$ to $y'$.	

\section{Tropical curves and their moduli}\label{sec:trop}
\subsection{Tropical curves}
 A \emph{metric graph} is a pair $(\Gamma,\l)$ where $\Gamma$ is a graph and $\l$ is a function $\l\col E(\Gamma)\to \mathbb{R}_{>0}$ called the length function. If $\ora{\Gamma}$ is an orientation on $\Gamma$, we define the \emph{tropical curve} $X$ associated to $(\ora{\Gamma},\ell)$ as
\[
X=\frac{\left(\bigcup_{e\in E(\ora{\Gamma})}I_e\cup V(\ora{\Gamma})\right)}{\sim}
\]
where $I_e=[0,\l(e)]\times\{e\}$ and $\sim$ is the equivalence relation generated by $(0,e)\sim s(e)$  and $(\l(e),e)\sim t(e)$. The tropical curve $X$ has a natural topology and its connected components have a natural structure of metric space. Note that the definition of $X$ does not depend on the chosen orientation. For two points $p,q\in I_e$ we denote by $\overline{pq}$ (respectively, $\overrightarrow{pq}$) the interval (respectively, oriented interval) $[p,q]\subset I_e$.\par

  We say that the tropical curves $X$ and $Y$ are \emph{isomorphic} if there is a bijection between $X$ and $Y$ that is an isometry over each connected component of $X$. We say that $\Gamma$ (respectively, $\overrightarrow{\Gamma}$) is a  \emph{model} (respectively, \emph{directed model}) of $X$ if there exists a length function $\l$ on $\Gamma$ such that $X$ and the tropical curve associated to $(\Gamma,\l)$ are isomorphic. Sometimes, we will use interchangeably the notion of tropical curve and of its isomorphism class. If $X$ and $Y$ are tropical curves and $\Gamma_X$ and $\Gamma_Y$ are models for $X$ and $Y$, we say that the pairs $(X,\Gamma_X)$ and $(Y,\Gamma_Y)$ are \emph{isomorphic} if there exists an isomorphism $f\col X\to Y$ of tropical curves such that $f(V(\Gamma_X))=V(\Gamma_Y)$ and such that $f$ induces an isomorphism of graphs between $\Gamma_X$ and $\Gamma_Y$.  \par
	Let $X$ be a tropical curve. If $X$ is connected, then every model of $X$ is a connected graph. Conversely, if a model of $X$ is connected, then so is $X$. Note that $X$, as topological space, might fail to be pure dimensional: this happens, for instance, if $\Gamma$ has isolated vertices.\par
  We say that a tropical curve $Y$ is a \emph{tropical subcurve} of $X$ if there exists an injection $Y\subset X$ that is an isometry over each connected component of $Y$. In this case, if $\Gamma_Y$ is a model of $Y$, one can choose a model $\Gamma_X$ of $X$ such that $\Gamma_Y$ is a subgraph of $\Gamma_X$. Conversely, if $\Gamma'$ is a subgraph of a model $\Gamma_X$ of $X$, then $\Gamma'$ induces a tropical subcurve $Y$ of $X$. Note that a tropical subcurve can be a single point. If $Y$ and $Z$ are tropical subcurves of $X$ then $Y\cap Z$ and $Y\cup Z$ are also tropical subcurves of $X$. 
  From now on  all tropical curves will be connected, and in particular pure dimensional, while we will allow possibly nonconnected (and hence possibly non pure dimensional) tropical subcurves. \par
	
Let $X$ be a tropical curve with a model $\Gamma_X$, and let $Y\subset X$ be a tropical subcurve of $X$. Then, there exists a minimal refinement $\Gamma_{X,Y}$ of $\Gamma_X$ such that $Y$ is induced by a subgraph $\Gamma_Y$ of $\Gamma_{X,Y}$. We define
\[
\delta_{X,Y}:=\sum_{v\in V(\Gamma_Y)}\val_{E(\Gamma_{X,Y})\setminus E(\Gamma_Y)}(v)
\]
The definition of $\delta_{X,Y}$ does not depend on the choice of the model $\Gamma_X$ of $X$. 	When no confusion may arise we will simply write $\delta_Y$ instead of $\delta_{X,Y}$.

	We also define the set 
\[
\out(Y):=E(V(\Gamma_Y),V(\Gamma_{X,Y})\setminus V(\Gamma_Y)).
\]
Equivalently, the set $\out(Y)$ is the cut in $\Gamma_{X,Y}$ induced by $V(\Gamma_Y)$. The definition of the set $\out(Y)$ depends on the choice of the model $\Gamma_X$ (see the next example). \par 
\begin{Exa}
\label{exa:out}
Let $X$ be a tropical curve as in Figure \ref{fig:out}, with model $\Gamma_X$ whose vertices are the points $p_0,p_1,p_2,p_3,p_4$. Let $Y$ be the subcurve of $X$, given by 
\[
Y:=\overline{p_0q_0}\cup \overline{p_0q_1}\cup \overline{p_0q_3}\cup \overline{p_3q_4}\cup \{q_2\}.
\]
Then we have
\[
\out(Y)=\{\overline{q_0p_2},\overline{q_2p_1},\overline{q_3p_1},\overline{q_4p_4}\}.
\]
Note that the edge $\overline{q_1q_2}$ does not belong to $\out(Y)$. If we choose a model for $X$ with vertices the points $p_0,p_1,p_3,p_4$, then the edge $\overline{q_0p_1}$ will be in $\out(Y)$ instead of $\overline{q_0p_2}$. On the other hand, if we choose a model with vertices $p_0,p_1,p_2,p_3,p_4,p_5$, for some $p_5$ in the interior of $\overline{q_1q_2}$, then the edges $\overline{q_1p_5}$ and $\overline{q_2p_5}$ will be in $\out(Y)$.
\begin{figure}[h]
\begin{tikzpicture}[scale=5]
\begin{scope}[shift={(0,0)}]
\draw (0,0) to [out=30, in=150] (1,0);
\draw (0,0) to (1,0);
\draw (0,0) to [out=-30, in=-150] (1,0);
\draw (0.5,-0.144) to (0,-0.344);
\draw[fill] (0,0) circle [radius=0.02];
\node[left] at (0,0) {$p_0$};
\draw[fill] (1,0) circle [radius=0.02];
\node[right] at (1,0) {$p_1$};
\draw[fill] (0.5,0.144) circle [radius=0.02];
\node[above] at (0.5,0.144) {$p_2$};
\draw[fill] (0.5,-0.144) circle [radius=0.02];
\node[below] at (0.5,-0.144) {$p_3$};
\draw[fill] (0,-0.344) circle [radius=0.02];
\node[below] at (0,-0.344) {$p_4$};
\draw[fill] (0.3,0.115) rectangle (0.33,0.145);
\node[above] at (0.315,0.130) {$q_0$};
\draw[fill] (0.3,-0.015) rectangle (0.33,0.015);
\node[below] at (0.315,0) {$q_1$};
\draw[fill] (0.67,-0.015) rectangle (0.7,0.015);
\node[below] at (0.685,0) {$q_2$};
\draw[fill] (0.67,-0.145) rectangle (0.7,-0.115);
\node[below] at (0.685,-0.130) {$q_3$};
\draw[fill] (0.3,-0.205) rectangle (0.33,-0.235);
\node[below] at (0.315,-0.220) {$q_4$};
\end{scope}
\end{tikzpicture}
\caption{An example of the set $\out(Y)$.}
\label{fig:out}
\end{figure}
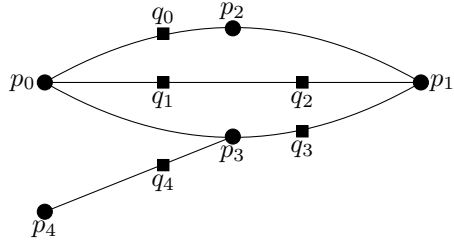
\end{Exa}

  Let $X$ be a tropical curve, $\Gamma_X$ a model of $X$, and $\l$ the induced metric on $\Gamma_X$. A specialization $\iota\col\Gamma_X\ra\Gamma'$ induces a metric $\l'$ on $\Gamma'$ defined as
\[
\l'\colon E(\Gamma')\stackrel{\iota_*}{\hookrightarrow} E(\Gamma)\stackrel{\l}{\to} \R_{>0}.
\]
Let $Y$ be the tropical curve associated to $(\Gamma',\l')$. Then there exists an induced function $\iota\col X\to Y$ that is constant on the edges of $\Gamma_X$ contracted by $\iota$. We call this function a \emph{specialization of $X$ to $Y$}.\par

   Let $X$ be a tropical curve and $(\ora{\Gamma},\l)$ be a directed model of $X$. Given $e\in E(\ora{\Gamma})$  and $a\in\R$ with $0\le a\le\l(e)$, we let $p_{a,e}$ be the point of $X$ lying on $e$ whose distance from $s(e)$ is $a$. A \emph{divisor} on $X$ is a map $\D\col X\to \mathbb{Z}$ such that $\D(p)\neq0$ for finitely many points $p\in X$. We define the \emph{support} of $\D$ as the set of points $p$ of $X$ such that $\D(p)\neq0$ and denote it by $\supp(\D)$.  

If $\Gamma$ is a model of $X$, then every divisor on $\Gamma$ can be seen as a divisor on $X$. Given a divisor $\D$ on $X$, the degree of $\D$ is the integer $\deg \D:=\sum_{p\in X} \D(p)$. We say that $\D$ is \emph{effective} if $\D(p)\ge0$ for every $p\in X$.  We let $\Div(X)$ be the Abelian group of divisors on $X$; we denote by $\Div^d(X)$ the subset of degree-$d$ divisors on $X$. If $X$ and $Y$ are tropical curves and if $\D$ and $\D'$ are divisors on $X$ and $Y$, respectively, we say that the pairs $(X,\D)$ and $(Y, \D')$ are \emph{isomorphic} if there is an isomorphisms $f\col X\to Y$ of tropical curves such that $\D(p)=\D'(f(p))$ for every $p\in X$.\par 
   A \emph{rational function} on $X$ is a continuous, piecewise linear function $f\col X\ra \R$ with integer slopes. We say that a rational function $f$ on $\Gamma$ has \emph{slope $s$ over} $\ora{p_{e,a_1},p_{e,a_2}}$, for $a_1,a_2\in\R$ and $0\leq a_1< a_2\leq \ell(e)$, if the restriction of $f$ to the locus of points $p_{a,e}$ for $a_1\leq a \leq a_2$ is linear and has slope $s$.
	A \emph{principal divisor} on $\Gamma$ is a divisor 
\[
\div_X(f):=\sum_{p\in X} \ord_p(f) p\in \Div(X),
\]
where $f$ is a rational function on $X$ and $\ord_p(f)$ is the sum of the incoming slopes of $f$ at $p$. A principal divisor has degree zero. The \emph{support} of a rational function $f$ on $X$ is defined as the set $\supp(f)=\{p\in X;\ord_p(f)\neq 0\}$. We denote by $\Prin(X)$ the subgroup of $\Div(X)$ of principal divisors. Given divisors $\D_1,\D_2\in\Div(X)$, we say that $\D_1$ and $\D_2$ are \emph{equivalent} if $\D_1-\D_2\in \Prin(X)$. The \emph{Picard group} $\Pic(X)$ of $X$ is
\[
\Pic(X):=\Div(X)/\Prin(X).
\] 
The \emph{degree-$d$ Picard group} of $X$ is defined as $\Pic^d(X):=\Div^d(X)/\Prin(X)$.\par
 
	Let $f$ be a rational function on the tropical curve $X$ and $\Gamma$ be a model of $X$ such that $\supp(f)\subset V(\Gamma)$. Then $f$ is linear over each edge of $\Gamma$. If $f$ is nowhere constant, then it induces an orientation $\ora{\Gamma}$ on $\Gamma$, such that $f$ has always positive  slopes. In this case, we can define a positive flow $\phi_f$ on $\ora{\Gamma}$ where $\phi_f(e)$ is equal to the slope of $f$ over $e$, for every $e\in E(\ora\Gamma)$. It is clear that $\div_X(f)=\div(\phi_f)$, where $\div(\phi_f)$ is seen as a divisor on $X$. Note that the orientation $\ora{\Gamma}$ on $\Gamma$ induced by a rational function $f$ is acyclic, because there are no strictly increasing functions on the circle $S^1$. If $f$ is constant on a subset $\E\subset E(\Gamma)$, then we can contract all edges in $\E$ and get a specialization $\iota\col X\to Y$ of tropical curves. Clearly, $f$ induces a nowhere constant rational function on $Y$. Hence $f$ induces an acyclic orientation $\ora{\Gamma/\E}$ on $\Gamma/\E$ and a positive flow $\phi_f$ on $\Gamma/\E$.\par
  
	Fix a model $\Gamma_X$ for the tropical curve $X$. For every tropical subcurve $Y\subset X$ and for every $\l\in\R_{\geq0}$ such that $\l\leq \min_{e\in\out(Y)}\{\l(e)\}$, we define the divisor $\D_{Y,\l}$ on $X$ taking $p\in X$ to
\begin{equation}
\label{eq:chip}
\D_{Y,\l}(p)=\begin{cases}
             \begin{array}{ll}
						  -|\{e\in \out(Y); p=p_{e,0}\}| &\text{if $p\in Y$}; \\
							|\{e\in \out(Y); p=p_{e,\l}\}| &\text{if $p\notin Y$},
							\end{array}
							\end{cases}
\end{equation}							
where we consider the edges $e\in\out(Y)$ oriented in the direction away from $Y$. The divisor $\D_{Y,\l}$ is principal, since we have $\D_{Y,\l}=\div_X(f)$ where $f$ has slope $1$ over $\overrightarrow{p_{e,0}p_{e,\l}}$ for every edge $e\in\out(Y)$ and has slope $0$ everywhere else. Note that $f$ is well defined because the edges in $\out(Y)$ form a cut of $\Gamma_{X,Y}$, while the definition of $\D_{Y,\l}$ depends on the choice of the model $\Gamma_X$ of $X$. We call $\D_{Y,\l}$ a \emph{chip-firing divisor emanating from $Y$ with length $\l$}.
 
\begin{Exa}
\label{exa:chip}
We maintain the definitions in Example \ref{exa:out}, where the model $\Gamma_X$ has vertices $p_0,p_1,p_2,p_3,p_4$. Assume that the edges $\overline{q_0p_2}$, $\overline{q_2q_5}$, $\overline{q_3p_1}$ and $\overline{q_4q_6}$ have all the same length $\l$. Hence 
\[
\D_{Y,\l}=(p_2-q_0)+(q_5-q_2)+(p_1-q_3)+(q_6-q_4).
\]
\begin{figure}[h]
\begin{tikzpicture}[scale=5]
\begin{scope}[shift={(0,0)}]
\draw (0,0) to [out=30, in=150] (1,0);
\draw (0,0) to (1,0);
\draw (0,0) to [out=-30, in=-150] (1,0);
\draw (0.5,-0.144) to (0,-0.344);
\draw[fill] (0,0) circle [radius=0.02];
\node[left] at (0,0) {$p_0$};
\draw[fill] (1,0) circle [radius=0.02];
\node[right] at (1,0) {$p_1$};
\draw[fill] (0.5,0.144) circle [radius=0.02];
\node[above] at (0.5,0.144) {$p_2$};
\draw[fill] (0.5,-0.144) circle [radius=0.02];
\node[below] at (0.5,-0.144) {$p_3$};
\draw[fill] (0,-0.344) circle [radius=0.02];
\node[below] at (0,-0.344) {$p_4$};
\draw[fill] (0.3,0.115) rectangle (0.33,0.145);
\node[above] at (0.315,0.130) {$q_0$};
\draw[fill] (0.3,-0.015) rectangle (0.33,0.015);
\node[below] at (0.315,0) {$q_1$};
\draw[fill] (0.67,-0.015) rectangle (0.7,0.015);
\node[below] at (0.685,0) {$q_2$};
\draw[fill] (0.84,-0.015) rectangle (0.87,0.015);
\node[below] at (0.855,0) {$q_5$};
\draw[fill] (0.67,-0.145) rectangle (0.7,-0.115);
\node[below] at (0.685,-0.130) {$q_3$};
\draw[fill] (0.3,-0.205) rectangle (0.33,-0.235);
\node[below] at (0.315,-0.220) {$q_4$};
\draw[fill] (0.15,-0.265) rectangle (0.18,-0.295);
\node[below] at (0.165,-0.280) {$q_6$};
\end{scope}
\end{tikzpicture}
\caption{An example of a chip firing divisor.}
\label{fig:chip}
\end{figure}
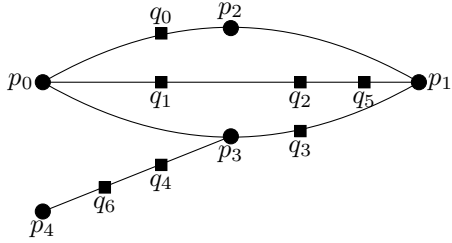

\end{Exa}

\begin{Lem}
\label{lem:valP}
If $f$ is a nowhere constant rational function on a tropical curve $X$, then there exists a point $p$ in $X$ such that $\ord_p(f)\geq\delta_{X,p}$.
\end{Lem}
\begin{proof}
The rational function $f$ induces a positive flow on a directed model $\ora{\Gamma_X}$ of $X$. The result follows from choosing $p$ to be a sink of $\ora{\Gamma_X}$.
\end{proof}
\begin{Rem}
\label{rem:uniquef}
Let $f$ and $f'$ be rational functions on a tropical curve $X$. Then $\div(f)=\div(f')$ if and only if $f-f'$ is a constant function. Indeed, it suffices to show that if $\div(f)=0$, then $f$ is constant. This follows by contracting the edges where $f$ is constant: if the resulting tropical curve is not a point, then using Lemma \ref{lem:valP} we see that $\div(f)\ne 0$.
\end{Rem}

  Let $X$ be a tropical curve. Given a directed model $\ora{\Gamma_X}$ of $X$, a \emph{$1$-form}  on $X$ is a formal sum $\omega=\sum_{e\in E(\ora{\Gamma_X})}\omega_e\cdot de$, where $\omega_e\in\R$. We say that $\omega$ is \emph{harmonic} if for every $v\in V(\ora{\Gamma_X})$, 
\[
\underset{s(e)=v}{\sum_{e\in E(\ora{\Gamma_X})}}\omega_e=\underset{t(e)=v}{\sum_{e\in E(\ora{\Gamma_X})}}\omega_e.
\]
Let $\Omega(X)$ be the real vector space of harmonic $1$-forms and $\Omega(X)^\vee$ be its dual. Note that $\Omega(X)$ does not depend on the choice of $\ora{\Gamma_X}$. Given edges $e,e'\in E(\ora{\Gamma_X})$ and points $p_{e',a_1},p_{e',a_2}\in e'$ for $a_1,a_2\in\mathbb R$, we define the \emph{integration of $de$ over} $\overrightarrow{p_{e',a_1}p_{e',a_2}}$ as
\[
\int_{p_{e',a_1}}^{p_{e',a_2}}de=\begin{cases}
                                 \begin{array}{ll}
														      a_2-a_1,&\text{ if $e'=e$;}\\
																	0,&\text{ otherwise.}
																	\end{array}
																	\end{cases}
\]
There is a natural isomorphism (see \cite[Lemma 2.1]{BF})
\begin{align*}
H_1(\Gamma_X,\R)&\to \Omega(X)^\vee\\
  \gamma&\mapsto \int_\gamma\col
\end{align*}
and hence $H_1(\Gamma_X,\mathbb{Z})$ can be viewed as a lattice in $\Omega(X)^\vee$. The \emph{Jacobian} of the tropical curve $X$ is defined as the real torus 
\begin{equation}\label{eq:Jtropdef}
J^{\trop}(X)=\Omega(X)^\vee/H_1(\Gamma_X,\mathbb{Z}).
\end{equation}

Note that the definition of the Jacobian does not depend on the chosen directed model of $X$.
	Fix a point $p_0$ in $X$ and assume that $p_0$ is a vertex of $\ora{\Gamma_X}$. Let $p_{e_1,a_1},\ldots,p_{e_d,a_d}$ be points of $X$. Choose a path $\gamma_i$ on $\ora{\Gamma_X}$ from $p_0$ to $s(e_i)$ for every $i=1,\dots,d$.  One can define a map
\begin{align*}
\alpha\col X^d&\longrightarrow \Omega(X)^\vee\\
    (p_{e_i,a_i})_{i=1,\ldots,d}&\longmapsto \sum_{i=1}^{d} \left(\int_{\gamma_i}+\int_{s(e_i)}^{p_{e_i,a_i}}\right).
\end{align*}	
		
The \emph{degree-$d$ Abel-Jacobi map} of the tropical curve $X$ is the composition of $\alpha$ with the quotient map $\Omega(X)^\vee\to J^\trop(X)$. Note that, while the map $\alpha$ may depend on the choices of the paths $\gamma_1,\dots,\gamma_d$, the degree-$d$ Abel-Jacobi map does not.

  An \emph{$n$-pointed tropical curve} is a pair $(X,\text{leg})$ where $\text{leg}\col I_n=\{0,\dots,n-1\} \to X$ is a function. For every $p\in X$, we set $\l(p):=|\text{leg}^{-1}(p)|$. A \emph{weight} on a tropical curve $X$ is a function $w\col X\to \mathbb{Z}_{\geq0}$, such that $w(p)\neq 0$ only for finitely many $p\in X$. Given an $n$-pointed weighted tropical curve $X$, we define 
\[
V(X):=\{p\in X; \text{ either }\delta_{X,p}\neq 2,\text{or }\l(p)\geq1,\text{or }w(p)\geq 1\}.
\]

Note that $V(X)\subset V(\Gamma)$ for every model $\Gamma$ of $X$; if $V(X)=V(\Gamma)$, we say that $\Gamma$ is a \emph{minimal model} of $X$. We say that an $n$-pointed weighted tropical curve $(X,w,\text{leg})$ is \emph{stable} if $\delta_{X,p}+2w(p)+\l(p)\ge 3$ for every point $p\in X$ such that $\delta_{X,p}\leq1$.
The genus $g(X,w)$ of a weighted tropical curve $(X,w)$ is defined by the formula
\[
2g(X,w)-2:=\sum_{p\in X}(2w(p)-2+\delta_{X,p}).
\]
The sum on right hand side of the last formula is finite, since $(\delta_{X,p}, w(p))\neq (2,0)$ only for finitely many $p\in X$.

If $(X,w,\text{leg})$ is a stable $n$-pointed weighted tropical curve, then there exists a unique stable weighted graph with $n$ legs $\Gamma_{st}$ that is a model of $X$ with induced weights and legs  satisfying the following property:  for every $p\in X$ such that either $w(p)\neq0$ or $\l(p)\neq0$, then $p\in V(\Gamma_{st})$. We call $\Gamma_{st}$ the \emph{stable model} of $X$.\par
 For the remainder of the paper we will usually omit to denote the weight $w$ of a weighted tropical curve.

\subsection{Tropical moduli spaces}
 Given a finite set $S\subset \R^n$ we define 
\[
\cone(S):=\left\{\sum_{s\in S}\lambda_ss|\lambda_s\in \mathbb R_{\ge0}\right\}.
\]
A subset $\sigma\subset \R^n$ is called a \emph{polyhedral cone} if $\sigma=\cone(S)$ for some finite set $S\subset \mathbb{R}^n$. If there exists $S\subset \mathbb{Z}^n$ with $\sigma=\cone(S)$ then $\sigma$ is called \emph{rational}.\par
   Every polyhedral cone is the intersection of finitely many closed half-spaces. The \emph{dimension} of $\sigma$, denoted $\dim(\sigma)$, is the dimension of the minimal linear subspace containing $\sigma$.
The \emph{relative interior} $\sigma^\circ$ is the interior of $\sigma$ inside this linear subspace. A \emph{face} of $\sigma$ is the intersection of $\sigma$ with some linear subspace $H\subset \mathbb{R}^n$ of codimension one such that $\sigma$ is contained in one of the closed half-spaces determined by $H$. A face of $\sigma$ is also a polyhedral cone. If $\tau$ is a face of $\sigma$ then we write $\tau\prec \sigma$. In the sequel we will use the terminology \emph{cone} to mean \emph{rational polyhedral cone}. \par
	 A morphism $f\col \tau\to\sigma$ between cones $\tau\subset \R^m$ and $\sigma\subset \R^n$ is the restriction to $\tau$ of an integral linear transformation $T\col\R^n\to \R^m$ such that $T(\tau)\subset\sigma$. We say that $f$ is an isomorphism if there exists an inverse morphism $f^{-1}\col\sigma\to\tau$.
	 A morphism $f\col\tau\to\sigma$ is called a \emph{face morphism} if $f$ is an isomorphism between $\tau$ and a (not necessarily proper) face  of $\sigma$.\par
	A \emph{generalized cone complex} $\Sigma$ is the colimit (as topological space) of a finite diagram $\mathbf{D}$ of cones with face morphisms. \par
	We say that $\sigma\in \mathbf{D}$ is \emph{maximal} if there is no proper face morphism $f\col\sigma\to\tau$ in $\mathbf{D}$. We say that $\Sigma$ is \emph{pure dimensional} if all maximal cones in $\mathbf{D}$ have the same dimension. We say that $\Sigma$ is \emph{connected through codimension one} if for every pair $\sigma,\sigma'$ of maximal cones there are two sequences of cones $\tau_1,\ldots,\tau_n\in \mathbf{D}$ and $\sigma_0,\ldots, \sigma_n\in \mathbf{D}$ such that
		\begin{enumerate}[label=(\roman*)]
	\item $\tau_i$ has codimension $1$ for every $i=1,\ldots, n$;
	\item $\sigma_i$ is maximal for every $i=0,\ldots,n$;
	\item $\sigma_0=\sigma$ and $\sigma_n=\sigma'$;
	\item there exists a face morphism $\tau_{i+1}\to \sigma_i$ in $\mathbf{D}$ for every $i=0,\ldots,n-1$, and a face morphism $\tau_i\to\sigma_i$ in $\mathbf{D}$ for every $i=1,\ldots, n$.
	\end{enumerate}\par
	 A \emph{morphism of generalized cone complexes} is a continuous map of topological spaces $f\col\Sigma\to\Sigma'$ such that for every cone $\sigma\in \mathbf{D}$ there exists a cone $\sigma'\in \mathbf{D}'$ such that the induced map $\sigma\to\Sigma'$ factors through a cone morphism $\sigma\to\sigma'$. \par
 A \emph{polyhedron} $\P\subset \R^n$ is an intersection of a finite number of half-spaces of $\R^n$. A face of a polyhedron $\P$ is the intersection of $\P$ and a hyperplane $H$ such that $\P$ is contained in a closed half-space determined by $H$. A morphism $f\col \P\ra \P'$ between polyhedra $\P\subset \R^n$ and $\P'\subset \R^m$ is the restriction to $\P$ of an affine map $T\col\R^n\to\R^m$ such that $T(\P)\subset \P'$. A morphism $f\col \P\to \P'$ of polyhedra is a \emph{face morphism} if the image of $f(\P)$ is a face of $\P'$ and $f$ is an isometry. A \emph{polyhedral complex} is the colimit (as topological space) of a finite poset $\mathbf{D}$ of polyhedra with face morphisms. \par
		
  For a graph $\Gamma$, the open cone $\mathbb{R}_{> 0}^{|E(\Gamma)|}$ parametrizes all  possible choices for the lengths of the edges of $\Gamma$. Hence, $M_\Gamma^{\trop}:=\R^{E(\Gamma)}_{>0}/\Aut(\Gamma)$ parametrizes isomorphism classes of pairs $(X,\Gamma_X)$, where $X$ is a tropical curve and $\Gamma_X$ is a model of $X$ isomorphic to $\Gamma$.  We will identify $E(\Gamma)$ with the canonical basis of $\mathbb{R}^{|E(\Gamma)|}$. If $\iota\col \Gamma\ra \Gamma'$ is a specialization, then there exists an inclusion $\R^{E(\Gamma')}\subset \R^{E(\Gamma)}$ induced by the inclusion $E(\Gamma')\subset E(\Gamma)$. If $\Gamma'$ is a refinement of $\Gamma$, there exists a map 
\begin{equation}
\label{eq:hat}
f\col\R^{E(\Gamma')}\to \R^{E(\Gamma)}
\end{equation}
given by $f((x_{e'})_{{e'}\in E(\Gamma')})=(y_e)_{e\in E(\Gamma)}$, with 
\[
y_e=\underset{b(e')=e}{\sum_{e'\in E(\Gamma')}}x_{e'}.
\]
(Recall that $b\col E(\Gamma')\to E(\Gamma)$ is the surjection induced by the refinement $\Gamma'$.)  \par
		If $X$ is a stable $n$-pointed genus-$g$ tropical curve, then $X$ admits exactly one stable model $\Gamma_X$. Hence the moduli space $M_{g,n}^{\text{trop}}$ of stable $n$-pointed tropical curves of genus $g$ is the generalized cone complex given as the colimit of the diagram whose cones are $\mathbb{R}_{\geq0}^{|E(\Gamma)|}$, where $\Gamma$ runs through all stable genus-$g$ weighted graphs with $n$ legs, with face morphisms specified by specializations. More precisely, if $\Gamma$ specializes to $\Gamma'$, then $\mathbb{R}_{\geq 0}^{|E(\Gamma')|}$ is a face of $\mathbb{R}_{\geq 0}^{|E(\Gamma)|}$ via the inclusion $E(\Gamma')\subset E(\Gamma)$.  For more details about $M_{g,n+1}^{\text{trop}}$ and its compactification $\overline{M}_{g,n+1}^{\textnormal{trop}}$, see \cite[Section 2]{M},  \cite[Sections 2.1 and 3.2]{BMV}, \cite[Section 3]{Caporaso1}, \cite[Section 3]{Caporaso}, and \cite[Section 4]{ACP}.\par

\section{Quasistability on graphs}\label{sec:quasigraph}

In this section we introduce the notion of quasistability for pseudo-divisors on graphs. We will study several properties of the poset formed by these divisors, in particular how it behaves under the operations of contraction and deletion of edges. All graphs will be considered connected unless otherwise specified. 

Let $\Gamma$ be a graph. Let $d$ be an integer. A \emph{degree-$d$ polarization} on $\Gamma$ is a function $\mu\col V(\Gamma)\to\R$ such that $\sum_{v\in V(\Gamma)}\mu(v)=d$. 

Let $\mu$ be a degree-$d$ polarization on $\Gamma$. For every $V\subset V(\Gamma)$ we set $\mu(V):=\sum_{v\in V}\mu(v)$.
Given a degree-$d$ divisor $D$ on $\Gamma$, we define 
\[
\beta_D(V):=\deg(D|_V)-\mu(V)+\frac{\delta_V}{2}.
\] 
If $\iota\col\Gamma\ra\Gamma'$ is a specialization, then there is a induced degree-$d$ polarization $\iota_*(\mu)$ on $\Gamma'$ defined as 
\[
\iota_*(\mu)(v'):=\underset{v \in \iota^{-1}(v')}{\sum_{v\in V(\Gamma)}}\mu(v).
\]
If $\E\subset E(\Gamma)$ is a nondisconnecting subset of edges, then there is an induced degree-$(d+|\E|)$ polarization $\mu_\E$ on $\Gamma_\E$ (the graph induced by removing the edges $\E$ in $E(\Gamma)$) given as 
\[
\mu_\E(v):=\mu(v)+\frac{1}{2}\val_\E(v).
\]
Given a subdivision $\Gamma^{\E}$ of $\Gamma$ for some $\E\subset E(\Gamma)$, there is an induced degree-$d$ polarization $\mu^\E$ on $\Gamma^\E$ given as
\[
\mu^\E(v):=
\begin{cases}
\begin{array}{ll}
\mu(v)& \text{if}\;v\in V(\Gamma);\\
0&\text{otherwise.}
\end{array}
\end{cases}
\]

We now prove a numerical result concerning the function $\beta_D$, which is, essentially, the same statement as in \cite[Lemma 3]{Es01} (in the case of invertible sheaves).
\begin{Lem}
\label{lem:cap}
Given subsets $V$ and $W$ of $V(\Gamma)$, we have
\[
\beta_D(V\cup W)+\beta_D(V\cap W)=\beta_D(V)+\beta_D(W)-|E(V,W)|.
\]
In particular, $\beta_D(V)+\beta_D(V^c)=\delta_V$.
\end{Lem} 
\begin{proof}
It is sufficient to prove that
\[
\delta_{V\cup W}+\delta_{V\cap W}=\delta_V+\delta_W-2|E(V,W)|.
\]
We can check that the last equality holds by proving that each edge $e\in E(\Gamma)$ is counted the same amount of times in each side. In the table below we enumerate all cases (up to symmetry). We write $z_1$ and $z_2$ for the (possibly coincident) vertices incident to $e$. 
\[
\begin{array}{l|c|c|c|c|c}
& \delta_{V\cup W} &\delta_{V\cap W}&\delta_{V}&\delta_W& |E(V,W)|\\
\hline
z_1,z_2\notin V\cup W&0&0&0&0&0\\
z_1\in V\setminus W, z_2\notin V\cup W&1&0&1&0&0\\
z_1,z_2\in V\setminus W &0&0&0&0&0\\
z_1\in V\setminus W, z_2\in W\setminus V&0&0&1&1&1\\
z_1\in V\cap W, z_2\notin V\cup W&1&1&1&1&0\\
z_1\in V\cap W, z_2\in V\setminus W&0&1&0&1&0\\
z_1,z_2\in V\cap W&0&0&0&0&0\\
\end{array}
\]
This concludes the proof.
\end{proof}

Let $\Gamma$ be a graph and $\mu$ a degree-$d$ polarization on $\Gamma$. We say that a degree-$d$ divisor $D$ on $\Gamma$ is $\mu$-\emph{semistable} if $\beta_D(V)\geq0$ for every $V\subset V(\Gamma)$. Given $v_0\in V(\Gamma)$, we say that a degree-$d$ divisor $D$ on $\Gamma$ is $(v_0,\mu)$-\emph{quasistable} if $\beta_D(V)\geq0$ for every $V\subsetneq V(\Gamma)$, with strict inequality if $v_0\in V$.
We say that a pseudo-divisor $(\E,D)$ is \emph{$\mu$-semistable} (respectively, $(v_0,\mu)$-\emph{quasistable}) if $D$ is $\mu^\E$-semistable (respectively, $(v_0,\mu^\E)$-quasistable) on $\Gamma^\E$. Clearly every $(v_0,\mu)$-quasistable pseudo-divisor is $\mu$-semistable.\par
Note that if $\Gamma$ is not connected, then the condition of quasistability can not be satisfied by any divisor on $\Gamma$.

\begin{Rem}
\label{rem:Vc} 
By Lemma \ref{lem:cap}, we have that $\beta_D(V)+\beta_D(V^c)=\delta_V$, hence $\beta_D(V)\geq 0$ if and only if $\beta_D(V^c)\leq \delta_{V^c}$. We deduce that a divisor $D$ on $\Gamma$ is $(v_0,\mu)$-quasistable if and only if $\beta_D(V)\leq\delta_V$ for every nonempty $V\subsetneq V(\Gamma)$, with strict inequality if $v_0\notin V$.  In the sequel, we will freely use both characterizations of quasistability.
\end{Rem}

The following is the combinatorial translation of a result due to Esteves about quasistable invertible sheaves on nodal curves.
\begin{Thm}
\label{thm:esteves}
Every divisor $D$ on $\Gamma$ is equivalent to a unique $(v_0,\mu)$-quasistable divisor.
\end{Thm}
\begin{proof}
This follows from \cite[Proposition 27 and Theorem 32]{Es01}. Alternatively, this is also a consequence of Theorem \ref{thm:quasistable} (which gives the analogue statement for tropical curves).
\end{proof}
\begin{Rem}
\label{rem:subdivision}
It is easy to see that if $D$ is a $(v_0,\mu^\E)$-quasistable divisor on $\Gamma^\E$ then $D(v)=0, -1$ for every exceptional vertex $v\in V(\Gamma^\E)$. Moreover if $\widehat{\Gamma}$ is a refinement of $\Gamma$ then $\mu$ induces a polarization $\widehat{\mu}$ in $\widehat{\Gamma}$ given by 
\[
\widehat{\mu}(v)=\begin{cases}
                           \begin{array}{ll}
													  0,&\text{ if $v$ is exceptional;}\\
														\mu(v),& \text{ if $v\in V(\Gamma)$},
														\end{array}
									\end{cases}
									\]
where we view $V(\Gamma)$ as a subset of $V(\widehat{\Gamma})$ via the injection $a\col V(\Gamma)\ra V(\wh{\Gamma})$ induced by the refinement (see Subsection \ref{subsec:graphs}). If $D$ is a $(v_0,\widehat{\mu})$-quasistable divisor on $\widehat{\Gamma}$ then for every edge $e\in E(\Gamma)$, we have that $D(v)=0$ for all but at most one exceptional vertex $v$ over $e$; if such a vertex $v$ over $e$ exists, then $D(v)=-1$. Hence, every $(v_0,\widehat{\mu})$-quasistable divisor $D$ of $\widehat{\Gamma}$ induces a $(v_0,\mu)$-quasistable pseudo-divisor $(\E',D')$ on $\Gamma$.
\end{Rem}

\begin{Rem}
 The fact that the degree of quasistable divisors are $0$ or $-1$ on exceptional vertices follows from the choice of which inequalities are strict in the definition of quasistability. We could have defined quasistability by asking that $-D$ is $(v_0,-\mu)$-quasistable. In this case quasistable divisors would have degree $0$ or $1$ on the exceptional vertices. All the constructions and results of the paper would hold in this alternative setting.
\end{Rem}

In the next result we study how quasistability behaves under any specialization and deletion of  nondisconnecting edges. This is an important tool to establish an interplay between quasistable pseudo-divisors on a graph and quasistable divisors on spanning connected  subgraphs (possibly changing the polarization). 

\begin{Prop}
\label{prop:spec}
Let $\Gamma$ be a graph, $v_0$ a vertex of $\Gamma$, and $\mu$ a degree-$d$ polarization on $\Gamma$. Assume that $\iota\col\Gamma\ra\Gamma'$ is a specialization of graphs. For every pseudo-divisor $(\E,D)$ on $\Gamma$ of  degree-$d$, the following properties hold:
\begin{enumerate}
\item[(i)] if $(\E,D)$ is $(v_0,\mu)$-quasistable then $\iota_*(\E,D)$ is an $(\iota(v_0),\iota_*(\mu))$-quasistable pseudo-divisor on $\Gamma'$;
\item[(ii)] $(\E,D)$ is $(v_0,\mu)$-quasistable if and only if $\E$ is nondisconnecting and $D_\E$ is a $(v_0,\mu_\E)$-quasistable divisor on $\Gamma_\E$.
\end{enumerate}
\end{Prop}
\begin{proof}
Let us start proving \emph{(i)}. Upon switching $\Gamma$ with $\Gamma^\E$, we can assume that $\E=\emptyset$. We need to show that for every $V'\subsetneq V(\Gamma')$ we have $\beta_{\iota_*(D)}(V')\geq0$, with strict inequality if $\iota(v_0)\in V'$. We know that
\[
\deg(\iota_*(D)|_{V'})=\deg(D|_{\iota^{-1}(V')}),\quad \iota_*(\mu)(V')=\mu(\iota^{-1}(V')),\quad\delta_{\Gamma',V'}= \delta_{\Gamma,\iota^{-1}(V')},
\]
which readily implies $\beta_{\iota_*(D)}(V)=\beta_D(\iota^{-1}(V))$. Moreover, $\iota(v_0)\in V'$ if and only if $v_0\in \iota^{-1}(V')$, hence the result follows.\par
Now we move to the proof of \emph{(ii)}. First we show that if $(\E,D)$ is $(v_0,\mu)$-quasistable then $\E$ is nondisconnecting. Assume, by contradiction, that $\E$ is disconnecting, and let $V$ be a subset of $V(\Gamma)$ such that the cut $E(V,V^c)$ is contained in $\E$. Note that
\[
\deg(D|_V)-\mu(V)+\deg(D|_{V^c})-\mu(V^c)=\delta_V,
\]
because $D(v)=-1$ for each exceptional vertex $v\in V(\Gamma^\E)$ (by the definition of pseudo-divisor). Then
\[
\left(\deg(D|_V)-\mu(V)-\frac{\delta_V}{2}\right)+\left(\deg(D|_{V^c})-\mu(V^c)-\frac{\delta_{V^c}}{2}\right)=0.
\]
However, since $(\E,D)$ is $(v_0,\mu)$-quasistable, we must have
\[
\deg(D|_V)-\mu(V)-\frac{\delta_V}{2}\leq 0 \quad \text{ and } \quad 
\deg(D|_{V^c})-\mu(V^c)-\frac{\delta_{V}}{2}\leq 0,
\]
hence the two inequalities are in fact equalities, which contradicts that one of them must be strict.\par
Now, we assume $\E$ nondisconnecting and we show the equivalence of the quasistability conditions. Fix $V\subset V(\Gamma)=V(\Gamma_\E)$. We have $\deg(D_\E|_V)=\deg(D|_V)$ and, if we let  $\kappa_\E(V)=\sum_{v\in V}\val_\E(v)$, then
\[
\mu_\E(V)=\mu(V)+\frac{1}{2}\kappa_\E(V)=\mu^\E(V)+\frac{1}{2}\kappa_\E(V).
\]
 Moreover we have $\delta_{\Gamma_\E,V}=\delta_{\Gamma^\E,V}-\kappa_\E(V)$. 
These equalities imply that
\begin{multline*}
\deg(D|_V)-\mu^\E(V)\pm\frac{\delta_{\Gamma^\E,V}}{2}=\\=\deg(D_\E|_V)-(\mu_\E(V)-\frac{1}{2}\kappa_\E(V))\pm\frac{\delta_{\Gamma_\E,V}+\kappa_\E(V)}{2}
\end{multline*}
and hence
\begin{equation}
\label{eq:DDE+}
\deg(D|_V)-\mu^\E(V)+\frac{\delta_{\Gamma^\E,V}}{2}=\deg(D_\E|_V)-\mu_\E(V)+\frac{\delta_{\Gamma_\E,V}}{2}+\kappa_\E(V)
\end{equation}
and
\begin{equation}
\label{eq:DDE-}
\deg(D|_V)-\mu^\E(V)-\frac{\delta_{\Gamma^\E,V}}{2}=\deg(D_\E|_V)-\mu_\E(V)-\frac{\delta_{\Gamma_\E,V}}{2}.
\end{equation}
If $(D,\E)$ is $(v_0,\mu^\E)$-quasistable, then the left hand side of Equation \eqref{eq:DDE-} is less or equal than $0$, and equality can only hold if $v_0\in V$. This means that $D_\E$ is $(v_0,\mu_\E)$-quasistable.\par
Now assume that $D_\E$ is $(v_0,\mu_\E)$-quasistable. Fix $V'\subset V(\Gamma^\E)$ and define $V=V'\cap V(\Gamma)$. This means that every vertex in $V'\setminus V$ is exceptional. Then
\[
\deg(D|_{V'})=\deg(D|_V)-|V'\setminus V|\quad\text{and}\quad \mu^\E(V')=\mu^\E(V).
\]
Combining with Equation \eqref{eq:DDE+}, we get
\begin{multline*}
\deg(D|_{V'})-\mu^\E(V')+\frac{\delta_{\Gamma^\E,V'}}{2}=\\
=\deg(D_\E|_V)-\mu_\E(V)+\frac{\delta_{\Gamma_\E,V}}{2}+\frac{\delta_{\Gamma^\E,V'}-\delta_{\Gamma^\E,V}}{2}-|V'\setminus V|+\kappa_\E(V).
\end{multline*}

Now, to prove that $(\E,D)$ is $(v_0,\mu)$-quasistable it is sufficient that
\[
\delta_{\Gamma^\E,V'}+2\kappa_\E(V)\geq\delta_{\Gamma^\E,V}+2|V'\setminus V|.
\]
We check that the last inequality holds by proving that the contribution of each edge $e\in E(\Gamma)$ to the left hand side is greater or equal than its contribution to the right hand side. Any $e\notin \E$ contributes the same in both sides. In the table below, we enumerate the remaining cases. We write $w_1$ and $w_2$ for the (possibly coincident) vertices incident to $e$, and  $w_e$ for the exceptional vertex over $e$.
\[
\begin{array}{l|c|c|c|c}
& \delta_{\Gamma^\E,V'} &\kappa_\E(V)&\delta_{\Gamma^\E,V}&|V'\setminus V|\\
\hline
w_1,w_2\in V, w_e\in V'&0&2&2&1\\
w_1\in V, w_2\notin V, w_e\in V'&1&1&1&1\\
w_1,w_2\notin V,w_e\in V'&2&0&0&1\\
w_1,w_2\in V, w_e\notin V'&2&2&2&0\\
w_1\in V,w_2\notin V, w_e\notin V'&1&1&1&0\\
w_1,w_2\notin V, w_e\notin V'&0&0&0&0
\end{array}
\]
This concludes the proof.
\end{proof}

We continue our analysis of quasistable divisors in the case when $\Gamma$ is tree-like.

\begin{Prop}
\label{prop:tree}
Let $\Gamma$ be a tree-like graph, $v_0$ a vertex of $\Gamma$, and $\mu$ a degree-$d$ polarization on $\Gamma$. Then there is exactly one $(v_0,\mu)$-quasistable divisor of degree-$d$ on $\Gamma$.
\end{Prop}
\begin{proof}
First, we note that a divisor $D$ on $\Gamma$ is $(v_0,\mu)$-quasistable if and only if $\iota_*(D)$ is $(\iota(v_0),\iota_*(\mu))$-quasistable, where $\iota$ is the contraction of all loops of $\Gamma$. In fact, loops do not contribute to any of the values $\deg(D|_V)$, $\mu(V)$ and $\delta_V$ for $V\subset V(\Gamma)$. So, we can safely assume that $\Gamma$ is a tree.\par
  We proceed by induction on the number of vertices. If $\Gamma$ has a only one vertex the result is clear. Assume that $\Gamma$ has at least two vertices and let $D_1$ and $D_2$ be $(v_0,\mu)$-quasistable divisors on $\Gamma$. Let $v$ be a vertex of $\Gamma$ of valence $1$ such that $v\neq v_0$.  Hence 
\[
-\frac{1}{2}\leq D_i(v)-\mu(V)< \frac{1}{2}, \text{ for $i=1,2$}
\]
which means
\[
D_i(v)= \left\lfloor\mu(V)+\frac{1}{2} \right \rfloor, \text{ for $i=1,2$}.
\]
Let $\iota\col \Gamma\to \Gamma'$ be the contraction of the only edge $e$ incident to $v$. By the induction hypothesis, there exists a unique $(\iota(v_0),\iota_*(\mu))$-quasistable  divisor on $\Gamma'$. In particular $\iota_*(D_1)=\iota_*(D_2)$, because, by Proposition \ref{prop:spec}, $\iota_*(D_1)$ and $\iota_*(D_2)$ are $(\iota(v_0),\iota_*(\mu))$-quasistable. This means that $D_1(w)=D_2(w)$ for every vertex $w$ not incident to $e$. Since $D_1(v)=D_2(v)$ and the degrees of $D_1$ and $D_2$ are equal, it follows that $D_1=D_2$.\par
We note that the proof above actually gives a algorithm to find a $(v_0,\mu)$-quasistable divisor on $\Gamma$ hence, in particular, it also provides the existence.
\end{proof}

\begin{Cor}
\label{cor:unique}
Let $\Gamma$ be a graph, $v_0$ a vertex of $\Gamma$, and $\mu$ a polarization on $\Gamma$. If $\E$ is the complement of a spanning tree of $\Gamma$, then there exists exactly one $(v_0,\mu^\E)$-quasistable divisor $D$ on $\Gamma^\E$ with $D(v')=-1$ for every $v'\in V(\Gamma^\E)\setminus V(\Gamma)$.
\end{Cor}
\begin{proof}
Just combine Propositions  \ref{prop:spec} and \ref{prop:tree}.
\end{proof}

Let $\Gamma$ be a graph, $v_0$ a vertex of $\Gamma$, and $\mu$ a polarization on $\Gamma$.
 The set $\Pos(\Gamma)$ of $(v_0,\mu)$-quasistable pseudo-divisors on $\Gamma$ is a poset where the partial order is given by  $(\E,D)\geq (\E',D')$ if $(\E,D)$ specializes to $(\E',D')$.\par

 If $\iota\col \Gamma\to \Gamma'$ is a graph specialization, by Proposition \ref{prop:spec} there is a natural map 
\begin{align}
\label{eq:i*} \iota_*\col   \mathcal{QD}_{v_0,\mu}(\Gamma)
\to&  \mathcal{QD}_{\iota(v_0),\iota_*(\mu)}(\Gamma')\\
 \nonumber (\E,D)\mapsto&\iota_*(\E,D).
\end{align}
This map is order preserving, hence it is continuous.\par
  Given a nondisconnecting subset $\E$ of $E(\Gamma)$, again by Proposition \ref{prop:spec} we have an induced injection 
\begin{align*}
\kappa_*\col \mathcal{QD}_{v_0,\mu_\E}(\Gamma_\E)
\to& \mathcal{QD}_{v_0,\mu}(\Gamma)\\
  (\E',D')\mapsto&(\E'\cup \E, D)
\end{align*}
where $D(v)=D'(v)$ for every $v\in V(\Gamma)$. The map $\kappa$ is order preserving, hence it is continuous. Moreover, $\kappa$ is an open injection, since Proposition \ref{prop:spec} implies that $\Im(\kappa)=\{(\E',D)\in \mathcal{QD}_{v_0,\mu}(\Gamma) ;
\E\subset \E'\}$.\par
		Given a specialization of graphs $\iota\col\Gamma\ra\Gamma'$ and inclusions $\E'\subset E(\Gamma')\subset E(\Gamma)$, we have an induced specialization $\iota_{\E'}\col  \Gamma_{\E'}\ra \Gamma'_{\E'}$ which makes the following diagram a commutative  diagram of topological spaces:
\[
\begin{CD}
\mathcal{QD}_{v_0,\mu_{\E'}}(\Gamma_{\E'}) 
@>\iota_{\E'*}>> 
\mathcal{QD}_{\iota(v_0),\iota_*(\mu_{\E'})}(\Gamma'_{\E'})
\\
@V\kappa_* VV @V\kappa'_*VV\\
\mathcal{QD}_{v_0,\mu}(\Gamma)@>\iota_*>> \mathcal{QD}_{\iota(v_0),\iota_*(\mu)}(\Gamma')
\end{CD}
\]

\begin{Exa}
\label{exa:poset}
Let $\Gamma$ be the graph with $2$ vertices $v_0$ and $v_1$ and $3$ edges $e_1,e_2,e_3$ connecting them. Let $\mu$ be the degree-$0$ polarization given by $\mu(v)=0$ for every $v\in V(\Gamma)$. The poset $ \mathcal{QD}_{v_0,\mu}(\Gamma)$ 
of pseudo-divisors of $\Gamma$ is as depicted in Figure \ref{fig:poset} (up to permutation of the edges).
\end{Exa}
\begin{figure}[hb]
\begin{tikzpicture}[scale=2.25]
\begin{scope}[shift={(0,0)}]
\draw (0,0) to [out=30, in=150] (1,0);
\draw (0,0) to (1,0);
\draw (0,0) to [out=-30, in=-150] (1,0);
\draw[fill] (0,0) circle [radius=0.02];
\draw[fill] (1,0) circle [radius=0.02];
\draw[fill] (0.5,0.144) circle [radius=0.02];
\draw[fill] (0.5,0) circle [radius=0.02];
\node at (-0.07,0) {1};
\node at (1.05,0) {1};
\node at (0.55,0.23) {-1};
\node at (0.55,0.06) {-1};
\end{scope}
\draw[->] (0.3,-0.15) to (0, -0.5);
\draw[->] (0.7,-0.15) to (1, -0.5);
\begin{scope}[shift={(-0.75,-0.8)}]
\draw (0,0) to [out=30, in=150] (1,0);
\draw (0,0) to (1,0);
\draw (0,0) to [out=-30, in=-150] (1,0);
\draw[fill] (0,0) circle [radius=0.02];
\draw[fill] (1,0) circle [radius=0.02];
\draw[fill] (0.5,0.144) circle [radius=0.02];
\node at (-0.07,0) {1};
\node at (1.07,0) {0};
\node at (0.5,0.23) {-1};
\draw[->] (0.3,-0.15) to (0, -0.5);
\draw[->] (0.7,-0.15) to (1, -0.5);
\end{scope}
\begin{scope}[shift={(+0.75,-0.8)}]
\draw (0,0) to [out=30, in=150] (1,0);
\draw (0,0) to (1,0);
\draw (0,0) to [out=-30, in=-150] (1,0);
\draw[fill] (0,0) circle [radius=0.02];
\draw[fill] (1,0) circle [radius=0.02];
\draw[fill] (0.5,0.144) circle [radius=0.02];
\node at (-0.07,0) {0};
\node at (1.05,0) {1};
\node at (0.5,0.23) {-1};
\draw[->] (0.3,-0.15) to (0, -0.5);
\draw[->] (0.7,-0.15) to (1, -0.5);
\end{scope}
\begin{scope}[shift={(-1.5,-1.6)}]
\draw (0,0) to [out=30, in=150] (1,0);
\draw (0,0) to (1,0);
\draw (0,0) to [out=-30, in=-150] (1,0);
\draw[fill] (0,0) circle [radius=0.02];
\draw[fill] (1,0) circle [radius=0.02];
\node at (-0.07,0) {1};
\node at (1.1,0) {-1};
\end{scope}
\begin{scope}[shift={(0,-1.6)}]
\draw (0,0) to [out=30, in=150] (1,0);
\draw (0,0) to (1,0);
\draw (0,0) to [out=-30, in=-150] (1,0);
\draw[fill] (0,0) circle [radius=0.02];
\draw[fill] (1,0) circle [radius=0.02];
\node at (-0.07,0) {0};
\node at (1.07,0) {0};
\end{scope}
\begin{scope}[shift={(1.5,-1.6)}]
\draw (0,0) to [out=30, in=150] (1,0);
\draw (0,0) to (1,0);
\draw (0,0) to [out=-30, in=-150] (1,0);
\draw[fill] (0,0) circle [radius=0.02];
\draw[fill] (1,0) circle [radius=0.02];
\node at (-0.07,0) {-1};
\node at (1.05,0) {1};
\end{scope}
\end{tikzpicture}
\caption{The poset of pseudo-divisors.}
\label{fig:poset}
\end{figure}
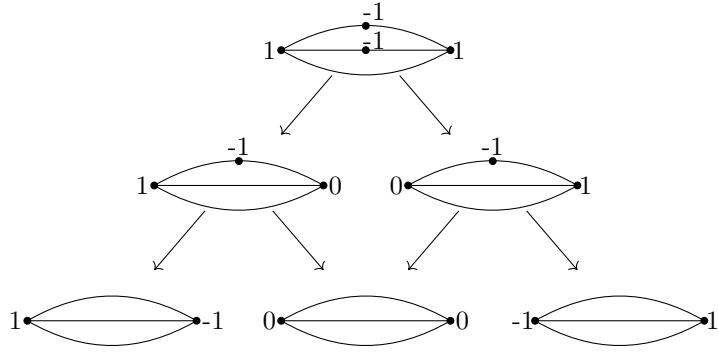

Now, we prove a series of topological properties of the poset  $\Pos(\Gamma)$ (in Propositions \ref{prop:chain} and \ref{prop:cod1}) and of the map $\iota_*$ defined in Equation \eqref{eq:i*} (in Proposition \ref{prop:jclosed} and Corollary \ref{cor:sep}).

\begin{Prop}
\label{prop:chain}
Let $\Gamma$ be a graph, $v_0$ a vertex of $\Gamma$, and $\mu$ a polarization on $\Gamma$. The poset $\Pos(\Gamma)$ is ranked of length $b_1(\Gamma)$.
\end{Prop}
\begin{proof}
Set $g=b_1(\Gamma)$. It is sufficient to prove that if $(\E,D)$ is a $(v_0,\mu)$-quasistable pseudo-divisor on $\Gamma$ such that $|\E|<g$, then we can find a $(v_0,\mu)$-quasistable pseudo-divisor $(\E',D')$ on $\Gamma$ with $\E\subset \E'$ and $|\E'|=|\E|+1$, and such that $(\E',D')$ specializes to $(\E,D)$. By Corollary \ref{cor:unique}, we can assume that $g\ge 1$. By Proposition \ref{prop:spec}, we have that $D_\E$ is a $(v_0,\mu_\E)$-quasistable divisor on $\Gamma_\E$, thus we can reduce to the case in which $\E=\emptyset$. In what follows, given $\E'\subset E(\Gamma)$, we will denote by $v_e$ the exceptional vertex of $\Gamma^{\E'}$ over $e$, for every $e\in \E'$.\par
   If $\Gamma$ has a loop $e$ incident to a vertex $v_1$, we can take the pseudo-divisor $(\E',D')$, where $\E'=\{e\}$ and $D'(v_1)=D(v_1)+1$, $D'(v_e)=-1$ and $D'(v)=D(v)$ for every $v\in V(\Gamma^{\E'})\setminus\{v_1,v_e\}$. It is clear that $(\E',D')$ is $(v_0,\mu)$-quasistable and $(\E',D')$ specializes to $(\emptyset,D)$. Therefore, we can also assume that $\Gamma$ has no loops.\par
	Fix an edge $e$ incident to $v_0$ and let $v_1\in V(\Gamma)$ be the other vertex incident to $e$. For $i=0,1$, we define the divisor $D_i$ on $\Gamma^{\{e\}}$ as 
\[
D_i(v):=\begin{cases}
\begin{array}{ll}
-1,&\text{if $v=v_e$};\\
D(v)+1,&\text{if $v=v_{1-i}$};\\
D(v),&\text{if $v\neq v_{1-i},v_e$.}
\end{array}
\end{cases}
\]
 If at least one between $D_0$ and $D_1$ is $(v_0,\mu^{\{e\}})$-quasistable, then we are done, since $(\{e\},D_i)$ specializes to $(\emptyset,D)$ for $i=0,1$. Otherwise, there are subsets $V_0', V_1'\subsetneq V(\Gamma^{\{e\}})$ such that, for $i=0,1$, we have 
\begin{equation}
\label{eq:vi}
\beta_{D_i}(V_i')\leq0,\text{ with strict inequality if $v_0\notin V_i'$}.
\end{equation}
Fix $i\in\{0,1\}$ and set $V_i:=V_i'\setminus\{v_e\}\subset V(\Gamma)$. We have 
\[
\beta_{D_i}(V_i')-\beta_D(V_i)=\deg(D_i|_{V_i'})-\deg(D|_{V_i})+\frac{\delta_{\Gamma^{\{e\}},V_i'}-\delta_{\Gamma,V_i}}{2}.
\]
By the table below,
\[
\begin{array}{l|c|c}
\{v_e,v_0,v_1\}\cap V_i'& \deg(D_i|_{V_i'})-\deg(D|_{V_i}) & \delta_{\Gamma^{\{e\}},V_i'} -\delta_{\Gamma,V_i}\\
\hline
\{v_e,v_0,v_1\}&\phantom{-}0&0\\
\{v_e,v_{1-i}\}&\phantom{-}0&0\\
\{v_e,v_{i}\}&-1&0\\
\{v_0,v_1\}&\phantom{-}1&2\\
\{v_e\}&-1&2\\
\{v_i\}&\phantom{-}0&0\\
\{v_{1-i}\}&\phantom{-}1&0\\
\emptyset&\phantom{-}0&0
\end{array}
\]
we deduce that, if $\{v_e,v_0,v_1\}\cap V_i'\ne\{v_e,v_i\}$, then $\beta_{D_i}(V_i')\geq\beta_D(V_i)$, and this contradicts Equation \eqref{eq:vi}. It follows that $\{v_e,v_0,v_1\}\cap V_i'=\{v_e,v_i\}$. This implies that $\beta_{D_i}(V_i')=\beta_D(V_i)-1$, which in turn implies that 
\begin{equation}
\label{eq:vi1}
\beta_D(V_i)\leq1,\text{ with strict inequality if $i=1$}.
\end{equation}\par

  If $V(\Gamma)=\{v_0, v_1\}$, then we have $V_i=\{v_i\}$, and hence $\beta_D(\{v_0\})+\beta_D(\{v_1\})<2$. However, by Lemma \ref{lem:cap}, we have that $\beta_D(\{v_0\})+\beta_D(\{v_1\})$ is the number of edges between $v_0$ and $v_1$ which is at least 2, since $g\ge 1$. This gives rise to a contradiction, so we are done when $\Gamma$ has exactly two vertices.\par

	 We now proceed by induction on the number of vertices of $\Gamma$. Let $\iota\col\Gamma\to\Gamma'$ be the contraction of all edges in $E(v_0,v_1)$. By the induction hypothesis, there are an edge $e'\in E(\Gamma')\subset E(\Gamma)$ and a $(\iota(v_0),\iota_*(\mu))$-quasistable pseudo-divisor $(\{e'\},\widehat{D})$ on $\Gamma'$ such that $(\{e'\},\widehat{D})$ specializes to $(\emptyset,\iota_*(D))$. This means that if $w_0$ and $w_1$ are the vertices incident to $e'$, then $\widehat{D}(w_0)=\iota_*(D)(w_0)+1$, $\widehat{D}(v_{e'})=-1$ and $\widehat{D}(w)=\iota_*(D)(w)$ for every $w\in V(\Gamma')\setminus\{w_0,v_{e'}\}$. For $i=0,1$, if $\iota^{-1}(w_i)$ is a single vertex, we abuse notation and write $w_i$ for such vertex. Otherwise, we have that $\iota^{-1}(w_i)=\{v_0,v_1\}$, and again we abuse notation and write $w_i$ for the vertex in $\{v_0,v_1\}$ attached to $e'$. We define the pseudo-divisor $(\{e'\},D')$ on $\Gamma$ as 
\[
D'(v):=\begin{cases}
\begin{array}{ll}
-1&\text{if $v=v_{e'}$;}\\
D(v)+1&\text{if $v=w_0$;}\\
D(v)&\text{if $v\neq w_0,v_{e'}$.}
\end{array}
\end{cases}
\]
Then $\iota_*(\{e'\},D')=(\{e'\},\widehat{D})$ and there is a specialization of $(\{e'\},D') $ to $(\emptyset,D)$. All that is left to prove is that $(\{e'\},D')$ is $(v_0,\mu)$-quasistable. By contradiction, assume that there is a subset $V'\subsetneq V(\Gamma^{\{e'\}})$ such that 
\begin{equation}
\label{eq:v'}
\beta_{D'}(V')\leq 0,\text{ with strict inequality if $v_0\notin V'$}. 
\end{equation}
By an argument similar to the one using the above table, we have $\{w_0,w_1, v_{e'}\}\cap V'=\{w_1,v_{e'}\}$ and hence $\beta_{D'}(V')=\beta_D(V)-1$, where $V=V'\setminus\{v_{e'}\}\subset V(\Gamma)$, Again, this implies that 
\begin{equation}
\label{eq:v'1}
\beta_{D}(V)\leq 1,\text{ with strict inequality if $v_0\notin V$}. 
\end{equation} 
Moreover, if $|\{v_0,v_1\}\cap V'|\neq 1$, then $\beta_{D'}(V')=\beta_{\widehat{D}}(\iota^{\{e'\}}(V'))$, because $\iota_*(\{e'\},D')=(\{e'\},\wh{D})$
and $V'$ is the inverse image of $\iota^{\{e'\}}(V')$ via $\iota^{\{e'\}}$. This contradicts Equation \eqref{eq:v'} and the fact that $\wh D$ is $(\iota(v_0),\iota_*(\mu))$-quasistable. Therefore there exists $i\in\{0,1\}$ such that $v_{1-i}\in V$ and $v_{i}\notin V$.\par
  By Lemma \ref{lem:cap} we have
\begin{equation}
\label{eq:2}
\beta_D(V\cup V_i)+\beta_D(V\cap V_i)+|E(V,V_i)|=\beta_D(V)+\beta_D(V_i)<2,
\end{equation}
where the inequality comes from Equations \eqref{eq:vi1} and \eqref{eq:v'1} (note that either $v_0\notin V$ or $i=1$, hence the inequality is indeed strict).\par
Note that $\beta_D(V\cup V_i)\geq0$ and $\beta_D(V\cap V_i)\geq0$, since $(\emptyset,D)$ is $(v_0,\mu)$-quasistable. We have three cases to consider.

 In the first case, $w_0\in V_i$ and $w_1\notin V_i$, hence $|E(V,V_i)|\geq2$, because the set $E(V,V_i)$ contains $e$ and $e'$. This contradicts Equation \eqref{eq:2}.\par 

 In the second case, $w_0\notin V_i$, hence $\beta_{D'}(V'\cup V_i)=\beta_D(V\cup V_i)-1$, because $\mu(V'\cup V_i)=\mu(V\cup V_i)$, $\deg(D'|_{V'\cup V_i})=\deg(D|_{V\cup V_i})+D'(v_{e'})$ and $\delta_{\Gamma^{\{e'\}},V'\cup V_i}=\delta_{\Gamma, V\cup V_i}$ (recall that $w_1\in V'$). However, since $v_0,v_1\in V'\cup V_i$, we have 
\[
\beta_{D'}(V'\cup V_i)=\beta_{\widehat{D}}(\iota^{\{e'\}}(V'\cup V_i))\geq 0,
\]
hence $\beta_D(V\cup V_i)\geq1$. But the set $E(V,V_i)$ contains $e$, hence $|E(V,V_i)|\geq 1$, and this contradicts Equation \eqref{eq:2}.\par

  In the third case, $w_0,w_1\in V_i$. The argument is similar to the one used in the second case. Indeed, we have $|E(V,V_i)|\geq 1$ and $\beta_{D'}(V'\cap (V_i\cup\{v_{e'}\}))=\beta_D(V\cap V_i)-1$. Since $v_0,v_1\notin V'\cap (V_i\cup\{v_{e'}\})$, we also have  
	\[
\beta_{D'}(V'\cap (V_i\cup\{v_{e'}\})=\beta_{\widehat{D}}(\iota^{\{e'\}}(V'\cap (V_i\cup\{v_{e'}\}))\geq 0,
\]
hence $\beta_D(V\cap V_i)\geq1$, contradicting again  Equation \eqref{eq:2}.
\end{proof}

\begin{Prop}
\label{prop:jclosed}
Let $\Gamma$ be a graph, $v_0$ a vertex of $\Gamma$, and $\mu$ be a polarization on $\Gamma$. If $\iota\col\Gamma\ra\Gamma'$ is  
a specialization of graphs, then the induced map $\iota_*\col \mathcal{QD}_{v_0,\mu}(\Gamma)\to \mathcal{QD}_{\iota(v_0),\iota_*(\mu)}(\Gamma')$ is closed and surjective.
\end{Prop}
\begin{proof}
We can assume that $\iota$ is the contraction of a single edge $e\in E(\Gamma)$, because a composition of closed and surjective maps is also closed and surjective.\par
 We begin proving that $\iota_*$ is closed. It is enough to prove that if $(\E_1',D_1')\geq (\E_2',D_2')$ in $\mathcal {QD}_{\iota(v_0),\iota_*(\mu)}(\Gamma')$ and $(\E_1',D_1')=\iota_*(\E_1,D_1)$ for some $(\E_1,D_1)$ in $\mathcal{QD}_{v_0,\mu}(\Gamma)$ then there is $(\E_2,D_2)\in \mathcal{QD}_{v_0,\mu}(\Gamma)$ such that $(\E_1,D_1)\geq (\E_2,D_2)$ and $\iota_*(\E_2,D_2)=(\E_2',D_2')$. Since $(\E_1',D_1')\geq (\E_2',D_2')$, there are an inclusion $\E_2'\subset \E_1'$ and a compatible specialization $\Gamma'^{\E_1'}\ra\Gamma'^{\E_2'}$ that takes $D_1'$ to $D_2'$. Since $\iota_*(\E_1,D_1)=(\E_1',D_1')$, we have that $\E_1'=\E_1\cap E(\Gamma')$ and the induced specialization $\Gamma^{\E_1}\ra\Gamma'^{\E_1'}$ takes $D_1$ to $D_1'$. We set $\E_2:=\E_1\setminus E(\Gamma')$. Then we have a commutative diagram of specializations
\[
\begin{CD}
\Gamma^{\E_1}@>>> \Gamma^{\E_2}\\
@VVV @VVV\\
\Gamma'^{\E_1'}@>>> \Gamma'^{\E_2'}
\end{CD}
\]
and the pushforward of $D_1$ to $\Gamma'^{\E_2'}$ must be $D_2'$. 
Define $D_2$ as the pushforward of $D_1$ to $\Gamma^{\E_2}$. Since the diagram is commutative, we have $\iota_*(\E_2,D_2)=(\E'_2,D'_2)$ which concludes the proof that $\iota_*$ is closed.

 We now prove the surjectivity of $\iota_*$. Since $\iota_*$ is closed, it is enough to prove that $\iota_*^{-1}(\E',D')$ is nonempty for every maximal element $(\E',D')$ of $\mathcal{QD}_{\iota(v_0),\iota_*(\mu)}(\Gamma')$. By Proposition \ref{prop:chain}, a pseudo-divisor $(\E',D')$ is maximal if and only if $\E'$ is the complement of a spanning tree.  
  So let $\E'$ be the complement of a spanning tree on $\Gamma'$. Then, by Corollary \ref{cor:unique}, there exists a unique $(\iota(v_0),\iota_*(\mu))$-quasistable pseudo-divisor $(\E',D')$ on $\Gamma'$. Since $\E'$ is nondisconnecting on $\Gamma$, then $\Gamma_{\E'}$ is connected, which means that there exists a $(v_0,\mu_{\E'})$-quasistable divisor $D$ on $\Gamma_{\E'}$ (see Theorem \ref{thm:esteves}). We can use Proposition \ref{prop:spec} to get a $(v_0,\mu)$-quasistable divisor $(\E',D)$ on $\Gamma$. Since $\iota_*(\E',D)=(\E',\iota_{\E'*}(D))$ is $(\iota(v_0),\iota_*(\mu))$-quasistable by Proposition \ref{prop:spec}, the uniqueness of $(\E',D')$ implies that $\iota_*(\E',D)=(\E',D')$, and we are done.
\end{proof}
\begin{Cor}
\label{cor:sep}
Preserve the hypothesis of Proposition \ref{prop:jclosed}.
If $\iota$ is a contraction of separating edges, then $\iota_*\col \mathcal{QD}_{v_0,\mu}(\Gamma)\to \mathcal{QD}_{\iota(v_0),\iota_*(\mu)}(\Gamma')$ is a homeomorphism.
\end{Cor}
\begin{proof}
 By Proposition \ref{prop:jclosed}, it is enough to prove that $\iota_*$ is injective. 
We can assume that $\iota$ is the contraction of a single edge $e\in E(\Gamma)$. Let $v_1$ and $v_2$ the vertices incident to $e$. Let $\Gamma_1$ be the connected component of $\Gamma_{\{e\}}$ containing $v_1$ and assume that $v_0\in\Gamma_1$. 
 Let $(\E_1,D_1)$ and $(\E_2,D_2)$ be $(v_0,\mu)$-quasistable divisors of $\Gamma$ such that $\iota_*(\E_1,D_1)=\iota_*(\E_2,D_2)$.
 We have $e\notin \E_1,\E_2$, since $e$ is a separating edge (see Proposition \ref{prop:spec}). It follows that $\E_1=\E_2=:\E$. Therefore, $D_1$ and $D_2$ are divisor on $\Gamma^\E$, so we can assume that $\E=\emptyset$ (upon changing $\Gamma,\Gamma'$ with $\Gamma^\E,\Gamma'^\E$).\par 
  Since $\iota_*(D_1)=\iota_*(D_2)$, we have  $D_1(v)=D_2(v)$ for every $v\notin\{v_1,v_2\}$. For $i\in\{1,2\}$ we have $0<\beta_{D_i}(V(\Gamma_1))\leq1$, which means that we have 
\[
\mu(V(\Gamma_1))-\frac{1}{2}<\deg(D_i|_{V(\Gamma_1)})\leq\mu(V(\Gamma_1))+\frac{1}{2}.
\]
This implies that $\deg(D_1|_{V(\Gamma_1)})=\deg(D_2|_{V(\Gamma_1)})$, because both are integers. We deduce that $D_1(v_1)=D_2(v_1)$, and hence $D_1=D_2$.
\end{proof}

\begin{Prop}
\label{prop:cod1}
Let $\Gamma$ be a graph, $v_0$ be a vertex of $\Gamma$, and $\mu$ be a polarization on $\Gamma$ of degree $d$. Then the poset $\mathcal{QD}_{v_0,\mu}(\Gamma)$ is connected in codimension 1.
\end{Prop}

\begin{proof}
Set $g:=b_1(\Gamma)$. We start proving the result for graphs with $g=1$. In this case the statement is equivalent to prove that $\mathcal{QD}_{v_0,\mu}(\Gamma)$ is connected. Since $g=1$, by Corollary \ref{cor:sep} we can reduce to the case in which $\Gamma$ is a cycle.  \par
We proceed by induction on the number of edges. The statement is clear if there is a single edge (which is a loop). Assume that $\Gamma$ has at least two edges and fix an edge $e\in E(\Gamma)$. By Corollary \ref{cor:unique}, there exists exactly one $(v_0,\{e\})$-quasistable pseudo-divisor $(\{e\},D)$ of degree-$d$ on $\Gamma$. Since $e$ is not a loop, there are exactly two divisors $D_1$ and $D_2$ on $\Gamma$ such that $(\{e\},D)\ra (\emptyset,D_i)$ for $i=1,2$.
	Let $\iota\col\Gamma\to\Gamma/\{e\}$ be the contraction of the edge $e$. Then 
\[
(\emptyset,D'):=\iota_*(\{e\},D)=\iota_*(\emptyset, D_1)=\iota_*(\emptyset, D_2).
\]
We claim that the set of $(v_0,\mu)$-quasistable pseudo-divisors on $\Gamma$ specializing to the pseudo-divisor $(\emptyset,D')$ via $\iota$ is precisely $\{(\{e\},D),(\emptyset,D_1),(\emptyset,D_2)\}$. Indeed, the map $\iota_*\col \mathcal{QD}_{v_0,\mu}(\Gamma)\to \mathcal{QD}_{\iota(v_0),\iota_*(\mu)}(\Gamma/\{e\})$ is surjective by Proposition \ref{prop:jclosed}. However 
\[
|\mathcal{QD}_{v_0,\mu}(\Gamma)|=2|E(\Gamma)|\quad\text{and}\quad| \mathcal{QD}_{\iota(v_0),\iota_*(\mu)}(\Gamma/\{e\})|=2|E(\Gamma')|,
\]
and hence 
\[
|\mathcal{QD}_{v_0,\mu}(\Gamma)|=|\mathcal{QD}_{\iota(v_0),\iota_*(\mu)}(\Gamma/\{e\})|+2.
\]
Since $|\iota_*^{-1}(\emptyset, D')|\geq3$,  the equality must hold, from which the claim follows.\par
  If we consider $\mathcal{QD}_{v_0,\mu}(\Gamma)$ as a graph whose edges are its minimal elements and whose vertices are its remaining elements (which are maximal), then $\iota_*$ is simply the contraction of the two edges corresponding to $(\emptyset,D_1)$ and $(\emptyset,D_2)$, and meeting each other at the vertex corresponding to $(\{e\},D)$. Since, by the induction hypothesis,  $\mathcal{QD}_{\iota(v_0),\iota_*(\mu)}(\Gamma/\{e\})$ is connected, we have that $\mathcal{QD}_{v_0,\mu}(\Gamma)$ is connected as well.\par
	Now we prove the general case. By Lemma \ref{lem:tree}, it is sufficient to prove that if $\E_1$ and $\E_2$ are nondisconnecting subsets of $E(\Gamma)$ of cardinality $g$ such that $|\E_1\cap\E_2|=g-1$, then there is a path in codimension 1 in $\mathcal{QD}_{v_0,\mu}(\Gamma)$  connecting any two pseudo-divisors $(\E_1,D_1)$ and $(\E_2,D_2)$. Let $\Gamma':=\Gamma_{\E_1\cap \E_2}$. Note that $g_{\Gamma'}=1$. We have a open injection $\mathcal{QD}_{v_0,\mu_{\E_1\cap\E_2}}(\Gamma')\to \mathcal{QD}_{v_0,\mu}(\Gamma)$. Then it is sufficient to prove that there exists a path from $(\E_1\setminus \E_2,D_1)$ to $(\E_2\setminus \E_1,D_2)$ in $\mathcal{QD}_{v_0,\mu_{\E_1\cap\E_2}}(\Gamma')$. The existence of this path follows from the case $g_{\Gamma'}=1$ already proved.
	\end{proof}

Let $d$ be an integer. 
We say that $\mu$ is a \emph{polarization (of degree $d$)}  on $\Graph_{g,1}$ if $\mu$ is a collection of polarizations $\mu_\Gamma$ of degree $d$ for every genus-$g$ weighted stable graph with $1$ leg $\Gamma$, such that $\mu_{\Gamma'}=\iota_*(\mu_\Gamma)$ for every specialization $\iota\col\Gamma\to\Gamma'$. In this case, we call $\mu$ a \emph{universal genus-$g$ polarization (of degree $d$)}. This polarization extends to every genus-$g$ semistable graph, since every genus-$g$ semistable graph is a subdivision of a stable graph. 
\begin{Exa}
\label{exa:pol}
Let $d$ be an integer. We give two examples of universal genus-$g$ polarizations of degree $d$. We have the \emph{canonical polarization}, given by
\[
\mu_\Gamma(v)=\frac{d(2w(v)-2+\val(v))}{2g-2}.
\]
We also have a polarization concentrated at the marked vertex $v_0$ of $\Gamma$, given by
\[
\mu_\Gamma(v)=\begin{cases}
               \begin{array}{ll}
							0,&\text{ if $v\neq v_0$;}\\
							d,&\text{ if $v=v_0$}.
							\end{array}
							\end{cases}
\]	
Note that any linear combination of the above universal polarizations gives rise to a  universal genus-$g$ polarization of degree $d$. These are the only universal polarizations one can consider on $\Graph_{g,1}$: this is proved in \cite[Section 5]{KP}.					
\end{Exa}

  If $\mu$ is a universal genus-$g$ polarization, we define the category $\mathbf{QD}_{\mu,g}$ whose objects are triples $(\Gamma,\E,D)$ where $\Gamma$ is a genus-$g$ stable  weighted graph with $1$ leg and $(\E,D)$ is a $(v_0,\mu_\Gamma)$-quasistable pseudo-divisor of $\Gamma$, and whose morphisms are given by the specializations. We define the poset $\mathcal{QD}_{\mu,g}$ associated to $\mathbf{QD}_{\mu,g}$ as
\[
\mathcal{QD}_{\mu,g}:=\{(\Gamma,\E,D);\;(\Gamma,\E, D)\in \mathbf{QD}_{\mu,g}\}/\sim,
\]
where $(\Gamma,\E,D)\sim (\Gamma',\E',D')$ if there exists an isomorphism $\iota\col\Gamma\ra\Gamma'$ such that $(\E',D')=\iota_*(\E,D)$, and the ordering is given by specializations.

We end this section studying some topological properties of $\mathcal{QD}_{\mu,g}$.

\begin{Thm}
\label{thm:maingraph}
Let $\mu$ be a universal genus-$g$ polarization of degree $d$. Then the poset $\mathcal{QD}_{\mu,g}$ has pure dimension $4g-2$ and is connected in codimension $1$. Moreover the natural forgetful map $f\col \mathcal{QD}_{\mu,g}\to \grap_{g,1}$ is continuous and for every genus-$g$ weighted graph $\Gamma$ with 1 leg, we have 
\[
f^{-1}([\Gamma])=\mathcal{QD}_{v_0,\mu_\Gamma}(\Gamma)/\Aut(\Gamma).
\]
\end{Thm}
\begin{proof}
We prove that $\mathcal{QD}_{\mu,g}$ has pure dimension $4g-2$, which is equivalent to the fact that every maximal element of $\mathcal{QD}_{\mu,g}$ is of the form $(\Gamma,\E,D)$ where $\Gamma$ is a $3$-regular graph with $1$ leg and with zero weight function, and where $|\E|=g$. (Recall that a $3$-regular graph with $1$ leg is a graph such that every vertex $v$ satisfies $\val(v)+\l(v)=3$.)  If $\Gamma$ is not a $3$-regular graph with $1$ leg and with zero weight function, then there exists a specialization $\Gamma'\ra\Gamma$ with $\Gamma'$ a $3$-regular graph with $1$ leg and with zero weight function because $M_{g,1}^\trop$ is pure of dimension $3g-2$. By Proposition \ref{prop:jclosed}, there is a specialization $(\Gamma',\E',D')\ra (\Gamma,\E,D)$ of $(v_0,\mu)$-quasistable pseudo-divisors. If $|\E|<g$, by Proposition \ref{prop:chain} there is a specialization $(\Gamma, \E',D')\ra (\Gamma,\E,D)$ with $|\E'|=g$.\par
 We now prove that $\mathcal{QD}_{\mu,g}$ is connected in codimension $1$. Let $(\Gamma,\E,D)$ and $(\widehat{\Gamma},\widehat{\E},\widehat{D})$ be two maximal elements of $\mathcal{QD}_{\mu,g}$. By  \cite[Theorem 3.2.5]{BMV} and \cite[Fact 4.12]{Caporaso}, there are two sequences: the first is a sequence $\Gamma=\Gamma_0,\Gamma_1,\ldots ,\Gamma_n=\widehat{\Gamma}$ of $3$-regular graphs with $1$ leg and the second is a sequence $\Gamma'_1 , \Gamma'_2 , \ldots , \Gamma'_n$ of codimension-$1$ graphs with $1$ leg; These sequences are endowed with specializations $\iota_k\col \Gamma_k\ra\Gamma'_{k+1}$ for $k=0,\dots,n-1$, and $\iota'_k\col\Gamma_k\ra\Gamma'_k$ for $k=1,\dots,n$, each one of which is the contraction of precisely one edge which is not a loop. This implies that the first Betti number of $\Gamma'_k$ is $b_1(\Gamma'_k)=g$ for every $k=1,\dots,n$; choose a $(v_0,\mu_{\Gamma'_k})$-quasistable pseudo-divisor $(\E'_k,D_k)$ on $\Gamma'_k$ with $|\E'_k|=g$ (this can be done by Corollary \ref{cor:unique}). By Proposition \ref{prop:jclosed} there exist pseudo-divisors $(\E_k,D_k)$ and $(\widehat{\E}_k,\widehat{D}_k)$ on $\Gamma_k$ for every $k=0,\ldots,n$ such that 
\[
\iota_{k,*}(\widehat{\E}_k,\widehat{D}_k)=(\E_{k+1}',D_{k+1}')\text{ for }k=0,\ldots,n-1,\text{ and }(\widehat{\E}_n,\widehat{D}_n)=(\widehat{\E},\widehat{D}),
\]
and such that
\[
(\E_0,D_0)=(\E,D),\text{ and } \iota'_{k,*}(\E_k,D_k)=(\E_k',D_k')\text{ for }k=1,\ldots,n.
\]
So there is a path in codimension $1$ from $(\Gamma_k,\widehat{\E}_k,\widehat{D}_k)$ to $(\Gamma_{k+1},\E_{k+1},\widehat{D}_{k+1})$ for every $k=0,\dots,n-1$. However, Proposition \ref{prop:cod1} shows that $(\Gamma_k,\E_k,D_k)$ and $(\Gamma_k,\widehat{\E}_k,\widehat{D}_k)$ are connected in codimension $1$ for every $k=0,\dots,n$. This finishes the proof that $\mathcal{QD}_{\mu,g}$ is connected in codimension $1$.\par
  The forgetful map $f$ is clearly order preserving. Let us prove that $f^{-1}([\Gamma])=\mathcal{QD}_{v_0,\mu}(\Gamma)/\Aut(\Gamma)$. There is a natural order-preserving map $h\col \mathcal{QD}_{v_0,\mu_\Gamma}(\Gamma)\to \mathcal{QD}_{\mu,g}$, and we have that $f^{-1}([\Gamma])=\Im(h)$. Moreover, $h(\E,D)=h(\E',D')$ if and only if there exists an automorphism $\iota\col\Gamma\ra\Gamma$ such that $\iota_*(\E,D)=(\E',D')$. This means that $\Im(h)=\mathcal{QD}_{v_0,\mu_\Gamma}(\Gamma)/\Aut(\Gamma)$.
\end{proof}

\section{The universal tropical Jacobian}\label{sec:univtropJ}

In this section we will extend the results of Section \ref{sec:quasigraph} to tropical curves. The analogues of the posets appearing in Section \ref{sec:quasigraph} will be polyhedral complexes. Moreover, we will introduce the Jacobian of a tropical curve by means of quasistable divisors, and prove that it is homeomorphic to the usual tropical Jacobian. 

Let $X$ be a tropical curve. A degree-$d$ polarization on $X$ is a function $\mu\col X\to\R$ such that $\mu(p)=0$ for all, but finitely many $p\in X$, and $\sum_{p\in X}\mu(p)=d$. We define the \emph{support} of $\mu$ as 
\[
\supp(\mu):=\{p\in X;\mu(p)\neq 0\}.
\]
   Let $\mu$ be a degree-$d$ polarization on $X$. For every tropical subcurve $Y\subset X$, we define $\mu(Y):=\sum_{p\in Y}\mu(p)$. 
For any divisor $\D$ on $X$ and every tropical subcurve $Y\subset X$, we set 
\[
\beta_\D(Y):=\deg(\D|_Y)-\mu(Y)+\frac{\delta_Y}{2}.
\]

We define the set $\Rel$ of \emph{relevant points of $X$} (with respect to $\mu$) and the set $\Rel_\D$ of \emph{$\D$-relevant points} as
\[
\Rel:=V(X)\cup\supp(\mu) 
\quad \text{ and } \quad
\Rel_\D:=\Rel\cup\supp(\D).
\]
Note that if $p$ is not $\D$-relevant, then $\beta_\D(p)=1$. Given a tropical subcurve $Y\subset X$, we also define 
\[
\rel_\D(Y)=|\Rel_\D\setminus Y|.
\]
We define the graphs $\Gamma_X$, $\Gamma_{X,\D}$, $\Gamma_{Y,\D}$, as the models of $X$ whose sets of vertices are 
\[
V(\Gamma_X)=\Rel, \quad\quad V(\Gamma_{X,\D})=\Rel_\D, \quad\quad V(\Gamma_{Y,\D})=V(Y)\cup \Rel_\D.
\]

Note that a degree-$d$ polarization $\mu$ on $X$ induces a degree-$d$ polarization on $\Gamma_X$, $\Gamma_{X,\D}$ and $\Gamma_{Y,\D}$ which, abusing notation, we will denote by $\mu$.

\begin{Lem}
\label{lem:beta}
Let $X$ be a tropical curve and $Y,Z$ be tropical subcurves of $X$. Then
\[
\beta_\D(Y\cap Z)+\beta_\D(Y\cup Z)=\beta_\D(Y)+\beta_\D(Z).
\]
In particular, if $Y\cap Z$ consists  of a finite number of non $\D$-relevant points, then 
\[
\beta_\D(Y\cup Z)=\beta_\D(Y)+\beta_\D(Z)-|Y\cap Z|.
\]
\end{Lem}
\begin{proof}
The proof of the first equation follows  simply observing that a point $p\in X$ contributes the same in each side of the equality. 
On the other hand, if $Y\cap Z$ consists of non $\D$-relevant points, then
\[
\beta_\D(Y\cap Z)=\sum_{p\in Y\cap Z}\beta_\D(p)=|Y\cap Z|,
\]
and hence the second equality holds.
\end{proof}

Given a tropical subcurve $Y\subset X$ and $\epsilon\in \mathbb{R}_{>0}$, we define the tropical subcurve
\[
Y^{\epsilon}:=\overline{\bigcup_{p\in Y}B_\epsilon(p)},
\]
where $B_\epsilon(p)$ is the ball in $X$ with radius $\epsilon$ and center $p$. Note that,  for a sufficiently small $\epsilon\in \mathbb R_{>0}$, we have $\beta_\D(Y^\epsilon)=\beta_\D(Y)$. \par

\begin{Def}\label{def:quasistable}
Let $X$ be a tropical curve. Let $\mu$ be a degree-$d$ polarization on $X$ and $\D$ be a divisor on $X$. We say that $\D$ is \emph{$\mu$-semistable}  if for every tropical subcurve $Y\subset X$ we have $\beta_\D(Y)\geq0$. Given a point $p_0$ of $X$, we say that $\D$ is \emph{$(p_0,\mu)$-quasistable} if it is $\mu$-semistable and $\beta_\D(Y)>0$ for every proper subcurve $Y\subset X$ with $p_0\in Y$. 
\end{Def}

Note that, equivalently, $\D$ is $(p_0,\mu)$-quasistable if and only if for every tropical subcurve $Y\subset X$ we have $\beta_\D(Y)\leq\delta_{X,Y}$, with strict inequality if $p_0\notin Y$. Indeed, we can assume that $Y$ has no $\D$-relevant points in its border (just change $Y$ with some $Y^{\epsilon}$ for a sufficiently small $\epsilon$), interchange $Y$ with $\overline{X\setminus Y}$, and use Lemma \ref{lem:beta}.

The quasistability for a tropical curve and for one of its model are closely related, as it is illustrated by the next proposition and the subsequent corollary.

\begin{Prop}
\label{prop:quasiquasi}
Let $(X,p_0)$ be a pointed tropical curve and $\mu$ be a degree-$d$ polarization on $X$. A degree-$d$ divisor $\D$ on $X$ is $(p_0,\mu)$-quasistable if and only if $D$ is $(p_0,\mu)$-quasistable on $\Gamma_{X,\D}$, where $D$ is the divisor $\D$ seen as divisor on $\Gamma_{X,\D}$.
\end{Prop}
\begin{proof}
Given a subset $V\subset V(\Gamma_{X,\D})$ there is an induced tropical subcurve $Y_V$ of $X$ defined as $Y_V:=V\cup\bigcup_{e\in E(V)} e$,  where $E(V)$ are all the edges in $E(\Gamma)$ connecting two (possibly coincident) vertices in $V$. Then $\beta_D(V)=\beta_\D(Y_V)$, and this proves the  ``only if'' implication.\par
Conversely, let $Y$ be a subcurve of $X$ and define $V=Y\cap V(\Gamma_{{X,\D}})$. Let us prove that $\beta_\D(Y)\geq\beta_\D(Y_V)$. 

First we show that, given an edge $e\in E(\Gamma_{X,\D})\setminus E(V)$, we have $\beta_\D(Y)\geq\beta_\D(Y_e)$, where $Y_e:=Y\setminus e^\circ$ and $e^\circ$ is the interior of $e$.
By Lemma \ref{lem:beta}, we have
\[
\beta_\D(Y)+\beta_\D(Y_e\cap e)=\beta_\D(Y_e)+\beta_\D(Y\cap e),
\]
then it is sufficient to prove that $\beta_\D(Y\cap e)\geq\beta_\D(Y_e\cap e)$.

We claim that $\delta_{Y\cap e}\geq \delta_{Y_e\cap e}$. 
Indeed, if $Y_e\cap e=\emptyset$ the result is trivial. So, we can assume that $Y_e\cap e=v$, where $v$ is a vertex of $\Gamma_{X,\D}$. Since the other vertex $v'$ incident to $e$ satisfies $v'\notin Y$, we have that $\delta_{Y\cap e}\geq\val(v)$, which proves the claim.\par

 So we get
\begin{align*}
\beta_\D(Y\cap e)=&\deg(\D|_{Y\cap e})-\mu(Y\cap e)+\frac{\delta_{Y\cap e}}{2}\\
                =&\deg(\D|_{Y_e\cap e})-\mu(Y_e\cap e)+\frac{\delta_{Y\cap e}}{2}\\
								\geq&\deg(\D|_{Y_e\cap e})-\mu(Y_e\cap e)+\frac{\delta_{Y_e\cap e}}{2}\\
                =&\beta_\D(Y_e\cap e).
\end{align*}

The fact that $\beta_\D(Y)\geq\beta_\D(Y_e)$ for every $e\in E(\Gamma_{X,\D})\setminus E(V)$ means that we can assume that $Y$ does not contain any point in the interior of an edge outside $E(V)$.\par
Now to prove that $\beta_\D(Y)\geq\beta_\D(Y_V)$, it suffices to show that, given an edge $e\in E(V)$, we have $\beta_\D(Y)\geq\beta_\D(Y^e)$, where $Y^e:=Y\cup e$. Again, by Lemma \ref{lem:beta}, we have
\[
\beta_\D(Y^e)+\beta_\D(Y\cap e)=\beta_\D(Y)+\beta_\D(e).
\]
So it is enough that $\beta_\D(Y\cap e)\geq\beta_\D(e)$. Note that both $\D$ and $\mu$ are supported on $V(\Gamma_{X,\D})$, and $Y\cap e$ contains both vertices incident to $e$.  Hence 
\begin{align*}
\beta_\D(Y\cap e)=&\deg(\D|_{Y\cap e})-\mu(Y\cap e)+\frac{\delta_{Y\cap e}}{2}\\
                =&\deg(\D|_{e})-\mu(e)+\frac{\delta_{Y\cap e}}{2}\\
								\geq&\deg(\D|_{e})-\mu(e)+\frac{\delta_{e}}{2}\\
                =&\beta_\D(e).
\end{align*}
This concludes the proof.
\end{proof}
\begin{CorDef}
\label{cor:quasiquasi}
Preserve the notations of Proposition \ref{prop:quasiquasi} and let $\D$ be a $(p_0,\mu)$-quasistable degree-$d$ divisor on $X$. Then $\Gamma_{X,\D}$ is an $\E$-subdivision of $\Gamma_X$ for some $\E\subset E(\Gamma_X)$, and the pair $(\E,D)$ is a $(p_0,\mu)$-quasistable degree-$d$ pseudo-divisor on $\Gamma_X$, where $D$ is the divisor $\D$ seen as a divisor on $\Gamma_X^\E$. We call $(\E,D)$ the pseudo-divisor on $\Gamma_X$ induced by $\D$.
\end{CorDef}
\begin{proof}
The fact that $\Gamma_{X,\D}$ is an $\E$-subdivision for some $\E\subset E(\Gamma_X)$ comes from Remark \ref{rem:subdivision}. The remaining statements are clear from Proposition \ref{prop:quasiquasi}.
\end{proof}

Let $\iota\col X\to Y$ be a specialization  of tropical curves. Let $\mu$ be a degree-$d$ polarization on $X$ and $\D$ be a divisor on $X$ of degree $d$. We define a degree-$d$ polarization $\iota_*(\mu)$ on $Y$ and a degree-$d$ divisor $\iota_*(\D)$ on $Y$ as
\[
\iota_*(\mu)(p')=\sum_{p\in \iota^{-1}(p')}\mu(p)\quad\text{ and }\quad\iota_*(\D)(p')=\sum_{p\in \iota^{-1}(p')}\D(p).
\]

\begin{Lem}\label{lem:i*}
Let $\iota\col X\to Y$ be a specialization of tropical curves. Let $\mu$ be a degree-$d$ polarization of $X$ and $\D$ a $(p_0,\mu)$-quasistable divisor on $X$ of degree $d$, for some $p_0\in X$. Then $\iota_*(\D)$ is a $(\iota(p_0),\iota_*(\mu))$-quasistable  divisor on $Y$ of degree $d$.
\end{Lem}
\begin{proof}
For every proper tropical subcurve $Z$ of $Y$, we have 
\begin{align*}
\beta_{\iota_*(\D)}(Z)=&\deg(\iota_*(\D)|_Z)-\iota_*(\mu)(Z)+\frac{\delta_{Z}}{2}\\
                     =&\deg(\D|_{\iota^{-1}(Z)})-\mu(\iota^{-1}(Z))+\frac{\delta_{\iota^{-1}(Z)}}{2}\\
										 =&\beta_\D(\iota^{-1}(Z)),
\end{align*}
and the last term is always nonnegative, and it is positive if $p_0\in \iota^{-1}(Z)$ or, equivalently, if $\iota(p_0)\in Z$. This proves that $\iota_*(\D)$ is $(\iota(p_0),\iota_*(\mu))$-quasistable.
\end{proof}

Next we have a key result stating that quasistable divisors can be chosen as canonical representatives for equivalence classes of divisors on a tropical curve.

\begin{Thm}
\label{thm:quasistable}
Let $(X,p_0)$ be a pointed tropical curve and $\mu$ a degree-$d$ polarization on $X$. Given a divisor $\D$ on $X$ of degree $d$, there exists a unique degree-$d$ divisor equivalent to $\D$ which is $(p_0,\mu)$-quasistable.
\end{Thm}

\begin{proof}
The set of points $p\in X$ such that $\{\D(p),\mu(p)\}\neq\{0\}$ is finite. Note that by Lemma \ref{lem:beta}, there exists a unique minimal subcurve $Y\subset X$ with $\beta_\D(Y)$ minimal, i.e., $\beta_\D(Z)\geq\beta_\D(Y)$ for every tropical subcurve $Z\subset X$ and the inequality is strict if $Z\subsetneq Y$. \par
 Let $\l_0$ be the minimum of the lengths of the edges in $\out(Y)$ (we use $\Gamma_{X,\D}$ as a model for $X$ to define $\out(Y)$). For each $e\in\out(Y)$ we give the orientation away from $Y$ and define the tropical subcurve $I_{e,\l_0}:=\ol{p_{e,o}p_{e,\l_0}}$ and we let $I_{e,\l_0}^\circ$ be its interior. Let $\F:=\D_{Y,\l_0}$ be the chip-firing divisor emanating from $Y$ with length $\l_0$. Define $\D':=\D-\F$, and set 
\[
\Delta_Y:=Y\cap \overline{X\setminus Y}
\quad \text{ and } \quad 
\Delta_{\F,Y}:=\supp(\F)\setminus\Delta_Y.
\] 
In other words, $\Delta_{\F,Y}$ are the points $p\in X$ where $\F(p)>0$. Note that $\mu(e^\circ)=0$ and $\supp(\D)\cap e^\circ=\emptyset$ for every $e\in\out(Y)$.\par
  In what follows, we will prove that, for every subcurve $Z$ of $X$, then $\beta_{\D'}(Z)\geq \beta_\D(Y)$, and if the equality holds, then $Z$ is strictly bigger then $Y$, in the sense that $\rel_{D'}(Z)<\rel_D(Y)$.\par

We begin proving some basic inequalities. Let $Z$ be a tropical subcurve of $X$ such that $Z\cap Y=\emptyset$ and $Z\cap I_{e,\l_0}^\circ=\emptyset$ for every $e\in \out(Y)$. Define 
\[
W:=Z\cup Y\cup\underset{p_{e,\l_0}\in Z}{\bigcup_{e\in\out(Y)}}I_{e,\l_0}.
\]
We have
\begin{align*}
\beta_\D(W)=&\deg(\D|_W)-\mu(W)+\frac{\delta_{W}}{2}\\
          =&\deg(\D|_Y)+\deg(\D|_Z)-\mu(Y)-\mu(Z)+\frac{\delta_{Y}+\delta_{Z}}{2}\\&-|\{e\in\out(Y);p_{e,\l_0}\in Z\}|\\
=&\beta_\D(Y)+\beta_\D(Z)-|\{e\in\out(Y);p_{e,\l_0}\in Z\}|\\
=&\beta_\D(Y)+\beta_\D(Z)-\deg(\F|_{Z\cap\Delta_{\F,Y}}),
\end{align*}
where the last equality comes from Equation \eqref{eq:chip}.  By the minimal property of $Y$ we have $\beta_\D(W)\geq\beta_\D(Y)$. We deduce that, for every tropical subcurve $Z$ of $X$ such that $Z\cap Y=\emptyset$ and $Z\cap I_{e,\l_0}^\circ=\emptyset$ for every $e\in\out(Y)$,
\begin{equation}
\label{eq:betaF}
\beta_\D(Z)\geq\deg(\F|_{Z\cap\Delta_{\F,Y}}).
\end{equation}

Now, let us prove that for every subcurve $Z$ of $X$ we have $\beta_{\D'}(Z)\geq\beta_D(Y)$. If $Z'$ is a connected component of $Z$ contained in $I_{e,\l_0}^\circ$ for some $e\in$, then $\beta_{\D'}(Z')=1$. By Lemma \ref{lem:beta}, we have that $\beta_{\D'}(Z)=\beta_{\D'}(Z\setminus Z')+1$. Then, we can restrict to the case where $Z$ has no connected components contained in $I_{e,\l_0}^\circ$ for every $e\in\out(Y)$.  Define $Z_1:=Z\cap Y$ and 
\[
Z_2:=\overline{Z\setminus(Y\cup\bigcup_{e\in\out(Y)}I_{e,\l_0})}.
\]
Then
\begin{align*}
\beta_\D(Z)=&\deg(\D|_Z)-\mu(Z)+\frac{\delta_{Z}}{2}\\
          =&\deg(\D|_{Z_1})+\deg(\D|_{Z_2})-\mu(Z_1)-\mu(Z_2)+\frac{\delta_{Z_1}+\delta_{Z_2}}{2}\\&-|\{e\in\out(Y);I_{e,\l_0}\subset Z\}|\\
          =&\beta_\D(Z_1)+\beta_\D(Z_2)-|\{e\in\out(Y);I_{e,\l_0}\subset Z\}|,
\end{align*}
from which we get 
\begin{align*}
\beta_{\D'}(Z)  =& \beta_\D(Z)-\deg(\F|_Z) \\
 =&\beta_\D(Z)-\deg(\F|_{Z\cap \Delta_Y})-\deg(\F|_{Z\cap\Delta_{\F,Y}})\\
=&\beta_\D(Z_1)+(\beta_\D(Z_2)-\deg(\F|_{Z_2\cap\Delta_{\F,Y}}))\\&-(|\{e\in\out(Y);I_{e,\l_0}\subset Z\}|+\deg(\F|_{Z_1\cap \Delta_Y})).
\end{align*}
However, by Equation \eqref{eq:chip}, we have 
\begin{equation}
\label{eq:Z1}
|\{e\in\out(Y);I_{e,\l_0}\subset Z\}|+\deg(\F|_{Z_1\cap \Delta_Y})\leq0
\end{equation}
and, by Equation \eqref{eq:betaF}, and the fact that $Z\cap Y=\emptyset$ and $Z\cap I_{e,\l_0}^\circ=\emptyset$ for every $e\in \out(Y)$, we have
\[
\beta_\D(Z_2)-\deg(\F|_{Z_2\cap\Delta_{\F,Y}})\geq0,
\]
so we deduce that $\beta_{\D'}(Z)\geq\beta_\D(Z_1)\geq\beta_\D(Y)$ (recall the minimal property of $Y$). Moreover, if $\beta_{\D'}(Z)=\beta_\D(Y)$, then $Z_1=Y$ and equality holds in Equation \eqref{eq:Z1}. In this case, $e\subset Z$ for every $e\in\out(Y)$  because $Z_1=Y$, from which we get $\deg(\F|_{Z_1\cap \Delta_Y})=\deg(\F|_{\Delta_Y})=|\out(Y)|$, and hence $\rel_{\D'}(Z)<\rel_\D(Y)$, because $Z$ contains the vertices incident to any edge $e_0\in \out(Y)$ of length $\l_0$.\par

Repeating the process, we eventually arrive in the case where the minimal subcurve $Y\subset X$ with $\beta_\D(Y)$ minimal is empty. In this case, $\beta_{\D}(Y)=0$ and $\D$ is equivalent to a $\mu$-semistable divisor $\D'$. To prove that every $\mu$-semistable divisor $\D$ is equivalent to a $(p_0,\mu)$-quasistable divisor, just repeat the above process for the minimal curve $Y$ containing $p_0$ with $\beta_\D(Y)=0$.

Finally we prove the uniqueness in the statement. 
Assume that  $\D_1$ and $\D_2$ are $(p_0,\mu)$-quasistable divisors on $X$ of  degree $d$ such that $\D_1\sim \D_2$, then $\D_1=\D_2$. 
Since $\D_1\sim \D_2$, there exists a rational function $f$ on $X$ such that $\D_1=\D_2+\div(f)$. By contradiction, assume that $f$ is not constant.\par
Assume that $f$ is also nowhere constant. By Lemma \ref{lem:valP}, there exists a point $p\in X$ such that $\ord_p(f)\geq\delta_{X,p}$, then $\deg(\D_1|_p)\geq\deg(\D_2|_p)+\delta_{X,p}$. Using that $\D_1$ and $\D_2$ are $(p_0,\mu)$-quasistable, we get
\[
\delta_{X,p}\geq\beta_{\D_1}(p)\geq\beta_{\D_2}(p)+\delta_{X,p}\geq\delta_{X,p}.
\]
Then $\beta_{\D_1}(p)=\delta_{X,p}$ and $\beta_{\D_2}(p)=0$. If $p\neq p_0$, the first equality is a contradiction, if $p= p_0$, the second equality is a contradiction.\par
Assume now that there are segments of $X$ over which $f$ is constant. Let $\iota\col X\ra X'$ be the specialization of tropical curves obtained by contracting all the maximal segments of $X$ over which $f$ is constant. Note that $X'$ is not a point because $f$ is not constant. Then $\iota_*(\D_1)$ and $\iota_*(\D_2)$ are $(\iota(p_0),\iota_*(\mu))$-quasistable divisors on $X'$ by Lemma \ref{lem:i*}, and $f$ induces a rational function $f'$ on $X'$ which is nowhere constant, and such that $\iota_*(\D_1)=\iota_*(\D_2)+\div(f')$. By the first part of the proof, we have a contradiction.
\end{proof}

Let $(X,p_0)$ be a pointed tropical curve with length function $\ell$, and let $\mu$ be a degree-$d$ polarization on $X$.  Let $\Gamma_X$ be the model of $X$ whose vertices are the relevant points of $X$. For each $(p_0,\mu)$-quasistable pseudo-divisor $(\E,D)$ on $\Gamma_X$ we define polyhedra
\begin{align*}
\P_{(\E,D)}:=&\prod_{e\in\E}e=\prod_{e\in\E}[0,\l(e)]\subset \R^{\E}\\
\P^\circ_{(\E,D)}:=&\prod_{e\in\E}e^\circ=\prod_{e\in\E}(0,\l(e))\subset \R^{\E},
\end{align*}
where $e^\circ$ denotes the interior of an edge $e$. 
Note that if $\E=\emptyset$, then $\P_{\E,D}$ is just a point. If $(\E,D)\ra(\E',D')$ is a specialization of pseudo-divisors, then we have an induced face morphism of polyhedra $f\col\P_{(\E',D')}\to\P_{(\E,D)}$. The polyhedron $\P^\circ_{(\E,D)}$ parametrizes $(p_0,\mu)$-quasistable divisors on $X$,  
whose induced pseudo-divisor on $\Gamma_X$ is $(\E,D)$.

\begin{Def}
\label{def:jtrop}
Let $(X, p_0)$ be a pointed tropical curve  and $\mu$ be a degree-$d$ polarization on $X$. The \emph{Jacobian of $X$ with respect to $(p_0,\mu)$} is the polyhedral complex
\[
J^\trop_{p_0,\mu}(X)=\lim_{\longrightarrow}\P_{(\E,D)},
\]
where the limit is taken over the poset $\mathcal{QD}_{p_0,\mu}(\Gamma_X)$. We have a set-theoretically decomposition
\[
J^\trop_{p_0,\mu}(X)=\coprod_{(\E,\D)} \P^\circ_{(\E,D)},
\]
where the union is taken over $(\E,\D)\in\mathcal{QD}_{p_0,\mu}(\Gamma_X)$.
\end{Def}

\begin{Exa}
Let $X$ be the tropical curve associated to the graph with $2$ vertices and $3$ edges connecting such vertices, as in Figure \ref{fig:poset}, where all edges have length $1$. Let $\mu$ be the degree-$0$ polarization on $X$ given by $\mu(p)=0$ for every $p\in X$ and assume that $p_0$ is the leftmost vertex. Then the polyhedral complex $J^\trop_{p_0,\mu}(X)$ is depicted in Figure \ref{fig:jac}.
\begin{figure}[ht]
\begin{tikzpicture}[scale=4]
\draw[ultra thick] (0,0) rectangle (1,1);
\draw[ultra thick] (0,0) -- (0,1) -- (-1,0) -- (-1,-1) -- (0,0);
\draw[ultra thick] (0,0) -- (1,0) -- (0,-1) -- (-1,-1) -- (0,0);
\draw[fill] (0,0) circle [radius=0.03];
\draw[fill] (0,1) circle [radius=0.03];
\draw[fill] (1,0) circle [radius=0.03];
\draw[fill] (1,1) circle [radius=0.03];
\draw[fill] (0,-1) circle [radius=0.03];
\draw[fill] (-1,0) circle [radius=0.03];
\draw[fill] (-1,-1) circle [radius=0.03];
\begin{scope}[shift={(0.3,0.5)},scale=0.5]
\draw (0,0) to [out=45, in=135] (1,0);
\draw (0,0) to (1,0);
\draw (0,0) to [out=-45, in=-135] (1,0);
\draw[fill] (0,0) circle [radius=0.02];
\draw[fill] (1,0) circle [radius=0.02];
\draw[fill] (0.5,0.21) circle [radius=0.02];
\draw[fill] (0.5,0) circle [radius=0.02];
\node[left] at (0,0) {1};
\node[right] at (1,0) {1};
\node[above] at (0.5,0.21) {-1};
\node[below] at (0.5,0) {-1};
\end{scope}
\begin{scope}[shift={(-0.7,0)},scale=0.5]
\draw (0,0) to [out=45, in=135] (1,0);
\draw (0,0) to (1,0);
\draw (0,0) to [out=-45, in=-135] (1,0);
\draw[fill] (0,0) circle [radius=0.02];
\draw[fill] (1,0) circle [radius=0.02];
\draw[fill] (0.5,0.21) circle [radius=0.02];
\draw[fill] (0.5,-0.21) circle [radius=0.02];
\node[left] at (0,0) {1};
\node[right] at (1,0) {1};
\node[above] at (0.5,0.21) {-1};
\node[below] at (0.5,-0.21) {-1};
\end{scope}
\begin{scope}[shift={(-0.25,-0.5)},scale=0.5]
\draw (0,0) to [out=45, in=135] (1,0);
\draw (0,0) to (1,0);
\draw (0,0) to [out=-45, in=-135] (1,0);
\draw[fill] (0,0) circle [radius=0.02];
\draw[fill] (1,0) circle [radius=0.02];
\draw[fill] (0.5,-0.21) circle [radius=0.02];
\draw[fill] (0.5,0) circle [radius=0.02];
\node[left] at (0,0) {1};
\node[right] at (1,0) {1};
\node[below] at (0.5,-0.21) {-1};
\node[above] at (0.5,0) {-1};
\end{scope}
\begin{scope}[shift={(1.15,0)},scale=0.2]
\draw (0,0) to [out=45, in=135] (1,0);
\draw (0,0) to (1,0);
\draw (0,0) to [out=-45, in=-135] (1,0);
\draw[fill] (0,0) circle [radius=0.02];
\draw[fill] (1,0) circle [radius=0.02];
\node[left] at (0,0) {0};
\node[right] at (1,0) {0};
\end{scope}
\begin{scope}[shift={(-0.1,1.1)},scale=0.2]
\draw (0,0) to [out=45, in=135] (1,0);
\draw (0,0) to (1,0);
\draw (0,0) to [out=-45, in=-135] (1,0);
\draw[fill] (0,0) circle [radius=0.02];
\draw[fill] (1,0) circle [radius=0.02];
\node[left] at (0,0) {0};
\node[right] at (1,0) {0};
\end{scope}
\begin{scope}[shift={(-1.1,-1.1)},scale=0.2]
\draw (0,0) to [out=45, in=135] (1,0);
\draw (0,0) to (1,0);
\draw (0,0) to [out=-45, in=-135] (1,0);
\draw[fill] (0,0) circle [radius=0.02];
\draw[fill] (1,0) circle [radius=0.02];
\node[left] at (0,0) {0};
\node[right] at (1,0) {0};
\end{scope}
\begin{scope}[shift={(0.9,1.1)},scale=0.2]
\draw (0,0) to [out=45, in=135] (1,0);
\draw (0,0) to (1,0);
\draw (0,0) to [out=-45, in=-135] (1,0);
\draw[fill] (0,0) circle [radius=0.02];
\draw[fill] (1,0) circle [radius=0.02];
\node[left] at (0,0) {1};
\node[right] at (1,0) {-1};
\end{scope}
\begin{scope}[shift={(-0.1,-1.1)},scale=0.2]
\draw (0,0) to [out=45, in=135] (1,0);
\draw (0,0) to (1,0);
\draw (0,0) to [out=-45, in=-135] (1,0);
\draw[fill] (0,0) circle [radius=0.02];
\draw[fill] (1,0) circle [radius=0.02];
\node[left] at (0,0) {1};
\node[right] at (1,0) {-1};
\end{scope}
\begin{scope}[shift={(-1.4,0)},scale=0.2]
\draw (0,0) to [out=45, in=135] (1,0);
\draw (0,0) to (1,0);
\draw (0,0) to [out=-45, in=-135] (1,0);
\draw[fill] (0,0) circle [radius=0.02];
\draw[fill] (1,0) circle [radius=0.02];
\node[left] at (0,0) {1};
\node[right] at (1,0) {-1};
\end{scope}
\begin{scope}[shift={(0.13,0.1)},scale=0.2]
\draw (0,0) to [out=45, in=135] (1,0);
\draw (0,0) to (1,0);
\draw (0,0) to [out=-45, in=-135] (1,0);
\draw[fill] (0,0) circle [radius=0.02];
\draw[fill] (1,0) circle [radius=0.02];
\node[left] at (0,0) {-1};
\node[right] at (1,0) {1};
\end{scope}
\begin{scope}[shift={(0.6,-0.6)},scale=0.3]
\draw (0,0) to [out=45, in=135] (1,0);
\draw (0,0) to (1,0);
\draw (0,0) to [out=-45, in=-135] (1,0);
\draw[fill] (0,0) circle [radius=0.02];
\draw[fill] (1,0) circle [radius=0.02];
\draw[fill] (0.5,-0.21) circle [radius=0.02];
\node[left] at (0,0) {1};
\node[right] at (1,0) {0};
\node[below] at (0.5,-0.21) {-1};
\end{scope}
\begin{scope}[shift={(-0.83,0.6)},scale=0.3]
\draw (0,0) to [out=45, in=135] (1,0);
\draw (0,0) to (1,0);
\draw (0,0) to [out=-45, in=-135] (1,0);
\draw[fill] (0,0) circle [radius=0.02];
\draw[fill] (1,0) circle [radius=0.02];
\draw[fill] (0.5,-0.21) circle [radius=0.02];
\node[left] at (0,0) {1};
\node[right] at (1,0) {0};
\node[below] at (0.5,-0.21) {-1};
\end{scope}
\begin{scope}[shift={(1.14,0.5)},scale=0.3]
\draw (0,0) to [out=45, in=135] (1,0);
\draw (0,0) to (1,0);
\draw (0,0) to [out=-45, in=-135] (1,0);
\draw[fill] (0,0) circle [radius=0.02];
\draw[fill] (1,0) circle [radius=0.02];
\draw[fill] (0.5,0.21) circle [radius=0.02];
\node[left] at (0,0) {1};
\node[right] at (1,0) {0};
\node[above] at (0.5,0.21) {-1};
\end{scope}
\begin{scope}[shift={(-1.45,-0.5)},scale=0.3]
\draw (0,0) to [out=45, in=135] (1,0);
\draw (0,0) to (1,0);
\draw (0,0) to [out=-45, in=-135] (1,0);
\draw[fill] (0,0) circle [radius=0.02];
\draw[fill] (1,0) circle [radius=0.02];
\draw[fill] (0.5,0.21) circle [radius=0.02];
\node[left] at (0,0) {1};
\node[right] at (1,0) {0};
\node[above] at (0.5,0.21) {-1};
\end{scope}
\begin{scope}[shift={(0.35,1.1)},scale=0.3]
\draw (0,0) to [out=45, in=135] (1,0);
\draw (0,0) to (1,0);
\draw (0,0) to [out=-45, in=-135] (1,0);
\draw[fill] (0,0) circle [radius=0.02];
\draw[fill] (1,0) circle [radius=0.02];
\draw[fill] (0.5,0) circle [radius=0.02];
\node[left] at (0,0) {1};
\node[right] at (1,0) {0};
\node[above] at (0.5,-0.05) {\tiny{-1}};
\end{scope}
\begin{scope}[shift={(-0.65,-1.1)},scale=0.3]
\draw (0,0) to [out=45, in=135] (1,0);
\draw (0,0) to (1,0);
\draw (0,0) to [out=-45, in=-135] (1,0);
\draw[fill] (0,0) circle [radius=0.02];
\draw[fill] (1,0) circle [radius=0.02];
\draw[fill] (0.5,0) circle [radius=0.02];
\node[left] at (0,0) {1};
\node[right] at (1,0) {0};
\node[above] at (0.5,-0.05) {\tiny{-1}};
\end{scope}
\end{tikzpicture}
\caption{The Jacobian $J_{p_0,\mu}^{trop}(X)$.}
\label{fig:jac}
\end{figure}
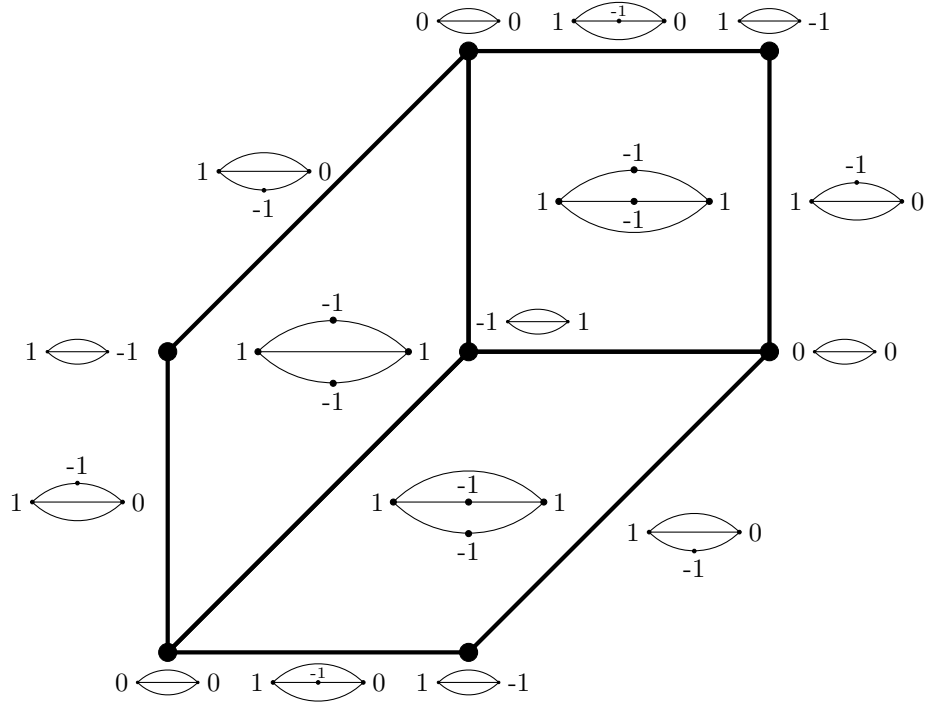

By the description of the poset $\mathcal{QD}_{p_0,\mu}(\Gamma_X)$ in Example \ref{exa:poset}, we see that $J^\trop_{p_0,\mu}(X)$ has $3$ cells of dimension 2, $6$ cells of  dimension $1$, and $3$ cells of dimension $0$. Note that the outer edges are identified making $J^\trop_{p_0,\mu}(X)$ a real torus. Also, the point associated to the zero divisor in $\Gamma_X$ is distinguished since it is contained in every cell of  dimension 1, while the other cells of  dimension 0 are contained in exactly $3$ cells of dimension 1.
\end{Exa}

\begin{Rem}
\label{rem:cont}
There is a natural map $f\col J^\trop_{p_0,\mu}(X)\to \mathcal{QD}_{p_0,\mu}(\Gamma_X)$ and this map is continuous. Indeed, if $(\E,D)$ is a pseudo-divisor on $\Gamma_X$ and $T=\overline{\{(\E,D)\}}$ is the closure of $\{(\E,D)\}$ in $\mathcal{QD}_{p_0,\mu}(\Gamma_X)$, then $f^{-1}(T)=\P_{(\E,D)}$ which is closed. So $f$ is continuous, since every closed set in $\mathcal{QD}_{p_0,\mu}(\Gamma_X)$ is a finite union of such closures.
\end{Rem}

Now we give some topological properties of the Jacobian of $X$ with respect to $(p_0,\mu)$, comparing it with the usual tropical Jacobian of $X$ (recall Equation \eqref{eq:Jtropdef}).

\begin{Thm}
\label{thm:jX}
Let $(X, p_0)$ be a pointed tropical curve and let $\mu$ be a degree-$d$ polarization on $X$. We have that $J^\trop_{p_0,\mu}(X)$ has pure dimension $g$,  it is connected in codimension $1$, and homeomorphic to $J^\trop(X)$.
\end{Thm}

\begin{proof} 
The fact that $J^\trop_{p_0,\mu}(X)$ has pure dimension $g$ and connected in codimension $1$ follows from Propositions \ref{prop:chain} and \ref{prop:cod1}.\par
 There exists a function $\alpha\colon J^\trop_{p_0,\mu}(X)\to J^\trop(X)$ that takes a $(p_0,\mu)$-quasistable divisor  $\D$ to the class of $\D-dp_0$ in $J^\trop(X)$. It follows from Theorem \ref{thm:quasistable} that $\alpha$ is a bijection. We now prove that $\alpha$ is a homeomorphism by showing that it is continuous (this is enough since $J^\trop_{p_0,\mu}(X)$ is compact and $J^\trop(X)$ is Hausdorff).\par
Fix an orientation on $\Gamma_X$ and fix (not necessarily oriented) paths from $p_0$ to every vertex of $\Gamma_X$. Let $(\E,D)$ be a $(p_0,\mu)$-quasistable divisor on $\Gamma_X$ and define $D_0\in \Div(\Gamma_X)$ by $D_0(v)=D(v)$ for every $v\in V(\Gamma_X)$. Let $\D_0$ be the divisor on $X$ induced by $D_0$.  Given a divisor $\D$ on $X$ parametrized by a point in $\P_{(\E,D)}$, then there are real numbers $x_e\in[0,\l(e)]$ such that $\D=\D_0-\sum_{e\in\E}p_{e,x_e}$.\par
  For every  $\D\in \P_{(\E,D)}$, we define the path $\gamma(\D)$ as 
\[
\gamma(\D):=\gamma(\D_0)+\sum_{e\in\E}\overrightarrow{p_{e,0}p_{e,x_e}},
\]
   This induces a map 
\[
\alpha_{(\E,D)}\colon\P_{(\E,D)}\to\Omega(X)^\vee
\] 
taking a divisor $\D$ to $\int_{\gamma(\D)}$. More precisely, the map $\alpha_{(\E,D)}$ is given by
\begin{align*}
\alpha_{(\E,D)}\col\prod_{e\in\E}[0,\l(e)]&\to \Omega(X)^\vee\\
(p_{e,x_e})_{e\in\E}&\mapsto \int_{\gamma(\D_0)}+\sum_{e\in\E}\frac{x_e}{\l(e)}\int_{e}.
\end{align*}
This means that $\alpha_{(\E,D)}$ is an affine map, hence it is continuous. The composition of this map with the quotient map $\Omega(X)^\vee\ra J^{trop}(X)$  is the restriction of $\alpha$ to $\P_{(\E,D)}$. Since $J^\trop_{p_0,\mu}(X)$ is the direct limit of the polyhedra $\P_{(\E,D)}$, then $\alpha$ is continuous.
\end{proof}

Now we move to the universal setting: we define a universal tropical Jacobian, give a modular description of the points that it parametrizes, and we prove the analogue of the properties stated in Theorem \ref{thm:maingraph}.

\begin{Def}
Let $(\Gamma,v_0)$ be a graph with $1$ leg and $\mu$ be a degree-$d$ polarization on $\Gamma$.  For each $(v_0,\mu)$-quasistable pseudo-divisor $(\E,D)$ on $\Gamma$, we define
\[
\sigma_{(\Gamma,\E,D)}:=\mathbb{R}^{E(\Gamma^\E)}_{\geq0} \quad \text{and}\quad \sigma^\circ_{(\Gamma,\E,D)}:=\mathbb{R}^{E(\Gamma^\E)}_{>0}.
\]
Note that, if $\iota\col(\Gamma,\E,D)\ra(\Gamma',\E',D')$ is a specialization, then there exists a natural inclusion $\iota\col \sigma_{(\Gamma',\E',D')}\to \sigma_{(\Gamma,\E,D)}$.  
Let $\mu$ be a universal genus-$g$ polarization. The \emph{universal Jacobian with respect to $\mu$} is defined as the generalized cone complex 
\[
J^\trop_{\mu,g}:=\lim_{\longrightarrow} \sigma_{(\Gamma,\E,D)}=\coprod_{(\Gamma,\E,D)} \sigma_{(\Gamma,\E,D)}^\circ/\text{Aut}(\Gamma,\E,D)
\] 
where the union is taken over all terns $(\Gamma,\E,D)$ running through all objects in the category $\mathbf{QD}_{\mu,g}^{op}$ (the opposite of the category $\mathbf{QD}_{\mu,g}$). 
\end{Def}

\begin{Prop}
\label{prop:parameter}
Let $\mu$ be a universal genus-$g$ polarization. The generalized cone complex $J^\trop_{\mu,g}$ parametrizes equivalence classes $(X,\D)$, where $X$ is a stable pointed tropical curve of genus $g$ and $\D$ is a $(p_0,\mu)$-quasistable divisor on $X$, where $p_0$ is the marked point of $X$.
\end{Prop}

\begin{proof}
Let $(X,p_0)$ be a stable pointed tropical curve of genus $g$. The stable model $\Gamma:=\Gamma_{st}$ of $X$ is a genus-$g$ stable weighted graph with $1$ leg. The polarization $\mu_{\Gamma}$ induces a polarization on $X$, such that $\Gamma_X=\Gamma$ (recall that $\Gamma_X$ denotes the model of $X$ whose vertices are the relevant points).\par
   It follows from Corollary \ref{cor:quasiquasi} that, if $\D$ is a $(p_0,\mu)$-quasistable divisor of $X$, then $\Gamma_{X,\D}=\Gamma^\E$ for some $\E\subset E(\Gamma_X)$. Note that $(\E,D)$ is a $(p_0,\mu)$-quasistable pseudo-divisor on $\Gamma$, where $D$ is the  divisor on $\Gamma^\E$ induced by $\D$. Therefore $X$ corresponds to a point in $\sigma^\circ_{(\Gamma_X,\E,D)}=\R^{E(\Gamma^\E)}_{>0}$, and the pair $(X,\D)$ corresponds to a point in $J^\trop_{\mu,g}$. If $(X,\D)$ and $(X',\D')$ are isomorphic, then the construction above will give rise to the same point in $J^\trop_{\mu,g}$. 
On the other hand if $(X,\D)$ and $(X',\D')$ corresponds to the same point in $J^\trop_{\mu,g}$ that belongs to a cell $\sigma^\circ_{(\Gamma,\E,D)}/\Aut(\Gamma,\E,D)$, then there is an isomorphism $\iota\col\Gamma_X\ra\Gamma_{X'}$ such that $\iota(\E)=\E'$ and $\iota_*(D)=D'$. Moreover, it follows that the metrics of $X$ and $X'$ are equal, hence $\iota$ induces an isomorphism of metric graphs between $X$ and $X'$ taking $\Gamma_{X,\D}$ to $\Gamma_{X',\D'}$, and hence $\D$ to $\D'$. This means that $(X,\D)$ and $(X',\D')$ are isomorphic.\par
	Conversely, it is clear that, for every triple $(\Gamma,\E,D)$, every point in $\sigma^\circ_{(\Gamma,\E,D)}$ corresponds to a pair $(X,\D)$ with $X$ a stable pointed tropical curve of genus $g$ and $\D$ a $(p_0,\mu)$-quasistable divisor on $X$.
\end{proof}

It follows from Proposition \ref{prop:parameter} that we have a natural forgetful map 
\[
\pi^{trop}\col J^\trop_{\mu,g}\to M^\trop_{g,1}.
\]

\begin{Rem}
\label{rem:contmu}
Note that the natural map $f\col J^\trop_{\mu,g}\to \mathcal{QD}_{\mu,g}$ is continuous. One can prove this fact as  in Remark \ref{rem:cont}.
\end{Rem}

\begin{Thm}
\label{thm:maintrop} 
Let $\mu$ be a universal genus-$g$ polarization. 
The generalized cone complex $J^\trop_{\mu,g}$ has pure dimension $4g-2$ and is connected in codimension $1$. The map $\pi^{trop}\col J^\trop_{\mu,g}\to M^\trop_{g,1}$ is a map of generalized cone complexes. For every stable pointed tropical curve $X$ of genus $g$, we have a homeomorphism
\[
(\pi^{trop})^{-1}([X])\simeq J^\trop_{p_0,\mu}(X)/\Aut(X).
\]
\end{Thm}
\begin{proof}
The fact that $J^\trop_{\mu,g}$ has pure dimension $4g-2$ and is connected in codimension $1$ follows from Theorem \ref{thm:maingraph} and Remark \ref{rem:contmu}.\par
 For each cone $\sigma_{(\Gamma,\E,D)}$, we have that $\pi^{\trop}$ induces a map $\sigma_{(\Gamma,\E,D)}\to M_{g,1}^\trop$ that factors through a chain of maps
\[
\sigma_{(\Gamma,\E,D)}=\R_{\geq0}^{E(\Gamma^\E)}\to \R_{\geq 0}^{E(\Gamma)}\ra \R_{\geq 0}^{E(\Gamma)}/\Aut(\Gamma) \subset M_{g,1}^\trop,
\]
 where the first one is defined in Equation \eqref{eq:hat} and the second one is the natural quotient map. Hence $\pi^{trop}$ is a morphism of generalized cone complexes.\par
   There exists a natural map $h\col J^\trop_{p_0,\mu}(X)\to J^\trop_{\mu,g}$ and we have $(\pi^{trop})^{-1}([X])=\Im(h)$. Moreover, $h(\D)=h(\D')$ if and only if there exists an automorphism $\alpha\col X\to X$ such that $\alpha_*(\D)=\D'$, which implies that $\Im(h)=J^\trop_{p_0,\mu}(X)/\Aut(X)$.
\end{proof}

\begin{Rem}
It is easy to check that the maximal cells of $J^\trop_{\mu,g}$ are of type $\sigma_{(\Gamma,\E,D)}/\Aut(\Gamma,\E,D)$ where $\Gamma$ is a 3-regular graph with weight function $0$ and $|\E|=g$.
Moreover, the codimension $1$ cells of $J^\trop_{\mu,g}$ are of the following type:
    \begin{enumerate}
		  \item $\sigma_{(\Gamma,\E,D)}/\Aut(\Gamma,\E,D)$ where $\Gamma$ has weight function $0$ and has exactly one vertex of valence $4$ and all other vertices of valence $3$, and $|\E|=g$;
			\item $\sigma_{(\Gamma,\E,D)}/\Aut(\Gamma,\E,D)$ where $\Gamma$ has weight function $0$ on all vertices but exactly one vertex of weight $1$ and valence $1$ and all other vertices of valence $3$, and $|\E|=g$;
		  \item $\sigma_{(\Gamma,\E,D)}/\Aut(\Gamma,\E,D)$ where $\Gamma$ is a $3$-regular graph of weight $0$, and $|\E|=g-1$. 
		\end{enumerate}
\end{Rem}

We end this section introducing a compactification of $J^{\trop}_{\mu,g}$.
First of all, we set
\[
\overline{\sigma}_{(\Gamma,\E,D)}:=\overline{\R}_{\geq0}^{E(\Gamma^\E)}\quad \text{ and }\quad \overline{\sigma}^\circ_{(\Gamma,\E,D)}:=\overline{\R}_{>0}^{E(\Gamma^\E)},
\]
where $\ol{\R}=\R\cup\{\infty\}$.
Then the compactification $\overline{J}^{\trop}_{\mu,g}$ of $J^{\trop}_{\mu,g}$ is defined as
\[
\overline{J}^{\trop}_{\mu,g}=\lim_{\longrightarrow}\overline{\sigma}_{(\Gamma,\E,D)}
\]
where $(\Gamma,\E,D)$ runs through all objects in the category $\mathbf{QD}_{\mu,g}^{op}$. We note that  $\overline{J}^{\trop}_{\mu,g}$ has a structure of extended generalized cone complexes as defined in \cite[Section 2]{ACP}. Moreover the map $\pi^\trop\col J^\trop_{\mu,g}\to M^\trop_{g,1}$ of generalized cone complexes extends to a map of extended generalized cone complexes
\[
\overline{\pi}^\trop\col \overline{J}^\trop_{\mu,g}\to \overline{M}^\trop_{g,1}
\]
as defined in \cite[Section 2]{ACP}.
The map $\overline{\pi}^\trop$ when restricted to the cone $\overline{\R}_{\geq0}^{E(\Gamma^\E)}$ is the map $\overline{\R}_{\geq0}^{E(\Gamma^\E)}\to \overline{\R}_{\geq0}^{E(\Gamma)} $
defined in Equation \eqref{eq:hat} where $y_e=\infty$ if $x_{e'}=\infty$ for some edge  $e'\in E(\Gamma^\E)$ over $e$.
\par

\begin{Rem}
The points of the boundary of $\overline{J}^\trop_{\mu,g}$ parametrizes pairs $(X,\D)$ where $X$ is an extended tropical curve (in the sense of \cite[Section 3.3]{Caporaso}) and $\D$  a $(p_0,\mu)$-quasistable divisor on $X$.  Here, the notion of $(p_0,\mu)$-quasistability for extended tropical curves can be given exactly as for tropical curves in Definition \ref{def:quasistable} (considering extended tropical subcurves). We do not know if there is an analogous of Theorem \ref{thm:quasistable} for extended tropical curves.
\end{Rem}

\section{The skeleton of the Esteves' universal Jacobian}\label{sec:skeleton}

\subsection{The Esteves' universal compactified Jacobian}\label{sec:EstJac}

  Let $(X,p_0)$ be a pointed nodal curve defined over an algebraically closed field $k$. Recall that a coherent sheaf $\I$ on $X$ is \emph{torsion-free} if it has no embedded components, \emph{rank-1} if it is invertible on a dense open subset of $C$, and \emph{simple} if 
$\text{Hom}(\I,\I)=k$. The \emph{degree} of $I$ is $\deg I=\chi(I)-\chi(\O_X)$.

Given an integer $d$, the \emph{degree-$d$ Jacobian} $\J_d(X)$ of $X$ is the scheme parametrizing the equivalence classes of invertible sheaves of degree $d$ on $X$. In general, $\J_d(X)$  is neither proper nor of finite type. A better behaved parameter space is obtained by 
resorting to torsion-free  rank-$1$ sheaves and to stability conditions.\par

The scheme $\J_d(X)$ is an open dense subscheme of the scheme $\mathcal{S}pl_d(X)$ parametrizing simple torsion-free rank-$1$ sheaves of degree-$d$ on $X$. We refer to \cite{AK} and \cite{Es01} for the construction of $\mathcal{S}pl_d(X)$ and its properties. Recall that $\mathcal{S}pl_d(X)$  is universally closed over $k$ and connected 
but, in general, not separated and only locally of finite type. To deal with a 
manageable piece of it, we resort to polarizations.

Let  $X_1,\ldots,X_m$ be the irreducible components of $X$. A \emph{degree-$d$ polarization on $X$} is any $m$-tuple of rational numbers $\mu=(\mu_1,\ldots,\mu_m)$ summing up to $d$.  For every proper subcurve $Y$ of $X$, we set 
  \[
  \mu(Y):=\sum_{X_i\subset Y} \mu_i.
  \]
The notion of a degree-$d$ polarization on a curve can also be given by a vector bundle $\mathcal{F}$ on $X$ such that $\deg(\mathcal{F})=-d\cdot\rank(\mathcal{F})$. In this case $\mu_i=-\deg(\mathcal{F}|_{X_i})/\rank(\mathcal{F})$ for every $i=1,\ldots,m$ (see \cite[Section 1.2]{Es01} and \cite[Remark 4.6]{KP}).\par
	
Note that $\mu$ can be seen as a degree-$d$ polarization on the dual graph $(\Gamma_X, v_0)$ of the pointed curve $X$. Conversely, every polarization on the dual graph $(\Gamma_X, v_0)$ induces a polarization on $X$.\par

Let $I$ be a rank-$1$ degree-$d$ torsion-free sheaf on $X$. We can define a pseudo-divisor $(\E_I,D_I)$ on $\Gamma_X$ as follows. The set $\E_I\subset E(\Gamma_X)$ is precisely the set of edges corresponding to nodes where $I$ is not locally free. For every $v\in V(\Gamma^\E)$, we set
\[
D_I(v)=\begin{cases}
        \deg(I|_{X_v}),&\text{ if $v\in V(\Gamma_X)$};\\
				-1,&\text{ if $v$ is exceptional},
				\end{cases}
\]
where $X_v$ is the component of $X$ corresponding to $v\in V(\Gamma_X)$. 
We call $(\E_I,D_I)$ the \emph{multidegree} of $I$. We say that $I$ is $\mu$-semistable  (respectively, \emph{$(p_0,\mu)$-quasistable}) if its multidegree $(\E_I,D_I)$ is a  $\mu$-semistable (respectively, $(v_0,\mu)$-quasistable) pseudo-divisor on $\Gamma_X$. We call $(\Gamma_X,\E_I,D_I)$ the \emph{dual graph} of $(X,I)$.
\par

Let $\mathcal{QD}_d(\Gamma_X)$ be the poset of all pseudo-divisors of degree-$d$ on $\Gamma_X$. The above construction gives rise to an anti-continuous function
\begin{align*}
\mathcal{S}pl_d(X)&\to \mathcal{QD}_d(\Gamma_X)\\
             I&\mapsto (\E_I,D_I).
\end{align*}		

The above notions naturally extend to families. 
 Let $f\col\X\to T$ be a  family of pointed nodal curves with section $\sigma\col T\to\X$.   
A polarization $\mu$ on $\X$ is the datum of polarizations on the fibers of $f$ that are compatible with specializations. We say that a sheaf $\I$ over $\X$ is \emph{$(\sigma,\mu)$-quasistable} if, for any closed point $t\in T$, the restriction of $\I$ to the fiber $f^{-1}(t)$ is a torsion-free rank-$1$ and $(\sigma(t),\mu)$-quasistable sheaf. There is an algebraic space $\ol{\J}_{\pi,\mu}$ parametrizing $(\sigma,\mu)$-quasistable sheaves over $\X$. This algebraic space is proper and of finite type   (\cite[Theorems A and B]{Es01}) and it represents the contravariant functor $\mathbf{J}_{\pi,\mu}$ from the category of locally Noetherian $T$-schemes to sets, defined on a $T$-scheme $B$ by
\[
\mathbf{J}_{\pi,\mu}(B):=\{(\sigma_B,\mu_B)\text{-quasistable sheaves over } \X\times_T B\stackrel{\pi_B}\lra B\}/\sim
\]
where $\sigma_B$ and $\mu_B$ are the pullback to $\X\times_T B$ of the section $\sigma$ and polarization $\mu$, and where $\sim$ is the equivalence relation given by $\I_1\sim \I_2$ if and only if there exists an invertible sheaf $\L$ on $B$ such that $\I_1\cong \I_2\otimes \pi_B^*\L$. \par

  This construction can be extended to the universal setting. More precisely, let $\overline{\M}_{g,1}$ be the Deligne-Mumford stack parametrizing stable pointed genus-$g$ curves, and let $\M_{g,1}$ be its open locus. Let $\overline{\M}_{g,2}\to\overline{\M}_{g,1}$ be the universal family over $\overline{\M}_{g,1}$. Let $\J_{d,g,1}\ra\M_{g,1}$ be the universal degree-$d$ Jacobian parametrizing invertible sheaves of degree-$d$ on smooth fibers of $\overline{\M}_{g,2}\to\overline{\M}_{g,1}$. 

For each universal degree-$d$ polarization $\mu$ over $\overline{\M}_{g,2}\to\overline{\M}_{g,1}$, there is  a proper and separated Deligne-Mumford stack $\ol{\J}_{\mu,g}$ over $\overline{\M}_{g,1}$ containing $\J_{d,g,1}$ as open dense subset. For every scheme $S$, we have 
\[
\ol{\J}_{\mu,g}(S)=\frac{\left\{(\pi,\sigma,\I);\begin{array}{l} \pi\col \X\to S\text{ is a family of stable pointed genus-g curves,}\\ \I \text{ is a $(\sigma,\mu)$-quasistable torsion free rank-1 sheaf on $\X$}\end{array}\right\}}{\sim}
\] 
where $(\pi_1,\sigma_1,\I_1)\sim(\pi_2,\sigma_2,\I_2)$ if there exist a $S$-isomorphism $f\col \X_1\to \X_2$ and an invertible sheaf $\L$ on $S$, such that $\sigma_2=f\circ\sigma_1$ and $\I_1\cong f^*\I_2\otimes\pi_1^*\L$. We refer to \cite[Theorems A and B]{M15} and \cite[Corollary 4.4 and Remark 4.6]{KP} for more details on the stack $\ol{\J}_{\mu,g}$. In what follows we will consider the universal compactified Jacobian $\ol{\J}_{\mu,g}$ where $\mu$ is the canonical polarization (recall Example \ref{exa:pol}).

\subsection{The stratification of $\overline{\J}_{\mu,g}$}

In this section we study the stratification of the open embedding $\J_{d,g,1}\subset \overline{\J}_{\mu,g}$ and we collect some local properties of $\overline{\J}_{\mu,g}$.

Let $\J_{(\Gamma,\E,D)}$ be the substack of $\overline{\J}_{\mu,g}$ that parametrizes tuples $(\pi\col \X\to S,\sigma,\I)$ where for each $s\in S$ we have that $(\X_s,\I|_{\X_s})$ have dual graph isomorphic to $(\Gamma,\E,D)$.   Let  $\J_{(\Gamma,\E,D)}\to \M_{\Gamma}$ be the forgetful map, and consider 
\[
\widetilde{\M}_\Gamma:=\prod_{v\in V(\Gamma)}\M_{w(v),\val(v)+\l(v)}.
\] Let $\C_\Gamma\to \M_{\Gamma}$ be the universal family over $\M_\Gamma$ and define $\widetilde{\C}_\Gamma:=\C_\Gamma\times_{\M_\Gamma} \widetilde{\M}_\Gamma$. Consider the partial normalization $\widetilde{\C}_{\Gamma,\E}$ of $\widetilde{\C}_\Gamma$ over the nodes corresponding to edges in $\E$, and form the commutative diagram
\[
\SelectTips{cm}{11}
\begin{xy} <16pt,0pt>:
\xymatrix{ \widetilde{\C}_{\Gamma,\E}\ar[dr]_{f}\ar[r] &\widetilde{\C}_\Gamma \ar[d]\ar[r]& \C_\Gamma\ar[d]\\
                                                   & \widetilde{\M}_\Gamma\ar[r] & \M_\Gamma   
}
\end{xy}
\]	
Let $F$ be the divisor on $\Gamma_\E$  such that $F(v)=\val_\E(v)$ (recall that $\Gamma_\E$ is the graph obtained by removing the edges in $\E$ of $\Gamma$). Let $\widetilde{\J}_{(\Gamma,\E,D)}:=\J_{f, D_\E-F}$ be the relative Jacobian  parametrizing invertible sheaves on $\widetilde{\C}_{\Gamma,\E}$ with multidegree $D_\E-F$.  If $\L$ is the universal invertible sheaf over $\widetilde{\C}_{\Gamma,\E}\times_{\widetilde{\M}_\Gamma} \widetilde{\J}_{(\Gamma,\E,D)}$ then the pushforward of $\L$ to $\widetilde{\C}_\Gamma\times_{\widetilde{\M}_\Gamma}\widetilde{\J}_{(\Gamma,\E,D)}$ is a torsion free rank-$1$ sheaf with multidegree $(\E,D)$. Hence we have a map $\widetilde{\J}_{(\Gamma,\E,D)}\to \J_{(\Gamma,\E,D)}$ inducing a map 
\[
g\col \widetilde{\J}_{(\Gamma,\E,D)}\to \J_{(\Gamma,\E,D)}\times_{\M_\Gamma}\widetilde{\M}_{\Gamma}.
\]
which is an isomorphism onto its image.
We have a diagram
\[
\SelectTips{cm}{11}
\begin{xy} <16pt,0pt>:
\xymatrix{ \widetilde{\J}_{(\Gamma,\E,D)}\ar[dr]_{h}\ar[r]^{g\;\quad\quad} &\J_{(\Gamma,\E,D)}\times_{\M_{\Gamma}} \widetilde{\M}_{\Gamma} \ar[d]\ar[r]& \widetilde{\M}_\Gamma\ar[d]\\
                                                   & \J_{(\Gamma,\E,D)}\ar[r] & \M_\Gamma   
}
\end{xy}
\]	
where $h$ is the composition of the projection and $g$. Recall that $\M_{\Gamma}=[\widetilde{\M}_\Gamma/\Aut(\Gamma)]$ (see \cite[Proposition 3.4.1]{ACP}); now we prove an analogous result for $\J_{\Gamma,\E,D}$.

\begin{Prop}\label{prop:stratades}
We have an isomorphism 
\[
\J_{(\Gamma,\E,D)}\simeq\left[\frac{\widetilde{\J}_{(\Gamma,\E,D)}}{\Aut(\Gamma,\E,D)}\right]
\]
 and $h$ is the quotient map.
\end{Prop}

\begin{proof}
First, note that
\[
\J_{(\Gamma,\E,D)}=\left[\frac{\J_{(\Gamma,\E,D)}\times_{\M_{\Gamma}}\widetilde{\M}_{\Gamma}}{\Aut(\Gamma)}\right],
\]
and we have a morphism
\begin{equation}
\label{eq:Jaut}
\left[\frac{\widetilde{\J}_{(\Gamma,\E,D)}}{\Aut(\Gamma,\E,D)}\right]\to\left[\frac{\J_{(\Gamma,\E,D)}\times_{\M_{\Gamma}}\widetilde{\M}_{\Gamma}}{\Aut(\Gamma)}\right]
\end{equation}
induced by the equivariant morphism $g$. Let us construct the inverse of this morphism.\par
  Let $\sigma$ be an automorphism of $\Gamma$. Define 
\[
g_\sigma\col \widetilde{\J}_{(\Gamma,\sigma(\E),\sigma(D))}\to \J_{(\Gamma,\sigma(\E),\sigma(D))}\times_{\M_\Gamma}\widetilde{\M}_\Gamma
\]
 as we did for $g$. Note that $\J_{(\Gamma,\sigma(\E),\sigma(D))}=\J_{(\Gamma,\E,D)}$ in $\ol{\J}_{\mu, g}$ for every $\sigma\in\Aut(\Gamma)$.\par

Since $\M_\Gamma=[\widetilde{\M}_\Gamma/\Aut(\Gamma)]$, a point of $\J_{(\Gamma,\E,D)}\times_{\M_{\Gamma}}\widetilde{\M}_{\Gamma}$ parametrizes a triple $(X,I,\tau)$, where  $[X]\in \M_\Gamma$ and $I$ is a torsion-free rank-1 sheaf on $X$, and where $\tau\col\Gamma_X\to \Gamma$ is an isomorphism such that $\tau_*(\E_I,D_I)=(\sigma(\E),\sigma(D))$ for some $\sigma\in\Aut(\Gamma)$. Moreover, the triple $(X,I,\tau)$ is parametrized by a point in $g_\sigma(\widetilde{\J}_{\Gamma,\sigma(\E),\sigma(D)})$ if and only if $\tau_*(\E_I,D_I)=(\sigma(\E),\sigma(D))$. So if
\[
g_\sigma(\widetilde{\J}_{(\Gamma,\sigma(\E),\sigma(D))})\cap g_{\sigma'}(\widetilde{\J}_{(\Gamma,\sigma'(\E),\sigma'(D))})\neq\emptyset
\]
then $(\sigma(\E),\sigma(D))=(\sigma'(\E),\sigma'(D))$, and hence $\sigma\Aut(\Gamma,\E,D)=\sigma'\Aut(\Gamma,\E,D)$. Conversely, if $\sigma\Aut(\Gamma,\E,D)=\sigma'\Aut(\Gamma,\E,D)$, then $(\sigma(\E),\sigma(D))=(\sigma'(\E),\sigma'(D))$, which implies $g_\sigma(\widetilde{\J}_{(\Gamma,\sigma(\E),\sigma(D))})=g_{\sigma'}(\widetilde{\J}_{(\Gamma,\sigma'(\E),\sigma'(D))})$.\par
	Then, we have
\[
\J_{(\Gamma,\E,D)}\times_{\M_{\Gamma}}\widetilde{\M}_{\Gamma}=\coprod_{i=1}^{N}g_{\sigma_i}(\widetilde{\J}_{(\Gamma,\sigma_i(\E),\sigma_i(D))}),
\]
where $N:=[\Aut(G)\col\Aut(\Gamma,\E,D)]$ and $\sigma_i$ are chosen as representatives of the left cosets of $\Aut(\Gamma,\E,D)$ in $\Aut(\Gamma)$. Note that $\Aut(\Gamma)$ identifies the connected components of $\J_{(\Gamma,\E,D)}\times_{\M_{\Gamma}}\widetilde{\M}_{\Gamma}$, because $\sigma(g_{\sigma'}(\widetilde{\J}_{(\Gamma,\sigma'(\E),\sigma'(D))})=g_{\sigma\sigma'}(\widetilde{\J}_{(\Gamma,\sigma\sigma'(\E),\sigma\sigma'(D))})$. The chosen elements $\sigma_i$ can be used to define morphisms 
\[
\rho_i:=\sigma_i^{-1}|_{g_{\sigma_i}(\widetilde{\J}_{(\Gamma,\sigma_i(\E),\sigma_i(D))})}\col g_{\sigma_i}(\widetilde{\J}_{(\Gamma,\sigma_i(\E),\sigma_i(D))})\to g(\widetilde{\J}_{(\Gamma,\E,D)})\cong \widetilde{\J}_{(\Gamma,\E,D)}
\]
 giving rise to an $\Aut(\Gamma)$-invariant morphism
\[
\J_{(\Gamma,\E,D)}\times_{\M_{\Gamma}}\widetilde{\M}_{\Gamma}\stackrel{(\rho_i)_{1\le i\le N}}{\longrightarrow}\widetilde{\J}_{(\Gamma,\E,D)}\to\left[\frac{\widetilde{\J}_{(\Gamma,\E,D)}}{\Aut(\Gamma,\E,D)}\right],
\]
which in turn induces the inverse of the morphism in Equation \eqref{eq:Jaut}.
\end{proof}

We need some results on the local geometry of $\ol{\J}_{\mu,g}$, where $\mu$  the canonical polarization. We will deduce them from known results about Caporaso-Pandharipande's compactification $\overline{\J}_{d,g,1}$ of $\J_{d,g,1}$. Let us introduce the moduli stack  $\overline{\J}_{d,g,1}$.

Let $\J_{d,g}\ra \M_{g}$ be the universal degree-$d$ Jacobian, parametrizing invertible sheaves on smooth curves. In \cite{C}, \cite{Pand} and \cite{C08}, Caporaso and Pandharipande introduced a moduli stack $\overline{\J}_{d,g}$ compactifying $\J_{d,g}$ over $\ol{\M}_g$. This stack can be viewed either as a moduli space of certain invertible sheaves, called \emph{balanced}, on semistable curves (this is done in \cite{C} and \cite{C08}), or as a moduli space of certain torsion-free rank-1 sheaves, called \emph{semistable}, on stable curves (this is done in \cite{Pand}). The two approaches give rise to isomorphic stacks (see \cite[Theorem 6.3]{EP}). For our purposes, it is better to follow the latter approach. 

The setting can be extended to pointed curves. 
We let $\ol{\J}_{d,g,1}$ be the contravariant functor from the category of schemes to that of sets, taking a $k$-scheme $S$ to
\[
\ol{\J}_{d,g,1}(S)=\frac{\left\{(\pi,\sigma,\I);\begin{array}{l} \pi\col \X\to S\text{ is a family of stable pointed genus-g curves,}\\ \I \text{ is a $\mu$-semistable torsion free rank-1 sheaf on $\X$}\end{array}\right\}}{\sim}
\] 
where $\sim$ is the same equivalence relation defined at the end of Section \ref{sec:EstJac}. Essentially the same proof of \cite[Theorem 6.3]{EP} shows that $\ol{\J}_{d,g,1}$ is isomorphic to the stack $\overline{\P}_{d,g,1}$ defined in \cite[Definition 4.1]{M11} by means of balanced invertible sheaves. Hence it follows from \cite[Theorem 4.2]{M11} that  $\ol{\J}_{d,g,1}$ is a smooth and irreducible algebraic Artin stack of dimension $4g -2$ and it is universally closed  over $\overline{\M}_{g,1}$.

\begin{Lem}\label{lem:openimm}
There exists an open immersion of stacks $\overline{\J}_{\mu,g}\to \overline{\J}_{d,g,1}$. In particular for every point $[(X,I)]\in \overline{\J}_{\mu,g}$ we have
\[
\wh{\O}_{\overline{\J}_{\mu,g},[(X,I)]}\simeq \wh{\O}_{ \overline{\J}_{d,g,1},[(X,I)]}.
\]
\end{Lem}	

\begin{proof}
Every $(p_0,\mu)$-quasistable torsion-free rank 1 sheaf on a stable pointed curves is $\mu$-semistable, hence there exists an injective map $\overline{\J}_{\mu,g}\to \overline{\J}_{d,g,1}$. It is an open immersion because quasistability is an open condition (see \cite[Proposition 34]{Es01}).
\end{proof}

\begin{Rem}
In general, the open immersion of Lemma \ref{lem:openimm} is not an isomorphism: there are $\mu$-semistable torsion free rank-$1$ sheaves that are not $(p_0,\mu)$-quasistable.
\end{Rem}

Let  $[(X,I)]$ be a point in $\overline{\J}_{\mu,g}$ and $(\Gamma,\E,D)$ be the dual graph of $(X,I)$. It follows from  Lemma \ref{lem:openimm} and \cite[Equation 8.4]{CKV} that, up to choosing an orientation of $\Gamma$,
	\begin{equation}
	\label{eq:Jloc}
	\widehat{\mathcal{O}}_{\overline{\J}_{\mu,g}[(C,I)]}=A\otimes k[[W_1,\ldots, W_{3g-2-|E(\Gamma)|},Z_1,\ldots, Z_{g-|\E|}]],
	\end{equation}
where 
\[
A:=\bigotimes_{e\in \E}\frac{k[[X_e,Y_e,T_e]]}{\langle X_eY_e-T_e\rangle}\otimes\left(\bigotimes_{e\in E(\Gamma)\setminus \E} k[[T_e]]\right).
\]

The variables appearing above have the following modular interpretations. The variables $W$'s correspond to variations of the complex structure of the irreducible components of $X$, the variables $Z$'s correspond to deformations of $I$ where it is invertible, the variable $T_e$ corresponds to the smoothing of $X$ at the node associated to $e$, the variables $X_e$ and $Y_e$ correspond to the deformation of $I$ at the node associated to $e$. Moreover, $X_e$ and  $Y_e$ correspond to the source and target of $e$, respectively; the construction is independent of the choice of the orientation.

We note that, although the results in \cite{CKV} hold for $\overline{\J}_{d,g}$, their arguments are based on  deformation theory and can be naturally extended to the pointed case needed to analyze $\overline{\J}_{d,g,1}$.\par
 The local ring at $[X]$ of $\overline{\M}_{g,1}$ is
	\begin{equation}
	\label{eq:Mloc}
	\widehat{\O}_{\overline{\M}_{g,1},[X]}=\bigotimes_{e\in E(\Gamma)} k[[T_e]]\otimes k[[W_1,\ldots, W_{3g-2-|E(\Gamma)|}]],
	\end{equation}
	and the forgetful map $\pi\col\overline{\J}_{\mu,g}\to \overline{\M}_{g,1}$ is induced locally at $[(X,I)]$ by the ring homomorphism $\pi^\#\col \widehat{\O}_{\overline{\M}_{g,1},[X]}\ra \widehat{\mathcal{O}}_{\overline{\J}_{\mu,g},[(X,I)]}$ given by $T_e\mapsto T_e$ and $W_i\mapsto W_i$ (see \cite[Equation 7.17]{CKV}).

Recall the notion of toroidal embedding of Deligne-Mumford stacks given in \cite[Definition 6.1.1]{ACP}.

	\begin{Prop}\label{prop:stratification}
The open immersion $\J_{d,g,1}\subset \overline{\J}_{\mu,g}$ is a toroidal embedding of Deligne-Mumford stacks and the forgetful map $\overline{\J}_{\mu,g}\to \ol{\M}_{g,1}$ is a toroidal morphism. Moreover, we have a partition  
		\begin{equation}\label{eq:strata}
		\overline{\J}_{\mu,g}=\coprod_{(\Gamma,\E,D)\in \mathcal{QD}_{\mu,g}}\J_{(\Gamma,\E,D)},
		\end{equation}
where each $\J_{(\Gamma,\E,D)}$ is irreducible, and the following properties are equivalent for every $(\Gamma,\E,D),(\Gamma',\E',D')\in \mathcal{QD}_{\mu,g}$:
\begin{enumerate}
 \item
 \label{thm-strata1}
$\J_{(\Gamma,\E,D)}\cap \overline{\J}_{(\Gamma',\E',D')}\neq \emptyset$;
\item
 \label{thm-strata2}
  $(\Gamma,\E,D)\geq (\Gamma',\E',D')$;
\item
 \label{thm-strata3}
$\J_{(\Gamma,\E,D)}\subset \ol{\J}_{(\Gamma',\E',D')}$.
\end{enumerate}
	\end{Prop}

\begin{proof}
The fact that $\J_{d,g,1}\subset \overline{\J}_{\mu,g}$ and the map $\overline{\J}_{\mu,g}\to \M_{g,1}$ are toroidal follows from Equations \eqref{eq:Jloc} and \eqref{eq:Mloc}, and from the local description of the map $\pi$.\par
  The partition \eqref{eq:strata} follows from the fact that each pair $(X,I)$ has a unique dual graph, up to isomorphism.	The locus $\J_{(\Gamma,\E,D)}$ is irreducible because the map $h\col \widetilde{\J}_{(\Gamma,\E,D)}\ra \J_{(\Gamma,\E,D)}$ is surjective by Proposition \ref{prop:stratades}, and $\widetilde{\J}_{(\Gamma,\E,D)}$ is irreducible.\par
	To prove the equivalences we will use \cite[Proposition 3.4.1]{CC} and \cite[Proposition 3.4.2]{CC}. These propositions concern invertible sheaves. However, we have a canonical construction for passing from the setting of torsion-free rank-$1$ sheaves to the one of invertible sheaves. More precisely, if $\I$ is a torsion free rank-$1$ sheaf on a family of nodal curves $\X\to B$, we can construct a family of nodal curves $\widetilde{\X}\to B$ as $\widetilde{\X}:=\mathbb{P}(\I^\vee)=\text{Proj}(\text{Sym}(\I^\vee))$ which will add a $\mathbb{P}^1$ over each node where $\I$ is not locally free. Moreover, we have that $\I$ is the pushforward of $\O_{\widetilde{\X}}(-1)$ via the structural morphism $\widetilde{\X}\to \X$ (see \cite[Proposition 5.5]{EP}). \par
		Let us prove that \eqref{thm-strata1} implies \eqref{thm-strata2}. If $[(X,I)]$ is a point in $\J_{(\Gamma,\E,D)}\cap \ol{\J}_{(\Gamma',\E',D')}$, then there exists a family of stable pointed curves $\X\to B$ with generic fiber in $\J_{(\Gamma',\E',D')}$ and special fiber $X$, and a $(\sigma,\mu)$-quasistable torsion-free rank-$1$ sheaf $\I$ on $\X$ that restricts to $I$ on $X$. If we take $\widetilde{\X}=\mathbb{P}(\I^\vee)$, then the dual graph of the generic fiber of $\widetilde{\X}$ will be $\Gamma'^{\E'}$, while the dual graph of the special fiber will be $\Gamma^\E$. As in \cite[Proposition 3.4.1]{CC}, we get a specialization $\iota\col\Gamma^\E\to\Gamma'^{\E'}$ with $\iota_*(D)=D'$, hence $(\Gamma,\E,D)\geq (\Gamma',\E',D')$.\par		
		Let us prove that \eqref{thm-strata2} implies \eqref{thm-strata3}. Assume that there is a specialization $\iota\col\Gamma^\E\to\Gamma'^{\E'}$ such that $\iota_*(D)=D'$; in particular $\E'$ is a subset of $\E$. Fix a point $[(X,I)]$ in $\J_{(\Gamma,\E,D)}$ and take $\widetilde{X}:=\mathbb{P}(I^\vee)$. Note that the dual graph of $(\widetilde{X},\O_{\wt{X}}(-1))$ is $(\Gamma^\E,\emptyset,D)$. Applying \cite[Proposition 3.4.2]{CC}, we find a family of nodal curves $\widetilde{\X}\to B$ and a line bundle $\L$ on $\widetilde{\X}$, such that the generic fiber of $(\widetilde{\X},\L)$ has dual graph $(\Gamma'^{\E'},\emptyset, D')$ and has special fiber $(\widetilde{X},\O_{\wt X}(-1))$. Let $f\col\widetilde{\X}\to \X$ be the contraction of all rational curves in fibers corresponding to edges in $\E$ and set $\I:=f_*\L$. So $\X\to B$ is a family of stable pointed curves and $\I$ is a torsion-free rank-$1$ sheaf such 
 the generic fiber of $(\X,\I)$ has dual graph $(\Gamma',\E',D')$ and special fiber $(X,I)$ (see \cite[Propositions 5.4 and 5.5]{EP}). This shows that $\J_{(\Gamma,\E,D)}\subset\ol{\J}_{(\Gamma',\E',D')}$.\par
		Finally, it is clear that \eqref{thm-strata3} implies \eqref{thm-strata1}.
\end{proof}

\begin{Rem}
Note that Proposition \ref{prop:stratification} tells us that the decomposition   \eqref{eq:strata} is a stratification of $\overline{\J}_{\mu,g}$ by $\mathcal{QD}_{\mu,g}$, in the sense of \cite[Definition 1.3.2]{CC}.
\end{Rem}

\subsection{The tropicalization of $\ol{\J}_{\mu,g}$}
The goal of this section is to show that the skeleton of the Esteves' universal Jacobian $\ol{\J}_{\mu,g}$ is precisely $\overline{J}^{\trop}_{\mu,g}$.

Let $\mathcal Y$ be a separated connected Deligne-Mumford stack over an algebraically closed field $k$. There is a Berkovich analytification $\Y^{an}$ of $\Y$ which is an analytic stack. We will work with the topological space underlying $\Y^{an}$, whose points are morphisms $\Spec(K)\to \Y$, where $K$ is a non Archimedean valued field extension of the trivially valued field $k$, up to equivalence by further valued field extensions  (see \cite{Be} and \cite[Section 3]{U17}). We abuse notation and use $\Y^{an}$ for both the stack and the topological space underlying it. Note that one can choose a representative $\Spec(K)\to \Y$ with $K$ complete for each point in $\Y^{an}$.  There is a distinguished subspace $\Y^{\beth}\subset \Y^{an}$ consisting of points $\Spec(K)\to \Y$ that extends to $\Spec(R)\to \Y$, where $R$ is the valuation ring of $K$. If $\Y$ is proper then $\Y^{\beth}=\Y^{an}$. \par
  
 We recall the notion of monodromy associated to a toroidal embedding. Let $U\subset \Y$ be a toroidal embedding of Deligne-Mumford stacks. Define the sheaves $\mon_{\Y}$ and $\eff_{\Y}$ as the \'etale sheaves over $\Y$ such that, for every \'etale morphism $V\to \Y$ from a scheme $V$, we have that $\mon_\Y(V)$ (respectively, $\eff_\Y(V)$) is the group of Cartier divisors on $V$ (respectively, the submonoid of effective Cartier divisors on $V$) supported on $V\setminus U_V$, where $U_V=U\times_\Y V$. \par
	For each stratum $W\subset \Y$ and a point $w\in W$, we have an action of the \'etale fundamental group $\pi_1^{et}(W,w)$ on the stalk $\mon_{\Y,w}$ preserving $\eff_{\Y,w}$. The \emph{monodromy group} $H_W$ is defined as the image of $\pi_1^{et}(W,w)$ in $\Aut(\mon_{\Y,w})$.\par
   For each stratum $W\subset \Y$, there is an associated extended cone
\[
\ol{\sigma}_W:=\Hom_{\text{monoids}}(\eff_{\Y,w},\ol{\R}_{\geq0}).
\]
One defines an extended generalized cone complex, called the \emph{skeleton} of $\Y$, as
\[
\ol{\Sigma}(\mathcal Y):=\lim_{\lra}\ol{\sigma}_W,
\]
where the arrows $\ol{\sigma}_W\to\ol{\sigma}_{W'}$ are given by the inclusions $W'\subset \ol{W}$, where $\ol W$ is the closure of $W$ in $\Y$, and are also given, when $W=W'$, by the monodromy group $H_W$. For more details, see \cite{T} and \cite[Section 6]{ACP}.

There is a retraction map ${\bf p}_{\mathcal Y}\col \Y^{\beth}\ra \ol{\Sigma}(\mathcal Y)$ defined as follows. Let $\psi\col \Spec(R)\to \Y$ be a point in $\Y^{\beth}$, with $R$ complete. Let $w\in \Y$ be the image of the closed point in $\Spec(R)$ and $W$ be the stratum of $\Y$ containing $w$. We have a chain of maps
\[
\eff_{\Y,w}\stackrel{\epsilon}{\lra} \widehat{\O}_{\Y,w}\stackrel{\psi^\#}{\lra} R\stackrel{\nu_R}{\lra} \ol{\R}_{>0},
\]
where $\epsilon$ is the map that takes an effective divisor to its local equation and $\nu_R$ is the valuation of $R$. Note that the composition is a morphism of monoids: one  defines ${\bf p}_{\Y}(\psi)\in \ol{\Sigma}(\Y)$ as the equivalence class of $\nu_R\circ\psi^\#\circ\epsilon\in\ol{\sigma}_{W}$. We refer to \cite[Section 6]{ACP} for details, in particular to \cite[Propositions 6.1.4. and 6.2.6]{ACP}.

The inclusions $\mathcal M_{g,n}\subset \ol{\mathcal M}_{g,n}$ and $\mathcal \J_{d,g,1}\subset \ol{\mathcal J}_{\mu,g}$ are embeddings of Deligne-Mumford stacks  (see \cite[Section 3.3]{ACP} and Proposition \ref{prop:stratification}). We let $\ol{\Sigma}(\ol{\mathcal {M}}_{g,n})$ and  $\ol{\Sigma}(\ol{\mathcal {J}}_{\mu,g})$  be the skeleta of $\ol{\mathcal  M}_{g,n}$ and $\ol{\mathcal  J}_{\mu,g}$, respectively.

The forgetful map $\pi\colon \ol{\mathcal  J}_{\mu,g}\ra \ol{\mathcal M}_{g,n}$ induces a natural map 
\[
\pi^{an}\col\ol{\mathcal J}_{\mu,g}^{an} \ra \ol{\mathcal M}_{g,n}^{an}.
\]

\begin{Prop}\label{prop:funct}
The map $\pi^{an}\col\ol{\mathcal J}_{\mu,g}^{an} \ra \ol{\mathcal M}_{g,n}^{an}$ restricts to a map of generalized extended cone complexes $\ol{\Sigma}(\pi)\col  \ol{\Sigma}(\ol{\J}_{\mu,g}) \ra  \ol{\Sigma}(\ol{\mathcal M}_{g,n})$. We have 
\[
{\bf p} _{\ol{\M}_{g,1}}\circ\pi^{an}=\ol{\Sigma}(\pi)\circ{\bf p} _{\ol{\J}_{\mu,g}}.
\] 
\end{Prop}

\begin{proof}
Since  $\ol{\J}_{\mu,g}$ and $\ol{\M}_{g,n}$ is proper, then $\ol{\J}_{\mu,g}^{an}=\ol{\J}_{\mu,g}^\beth$ and $\ol{\M}_{g,n}^{an}=\ol{\M}_{g,b}^{\beth}$. The result follows just combining \cite[Proposition 6.1.8]{ACP} and Proposition \ref{prop:stratification}.
\end{proof}

\begin{Def}
\label{def:tropmap}
The \emph{tropicalization map} 
\[
\trop_{\ol{\J}_{\mu,g}}\col \overline{\J}_{\mu,g}^{an}\to \ol{J}^{\trop}_{\mu,g}
\] 
is defined as follows. Fix a point $\psi\col \Spec(R)\to \overline{\J}_{\mu,g}$ with $R$ complete and let $[(X,I)]$ be the image of the closed point. We get a ring homomorphism
\[
\psi^\#\col \widehat{\O}_{\overline{\J}_{\mu,g},[(X,I)]}\to R.
\]
Let $(\Gamma,\E,D)$ be the dual graph of $[(X,I)]$. Recall the notation in Equation \eqref{eq:Jloc}. We define a length function $\l\col E(\Gamma^\E)\to \ol{\R}_{>0}$ given by $\l(e)=\nu_R(\psi^\#(T_e))$ for $e\in E(\Gamma)\setminus \E$, while $\l(e')=\nu_R(\psi^\#(X_e))$ and $\l(e'')=\nu_R(\psi^\#(Y_e))$ for $e\in \E$, where $e'$ and $e''$ are the edges of $\Gamma^\E$ over $e\in E(\Gamma)$ corresponding to the source and target of $e$, respectively. Then, we define $\trop(\psi)=[(X,\D)]$, where $X$ is the tropical curve $(\Gamma^\E,\l)$ and $\D$ is the divisor on $X$ induced by $D$. 
\end{Def}

  The previous definition is similar to the tropicalization map $\trop_{\ol{\M}_{g,1}}\col\ol{\M}_{g,1}^{an}\to\ol{\M}_{g,1}^\trop$  in \cite[Lemma-Definition 2.4.1]{Viv13} and \cite[Section 1.1]{ACP}. As for $\trop_{\ol{M}_{g,1}}$, we have that $\trop_{\ol{\J}_{\mu,g}}$ does not depends on the choices of representative $\psi$ and local coordinates. This will be also consequence of Theorem \ref{thm:tropj}.

In the next proposition we compute the monodromy group $H_{\J_{(\Gamma,\E,D)}}$ of the stratum $\J_{(\Gamma,\E,D)}$ of the toroidal embedding $\J_{d,g,1}\subset \overline{\J}_{\mu,g}$.

\begin{Lem}
\label{lem:mon}
We have $H_{\J_{(\Gamma,\E,D)}}=\Aut(\Gamma,\E,D)$.
\end{Lem}
\begin{proof}
Let $\Gamma_0$ be the graph with a single vertex of weight $g-1$ and a single loop. For  $i=0,\dots,g-1$, let $\Gamma_i$ be the graph with a single edge connecting two vertices of weight $i$ and $g-i$, with the leg on the vertex with weight $i$. For $i=0,\dots,g-1$, let $D_i$ be the unique $(v_0,\mu)$-quasistable divisor on $\Gamma_i$ (the uniqueness is due to Proposition \ref{prop:tree}). By \cite[Theorem 3.2 and Corollary 3.3]{MV14} and the fact that $\ol{\J}_{\mu,g}\to \ol{\J}_{d,g}$ is an open immersion (see Lemma \ref{lem:openimm}), we have that the boundary of $\ol{\J}_{\mu,g}$ is given by the divisors $\Delta_{i,\mu}:=\overline{\J}_{\Gamma_i,\emptyset, D_i}$ for $i=0,\ldots, g-1$. (The difference with the description of the boundary of $\ol{\J}_{d,g}$ is that, using the notations of \cite[Corollary 3.3]{MV14}, the quasistability condition in  $\ol{\J}_{\mu,g}$ selects just one between $\ol{\delta}_{i}^1$ and $\ol{\delta}_{i}^2$.) \par
   We now follow essentially the same argument of \cite[Proposition 7.2.1]{ACP}. Let $w$ be a point in $\J_{(\Gamma,\E,D)}$. The set $E(\Gamma^\E)$ forms a group basis for $\mon_{\ol{\J}_{\mu,g},w}$ and a monoid basis for $\eff_{\ol{\J}_{\mu,g},w}$. Moreover, the locally constant sheaf of sets on $\J_{(\Gamma,\E,D)}$ whose stalk at every point is the set of edges $E(\Gamma^\E)$, is trivial when pulled back to $\widetilde{\J}_{(\Gamma,\E,D)}$. Hence the pull back of $\mon_{\ol{\J}_{\mu,g}}$ and $\eff_{\ol{\J}_{\mu,g}}$ to $\widetilde{\J}_{(\Gamma,\E,D)}$ are trivial. It follows from Proposition \ref{prop:stratades} that $\J_{(\Gamma,\E,D)}=[\widetilde{\J}_{(\Gamma,\E,D)}/\Aut(\Gamma,\E,D)]$, hence  for $w\in \J_{(\Gamma,\E,D)}$, the action of $\pi_1^{et}(\J_{(\Gamma,\E,D)},w)$ on $\text{Mon}_{\J_{\mu,g},w}$ factors through its quotient $\Aut(\Gamma,\E,D)$.
\end{proof}
	
We are ready to prove the result on the tropicalization of the Esteves' universal compactified Jacobian. 
	
\begin{Thm}
\label{thm:tropj}
There is an isomorphism of extended generalized cone complexes
\[
\Phi_{\ol{\J}_{\mu,g}}\col \ol{\Sigma}(\overline{\J}_{\mu,g})\to \ol{J}^{\trop}_{\mu,g}.
\]
Moreover, the following diagram is commutative
\begin{eqnarray*}
\SelectTips{cm}{11}
\begin{xy} <16pt,0pt>:
\xymatrix{
\ol{\J}_{\mu,g}^{an} \ar@/^2pc/[rr]^{\trop_{\ol{\J}_{\mu,g}}} \ar[d]_{\pi^{an}} \ar[r]^{{\bf p}_{\ol{\J}_{\mu,g}}\;}   
  & \ar[r]^{{\Phi}_{\ol{\J}_{\mu,g}}} \ar[d]_{\ol{\Sigma}(\pi)} \ol{\Sigma}(\ol{\J}_{g,n}) & \ol{J}_{\mu,g}^{trop}   \ar[d]_{\pi^{trop}} \\
\ol{\M}_{g,1}^{an}  \ar@/_2pc/[rr]^{\trop_{\ol{\M}_{g,1}}} \ar[r]^{{\bf p}_{\ol{ {\M}}_{g,1}}\;}            &   \ar[r]^{{\Phi}_{\ol{\M}_{g,1}}}         \ol{\Sigma}(\ol{{\M}}_{g,1})  & \ol{M}_{g,1}^{trop}  
 }
\end{xy}
\end{eqnarray*}
\end{Thm}
\begin{proof}
  Let $w$ be a point in a stratum $W=\J_{(\Gamma,\E,D)}$ of $\ol{\J}_{\mu,g}$. Then there is an isomorphism of monoids $\eff_{\ol{\J}_{\mu,g},w}\to\mathbb{Z}_{\geq0}^{E(\Gamma^\E)}$ via Equation \eqref{eq:Jloc}, hence $\ol{\sigma}_W$ is isomorphic to $\ol{\sigma}_{(\Gamma,\E,D)}$. By \cite[Proposition 6.2.6]{ACP} and Lemma \ref{lem:mon}, we have 
\[
\ol{\Sigma}(\overline{\J}_{\mu,g})=\coprod_{(\Gamma,\E,D)} \ol{\sigma}^\circ_W/H_W\cong\coprod_{(\Gamma,\E,D)}\ol{\sigma}^\circ_{(\Gamma,\E,D)}/\Aut(\Gamma,\E,D).
\]
Moreover, by Proposition \ref{prop:stratification}, we have that $(\Gamma,\E,D)\geq (\Gamma',\E',D')$ in $\mathcal{QD}_{\mu,g}$ if and only if  $\J_{\Gamma,\E,D}\subset \overline{\J}_{\Gamma',\E',D'}$ hence the extended generalized cone complex  $\ol{\Sigma}(\J_{\mu,g})$ is isomorphic to $\ol{J}^\trop_{\mu,g}$.\par
    Let $\psi\col \Spec(R)\to \ol{\J}_{\mu,g}$ be a point in $\ol{\J}_{\mu,g}^{an}$, with $R$ complete, such that the image $w$ of the closed point in $\Spec(R)$ lies in the stratum $\J_{(\Gamma,\E,D)}$ of $\overline{\J}_{\mu,g}$. The set $E(\Gamma^\E)$ can be seen as a monoid basis of the free monoid $\eff_{\ol{\J}_{\mu,g},w}$. If $\trop(\psi)=[(X,\D)]\in\ol{J}_{\mu,g}$ and $\l\col E(\Gamma^\E)\to \ol{\R}_{>0}$ is the length function of $X$, then $\l$ factors through the composition
\[
\l\col E(\Gamma^\E)\to \eff_{\ol{\J}_{\mu,g},w}\stackrel{{\bf p}_{\ol{\J}_{\mu,g}}(\psi)}{\lra} \ol{\R}_{\geq0}.
\]
It follows that $\trop_{\ol{\J}_{\mu,g}}=\Phi_{\ol{\J}_{\mu,g}}\circ{\bf p}_{\ol{\J}_{\mu,g}}$.

The fact that the square in the left hand side of the diagram in the statement is commutative follows from Proposition \ref{prop:funct}. 

Finally, thanks to Equations \eqref{eq:Jloc} and \eqref{eq:Mloc}, we have that $\trop_{\ol{\M}_{g,1}}\circ\pi^{an}(\psi)$ is the tropical curve $(\Gamma,\l')$ where $\l'(e)=\nu_R(\psi^\#T_{e})$. On the other hand, $\pi^\trop\circ \trop_{\ol{\J}_{\mu,g}}(\psi)$ is the tropical curve $(\Gamma^\E,\l)$ where $\l$ is as in Definition \ref{def:tropmap}. Since $\nu_R(\psi^\#T_e)=\nu_R(\psi^\#X_e)+\nu_R(\psi^\#Y_e)$ by Equation \eqref{eq:Jloc}, we deduce that $\l'(e)=\l(e')+\l(e'')$ if $e\in \E$ and $e',e''$ are the edges of $\Gamma^\E$ over $e$, otherwise $\l'(e)=\l(e)$ if $e\notin \E$. Hence the tropical curves $(\Gamma,\l')$ and $(\Gamma^\E,\l)$ are isomorphic, and so
\[
\trop_{\ol{\M}_{g,1}}\circ\pi^{an}=\pi^\trop\circ \trop_{\ol{\J}_{\mu,g}}.
\]
This concludes the proof.
\end{proof}

\section*{Acknowledgments}	

We thank Lucia Caporaso, Renzo Cavalieri, Eduardo Esteves, Margarida Melo, Dhruv Ranganathan, and Martin Ulirsch for useful conversations and comments. The second author was supported by CNPq, processo 301314/2016-0.  We thank the referees for carefully reading the paper and for the constructive suggestions.

\bigskip
\noindent{\smallsc Alex Abreu and Marco Pacini, Universidade Federal Fluminense,
\\ 
Rua Prof. M. W. de Freitas, S/N, Istituto de Matem\'atica. Niter\'oi, Rio de Janeiro, Brazil. 24210-201. }\\
{\smallsl E-mail addresses: \small\verb?alexbra1@gmail.com?  \;\;  and \;\;   \small\verb?pacini.uff@gmail.com?}

\end{document}